\documentclass[12pt]{amsart}%
\usepackage{amsfonts}
\usepackage{amsmath}
\usepackage{amssymb}
\usepackage{graphicx}
\usepackage{hyperref}
\usepackage{color}
\makeatletter
\@addtoreset{equation}{section}
\makeatother

\newtheorem{theorem}{Theorem}[section]

\newtheorem{definition}[theorem]{Definition}

\newtheorem{lemma}[theorem]{Lemma}

\newtheorem{proposition}[theorem]{Proposition}
\newtheorem{remark}[theorem]{Remark}

\begin{document}

\title[Existence and Nonexistence of Extremals for Trudinger-Moser inequalities ] {Existence and Nonexistence of Extremals for Trudinger-Moser inequalities with $L^p$ type perturbation on any bounded planar domains
}
\author{Lu Chen, Rou Jiang, Guozhen Lu and Maochun Zhu}
\address{Key laboratory of Algebraic Lie Theory and Analysis of Ministry of Education,
School of Mathematics and Statistics, Beijing Institute of Technology, Beijing 100081, P. R. China; MSU-BIT-SMBU Joint Research Center of Applied Mathematics, Shenzhen MSU-BIT
University, Shenzhen 518172, China}
\email{chenlu5818804@163.com}
\address{School of Mathematics and Statistics, Wuhan University, Wuhan, 430072,
P. R. China}
\email{jiangrou2023@126.com}
\address{Department of Mathematics\\
University of Connecticut\\
Storrs, CT 06269, USA}
\email{guozhen.lu@uconn.edu}
\address{School of Mathematics and Statistics, Nanjing University of Science and
Technology, Nanjing, 210094, P. R. China\\}
\email{zhumaochun2006@126.com}

\thanks{The first author was partly supported by the National Key Research and Development Program (No.
2022YFA1006900) and National Natural Science Foundation of China (No. 12271027). The second author was supported by Natural Science Foundation of China (12471056). The third author was supported partly by a collaboration grant and Simons Fellowship from the Simons Foundation. The fourth author was supported by Natural Science Foundation of China (12071185). }

\maketitle

\begin{abstract}
In this study, we investigate the perturbed Trudinger-Moser inequalities as follows:
\[
S_\Omega(\lambda,p)=\sup_{\substack{u\in H_{0}^{1}(\Omega)\\\Vert\nabla u\Vert
_{L^{2}\left(  \Omega\right)  }\leq1}}\int_{\Omega}\left(  e^{4\pi u^{2}%
}-\lambda|u|^{p}\right)  dx,
\]
where $1\leq p<\infty$ and $\Omega$ is a bounded domain in $\mathbb{R}^2$. Our results demonstrate that there exists a threshold $\lambda^{\ast}(p)>0$ such that $S_\Omega(\lambda,p)$ is attainable if $\lambda<\lambda^{\ast}(p)$, but unattainable if $\lambda>\lambda^{\ast}(p)$ when $p\in[1,2]$. For $p>2$, however, we show that $S_\Omega(\lambda,p)$ is always attainable for any $\lambda\in \mathbb{R}$. These results are achieved through a refined blow-up analysis, which allow us to establish a sharp Dirichlet energy expansion formula for sequences of solutions to the corresponding Euler-Lagrange equations. The asymmetric nature of our problem poses significant challenges to our analysis. To address these, we will establish an appropriate comparison principle between radial and non-radial solutions of the associated Euler-Lagrange equations. Our study establishes a complete characterization of how $L^p$-type perturbations influence the existence of extremals for critical Trudinger-Moser inequalities on any bounded planar domains, this extends the classical Brezis-Nirenberg problem framework to the two-dimensional settings.

\end{abstract}

	\section{Introduction}

Let $\Omega\subset\mathbb{R}^{N}$, $N\geq2$ be a bounded domain and
$W_{0}^{1,p}(\Omega)$ (denoted as $H_{0}^{1}(\Omega)$ if $p=2$) represent the
completion of $C_{0}^{\infty}(\Omega)$ under the norm%
\[
\Vert \nabla u\Vert_{L^{p}(\Omega)}=\left(  \int_{\Omega}|\nabla u|^{p}dx\right)
^{\frac{1}{p}}\text{.}%
\]
The classical Sobolev embedding theorem tells us that $W_{0}^{1,N}\left(
\Omega\right)  \subseteq L^{p}\left(  \Omega\right)  $ for any $p>1$, but
$W_{0}^{1,N}\left(  \Omega\right)  \not \subset L^{\infty}\left(
\Omega\right)  $. One knows from the works of Yudovi\v{c} \cite{Yud},
Poho\v{z}aev \cite{Poh} and Trudinger \cite{Tru} that $W_{0}^{1,N}\left(
\Omega\right)  $ can be imbedded into the Orlicz space ${L_{\phi_N}}(\Omega)$
with the function $\phi_N(t)={\exp}\left(  {{{\left\vert t\right\vert }%
^{\frac{N}{{N-1}}}}}\right)  -1.$  In 1971, Moser \cite{Mos} sharped this embedding and proved the following Trudinger-Moser inequality
\begin{equation}
\sup_{\Vert\nabla u\Vert_{L^{N}(\Omega)}\leq1}\int_{\Omega}e^{\alpha|u|^{\frac{N}{N-1}}%
}dx<\infty\text{ iff }\alpha\leq\alpha_{N}:=N\omega_{N-1}
^{\frac{1}{N-1}}, \label{1}%
\end{equation}
where $\alpha_{N}=N\omega_{N-1}^{\frac{1}{N-1}}$, $\omega_{N-1}$ denotes
$\left(  N-1\right)  $-dimensional surface measure of the unit ball.

 Following the classical Trudinger-Moser inequality \eqref{1}, a wealth of related work has appeared. Examples include Trudiner-Moser  {\normalsize }inequalities on various geometric settings: manifolds with conic singularities (\cite{WXChen}), compact Riemannian manifolds (\cite{Fontana,Li1,Li2}), CR spheres (\cite{CohnLu-CPAM}), unbounded domains (\cite{do,Adachi-Tanaka,liruf}), hyperbolic spaces (\cite{MST,LuTang-ANS}), and the Heisenberg group (\cite{CohnLu-IUMJ,LamLu-AIM,LamLuTang-NA,LLZ-CVPDE}). Other developments include higher-order Adams inequalities (\cite{A,RS,LamLu-JDE}), Hardy-Trudinger-Moser inequalities (\cite{WangYe-AIM,LuYang-CVPDE,LLWY}), Hardy-Adams inequalities (\cite{LuYangQ1,LuYangQ2}), affine Trudinger-Moser inequalities (\cite{CLYZ,DLP}),  anisotropic Trudinger-Moser and Adams inequalities (\cite{LSXZ, Wang Xia 2012, ZoZ, ZYCZ}), trace type Trudinger-Moser inequalities (\cite{Cianchi,Chenluzhu-tra,Liliu,Yang1}) and Trudinger-Moser and Adams inequalities with degenerate potentials (\cite{Chenluzhu-PLMS,Chenluzhu-CVPDE,Chenluzhu-ANS}), Trudinger-Moser and Adams inequalities for radial functions without boundary conditions \cite{doOLuPonciano1, doOLuPonciano2, doOLuPonciano3},  etc, to just name a few and we refer the reader to the references within these works for further details.

An interesting problem related to the Trudinger-Moser inequalities lies in
investigating the existence of extremal functions. Carleson and Chang
\cite{Car} first established the existence of extremals for
Trudinger-Moser inequalities on the unit ball through symmetrization
rearrangement inequality combining with the ODE technique. We note that in
order to show the existence of extremal functions, Carleson and Chang
\cite{Car} first estimated the optimal concentration level of the Trudinger-Moser functional  on $W_{0}%
^{1,N}\left(  B_{1}\right)  $: if $\left\{  u_{i}\right\}  $ is a normalized
concentrating sequence in $W_{0}^{1,N}\left(  B_{1}\right)  $, then
\begin{equation*}
\underset{i\rightarrow\infty}{\lim\sup}%
{\displaystyle\int\nolimits_{B_{1}}}
e^{\alpha_{N}|u_{i}|^{\frac{N}{N-1}}}dx\leq|B_{1}|\left(  1+\exp(\sum
_{k=1}^{N-1}\frac{1}{k})\right),%
\end{equation*}
where the right hand side of the inequality above is now called the
Carleson-Chang limit. On the other hand, they constructed a test function such that the Carleson-Chang
limit can be surpassed which demonstrate the existence of extremal functions. After the work \cite{Car}, the existence of extremals of Trudinger-Moser inequalities on any bounded domains on
$\mathbb{R}^{n}$ were also established via harmonic transplantation (see \cite{Flu}, \cite{Lin}). For further developments, including the existence of extremals for Trudinger-Moser inequalities on bounded and unbounded domains, compact Riemannian manifolds, and anisotropic settings, as well as for higher-order Adams inequalities, one can see \cite{Li1, Li2, liruf, Yang1,ZoZ, LSXZ,Chenluzhu, DelaTorre, LuYangY-AIM, Chenluzhu-tra}, and the
references therein. It should be noted that the existence results in these works fundamentally rely on  blow-up analysis.

  We also note that in \cite{Man}, Mancini and Martinazzi present a completely different approach for proving the existence of extremals for the Trudinger-Moser inequality on the disk. They establish the existence result via a sharp Dirichlet energy expansion formula for sequences of subcritical maximizers, a method based on techniques introduced in \cite{Malchiodi} and involves performing a Taylor expansion of the subcritical maximizers near the blow-up point.
  This method can also be utilized to explore the nonexistence of extremals for Trudinger-Moser type inequalities. For example,  based on the works of \cite{Malchiodi,Man,Dru2}, Mancini and Thizy \cite{Man2} reveal that the Adimurthi-Druet inequality (see \cite{Adi})
\[
 \sup _ {u \in H_0^{1}(\Omega), \, \|\nabla u\|_{L^{2}(\Omega)} \leqslant 1} \int _ \Omega \exp (4 \pi u^2(1 + \alpha \| u \| _2^2 )) dx < \infty
\]
does not admit any extremal when the perturbation parameter $\alpha$ approaches $\lambda_1$, the first eigenvalue of $\Delta$ on $\Omega\subset\mathbb{R}^2$ with zero Dirichlet boundary condition. This conclusion is drawn by establishing a sharp expansion of the Dirichlet energy for blowing up sequences of solutions of the corresponding Euler-Lagrange equation. Furthermore, using a similar approach, Thizy \cite{Thi} provided some sharp conditions for the existence and nonexistence of the following perturbed Trudinger-Moser inequality on $\Omega\subset\mathbb{R}^2$: \begin{equation*}
S_{g,4\pi}\left(  \Omega\right)  :=\underset{u\in H_{0}^{1}(\Omega),\left\vert
\left\vert \nabla u\right\vert \right\vert _{L^{2}\left(  \Omega\right)  }%
^{2}\leq1}{\sup}\int_{\Omega}\left(  1+g\left(  u\right)  \right)  \exp\left(
4\pi u^{2}\right)  dx,%
\end{equation*}
where $g$ satisfies $g\left(  t\right)  \rightarrow0$ as $t\rightarrow\infty$. Thizy's result indicates that exponential perturbations can influence the existence and non-existence of extremals of the Trudinger-Moser inequality in bounded domains.

A natural question arises: If we replace the exponential perturbation with a lower-order perturbation, such as the $L^p$ type perturbation, can this lower-order perturbation still impact the existence and nonexistence of extremals for Trudinger-Moser inequalities? It is worth noting that this possibility of effect exists. Drawing a parallel with the well-known Brezis-Nirenberg problem on bounded domains, it is established that a subcritical $ L^p $ perturbation can indeed affect the existence of extremals for the perturbed critical Sobolev inequality. In some special cases of Trudinger-Moser setting, this question has  already been partially addressed. For example, in the recent work \cite{Chenluzhu}, Chen, Lu and Zhu investigated how $L^2$-type perturbations affect the existence and non-existence of extremals for Trudinger-Moser inequalities in the entire space $\mathbb{R}^2$. Their findings reveal that the vanishing phenomenon (first discovered in \cite{Ish}) on the entire space can influence the existence of extremals.   When $\Omega$ is a bounded domain, de Fegueiredo, do \'{O} and Ruf \cite{FDR} considered the following $L^{\frac{N}{N-1}}$ perturbed maximization problem \begin{equation}\label{permu}
S_{\Omega}(\lambda,\frac{N}{N-1})=\underset{u\in W_{0}^{1,N}(\Omega),\left\vert \left\vert
\nabla u\right\vert \right\vert _{L^{N}\left(  \Omega\right)  }^{N}\leq1}%
{\sup}\int_{\Omega}\left(  e^{\alpha_N u^{\frac{N}{N-1}}}-1-\lambda|u|^{\frac{N}{N-1}}\right)  dx,\end{equation}
where $\Omega$ is the unit ball of $\mathbb{R}^N$, and they established the existence of a maximizer for any $\lambda < \alpha_N$. They also  conjectured that $S_{\Omega}(\lambda,\frac{N}{N-1})$ would not be attained when $\lambda \geq \alpha_N$. Subsequently, in \cite{Li2006}, Li gave a negative answer to this conjecture by showing that the supremum above is still attained even $\Omega$ is a general bounded domain when $\lambda$ is slightly larger than $\alpha_N$. More recently, in \cite{Has}, Hashizume revisited the perturbed maximization problem \eqref{permu} within the unit disc in $\mathbb{R}^2$ and showed that there exists a positive threshold for the existence and non-existence of extremals by establishing a sharp Dirichlet energy expansion formula for sequences of maximizers. Indeed, Hashizume's work also explored the effects of general lower-order $L^p$ perturbations. Nevertheless, the question of whether a threshold exists for the existence and non-existence of extremals for $S_{\Omega}(\lambda,p)$ in general bounded domains remains unresolved.

In this paper, we aim to further explore this question by focusing on the following two-dimensional $L^p$ perturbed maximization problem:
 \begin{equation*}
S_{\Omega}(\lambda,p)=\underset{u\in H_{0}^{1}(\Omega),\left\vert \left\vert
\nabla u\right\vert \right\vert _{L^{2}\left(  \Omega\right)  }^{2}\leq1}%
{\sup}\int_{\Omega}\left(  e^{4\pi u^{2}}-1-\lambda|u|^{p}\right)  dx,\end{equation*}
where $\Omega$ is a bounded domain in $\mathbb{R}^2$, $1\leq p<\infty$, and $\lambda$ is a real number. We will clarify the effect of the lower order $L^{p}$ perturbed term and establish the existence of a threshold for the existence and non-existence of extremals. Our main results can be stated as follows:

\begin{theorem}
\label{th1.2}For $p>2$,  $S_{\Omega}\left(
\lambda,p\right)  $ is always attained for any $\lambda>0$.
\end{theorem}
\begin{theorem}
\label{th1.1} Let $p\in\left[  {1,2}\right]  $. Then there exists a positive
constant ${\lambda}^{\ast}={\lambda}^{\ast}\left(  p\right)  $ such that $S_{\Omega}\left(  \lambda,p\right)  \ $
is attained for $\lambda<{\lambda}^{\ast}$, while for $\lambda>{\lambda}^{\ast}$, $S_{\Omega}\left(
\lambda,p\right) =S_\Omega^\delta$, which is the optimal concentration level of the Trudinger-Moser functional on $\Omega$ (see (\ref{addcon})),  and  $S_{\Omega}\left(
\lambda,p\right)$  is not attained.
\end{theorem}

Our results provide a comprehensive depiction of the impact of \(L^p\)-type perturbations on the existence of extremals for critical Trudinger-Moser inequalities on any bounded planar domains, this can be understood as an extension of the classical Brezis-Nirenberg problem to the two-dimensional case.

Next, we outline the structure of the proof and describe the strategies employed to tackle the technical challenges associated with it. For $p>2$, we prove the existence of extremals by constructing test functions to show that the optimal concentration level of Trudinger-Moser inequalities can be exceeded. For $1\leq p\leq2$, we establish the existence of a threshold for the existence and non-existence of extremals by using refined blow up analysis and a sharp Dirichlet energy expansion formula as in \cite{Man2,Thi}. This case is more challenging than unit disc setting due to  the loss of radial symmetry. Assume no such threshold exists,  we first show that for any sequence $\lambda_k\rightarrow\infty$,  the corresponding maximizer sequence $u_k$ must  blow up (i.e. $c_k=\max_{x\in\Omega}{u_k(x)}\rightarrow\infty$). Furthermore, we show that the perturbation parameter $\lambda_k$ cannot grow too rapidly relative to  $c_k$: \begin{equation}\label{rela}\frac{\lambda_k}{c^p_k}=o_k(1),\end{equation} as $k\rightarrow\infty$. Rewriting the Dirichlet energy $\|\nabla u_k\|_{L^2}$ as the sum of Lebesgue and exponential integrals, we try to derive  a contradiction with $\|\nabla u_k\|_{L^2}=1$  by estimating both integrals separately as $k\rightarrow\infty$.  The estimate of the Lebesgue integral can be easily obtained via the convergence of $\{c_ku_k\}$.  For the integral of the exponential term, we split the domain $\Omega$ into a small ball $B_{r_{k,\delta}}(x_k)$ (see (\ref{55})) and its exterior $\Omega\setminus B_{r_{k,\delta}}(x_k)$. The exterior integral can be easily derived by the truncated technique. While estimating the small ball integral, we face two core difficulties. The first is the lack of radial symmetry. This difficulty is addressed by establishing a key tool--the comparison principle (Lemma \ref{507}).  This principle indicates that within $B_{r_{k,\delta}}(x_k) $, $\{u_k\}$ can be approximated by some radially symmetric solutions of the same equation as $ u_k $. This allows us to transfer the favorable asymptotic behavior of these radial functions (which can be obtained via techniques in \cite{Malchiodi}) to $\{u_k\}$ in $B_{r_{k,\delta}}(x_k)$. Second, the Lebesgue perturbation destroys both the non-negativity and monotonicity of the nonlinear terms in the Euler-Lagrange equations. Consequently, key tools--including the maximum principle and elliptic regularity theory for point-wise estimates of $\Delta u_{k}$ and $ u_{k}$ (e.g., Lemmas \ref{90}, \ref{85}) and level estimates of $u_{k}$ in $B_{r_{k,\delta}}(x_k)$ (Lemma \ref{90}))-cannot be applied as in \cite{Man2,Dru} directly. These estimates are crucial for the comparison principle and for controlling the Lebesgue perturbation term. To address this challenge, we introduce a smaller (perhaps) ball \(B_{R_k}(x_k)\) (see (\ref{213})) where the mean value of \(u_k\) on spheres is close to that of the radially symmetric solutions.    Through intricate analysis, we demonstrate that the error function between \(u_k\) and the radially symmetric solutions can be approximated by certain harmonic functions within this ball. By exploiting this property, we manage to achieve \(R_k = r_{k,\delta}\), and obtain the required point-wise and level set estimates.

This paper is organized as follows. Section 2 presents fundamental definitions, preliminary results, and establishes the existence of extremal functions for $p>2$  via a test function argument. In Section 3, we analyze the existence/non-existence of extremals for the perturbed Trudinger-Moser inequalities when  $1\leq p\leq2$. This includes:(i) Limiting behavior of maximizers near and far from blow-up points; (ii) Perturbation parameter estimates and the estimates of Lebesgue integral; (iii) Point-wise estimates of $\Delta u_k$ and $u_k$, level estimates of $u_k$ in $B_{r_{k,\delta}}(x_k)$, and higher-order blow-up analysis for certain radial functions. Finally, we establish the key tool-comparison principle and exponential integral estimates, derive a sharp Dirichlet energy expansion formula, and complete the proof of Theorem \ref{th1.1}. For the convenience of the reader, the analysis for the limiting behavior of maximizers far from blow-up points are arranged in the Appendix.

 Throughout the paper, $C$ denotes a nonnegative general constant
which may vary from line to line.

\section{Preliminary results and the attainability of perturbed
Trudinger-Moser inequalities for $p>2$}

In this section, we first introduce some basic definitions, lemmas and propositions, which will be useful to discuss the existence and non-existence of
extremals of perturbed Trudinger-Moser inequalities on any bounded domain. Then we will =show that extremal functions always exist by a test function argument in the case $p>2$.

 We recall that $\left\{  u_{k}\right\}  _{k}\subset H_{0}^{1}(\Omega)$ is called a
normalized concentrating sequence ((NCS) in short) at $y_{0}$ if it satisfies
\[
\int_{\Omega}|\nabla u_{k}|^{2}dx=1,\quad\lim_{k\rightarrow\infty}\int
_{\Omega\backslash B_{\varepsilon}(y_{0})}|\nabla u_{k}|^{2}dx=0\quad\text{for
any }\varepsilon>0.
\]
Note that if $\left\{  u_{k}\right\}  _{k}$ is a (NCS), then $u_{k}%
\rightharpoonup0$ weakly in $H_{0}^{1}(\Omega)$, and $u_{k}\rightarrow0$
strongly in $L^{p}(\Omega)$ for any $p\in\left[  1,+\infty\right)  $ by the
Sobolev compact imbedding.

\begin{proposition}
[\cite{Flu}]\label{lem2.2} If $\left\{  u_{k}\right\}  _{k}$ is a (NCS) at
$\ x\in\overline{\Omega}$, then
\[
\underset{k\rightarrow\infty}{\lim\sup}\int_{\Omega}{{e^{{4\pi}u_{k}^{2}}}%
}dx\leq\left\vert \Omega\right\vert +\pi er_{\Omega}^{2}\left(  x\right)  ,
\]
where ${r_{\Omega}}\left(  x\right)  $ is the harmonic radius at $x$ defined
as
\[
{r_{\Omega}}\left(  x\right)  :=\exp\left(  {-2\pi}\tau_{\Omega}\left(
x\right)  \right)  \text{,}%
\]
and $\tau_{\Omega}\left(  x\right)  $ is the Robin function of Laplacian
operator at $x$ with the Dirichlet boundary condition defined as follows.
\end{proposition}

\begin{definition}
[Robin function]\label{rem1}The Robin function $\tau_{\Omega}\left(  x\right)
$ is the trace of the regular part $H_{\Omega,x}\left(  y\right)  $ of the
Green function $G_{\Omega,x}\left(  y\right)  $ of Laplacian operator with
singularity at $x$, which can be expressed in the following form
\[
{G_{\Omega,x}}\left(  y\right)  =-\frac{1}{{2\pi}}\log\left\vert
{x-y}\right\vert -{H_{\Omega,x}}\left(  y\right)  \text{.}%
\]
 The minimum point of the Robin function or maximum point of the harmonic
radius is called a harmonic center of $\Omega$.
\end{definition}

\begin{remark}\label{re2.2}
\ If $x\in\partial\Omega$, then $\tau_{\Omega}\left(  x\right)  =0$, hence
\[
\underset{k\rightarrow\infty}{\lim\sup}\int_{\Omega}{{e^{{4\pi}u_{k}^{2}}}%
}dx\leq\left\vert \Omega\right\vert .
\]

\end{remark}

Let $S_{\Omega}^{\delta}$ be the optimal upper bound of the Trudinger-Moser
functional in (\ref{1}) over all the (NCS), from Proposition \ref{lem2.2}, we
have
\begin{equation}\label{addcon}
S_{\Omega}^{\delta}=\left\vert \Omega\right\vert +\pi e\sup_{x\in\Omega
}r_{\Omega}^{2}\left(  x\right)  .
\end{equation}
The next lemma follows from the definition of $S_{\Omega}\left(
\lambda,p\right)  $ and Proposition \ref{lem2.2}.

\begin{lemma}
\label{224} The supremum $S_{\Omega}\left(  \lambda,p\right)  $ is
non-increasing with respect to $\lambda$ and
\begin{equation}
S_{\Omega}\left(  \lambda,p\right)  \geq S_{\Omega}^{\delta} \label{lag}%
\end{equation}
for any $\lambda$ and $p$. Moreover, we have%
\begin{equation}\label{qq}
\underset{\lambda\rightarrow0^{+}}{\lim}S_{\Omega}\left(  \lambda,p\right)
=S_{\Omega}\left(  0,p\right)  \text{.} %
\end{equation}

\end{lemma}

\begin{proof}
Choosing $\lambda_{1}$, $\lambda_{2}>0$ such that $\lambda_{1}>\lambda_{2}$,
and assuming that $\{u_{k}\}_{k}\subset H_{0}^{1}\left(  \Omega\right)  $ is
the maximizing sequence of $S_{\Omega}\left(  \lambda_{1},p\right)  $, then%
\[
S_{\Omega}(\lambda_{1},p)=\underset{k\rightarrow\infty}{\lim}\int_{\Omega
}\left(  e^{4\pi u_{k}^{2}}-\lambda_{1}|u_{k}|^{p}\right)  dx\leq
\underset{k\rightarrow\infty}{\lim}\int_{\Omega}\left(  e^{4\pi u_{k}^{2}%
}-\lambda_{2}|u_{k}|^{p}\right)  dx\leq S_{\Omega}\left(  \lambda
_{2},p\right)  \text{.}%
\]
Thus, $S_{\Omega}\left(  \lambda,p\right)  $ is non-increasing with respect to
$\lambda$.

For any $\lambda,p>0$, it is clear that%
\begin{align*}
S_{\Omega}\left(  \lambda,p\right)   &  \geq\underset{\{u_{k}\}\text{:(NCS) }%
}{\sup}\underset{k\rightarrow\infty}{\lim\sup}\int_{\Omega}\left(  e^{4\pi
u_{k}^{2}}-\lambda|u_{k}|^{p}\right)  dx\\
&  =\underset{\{u_{k}\}\text{:(NCS)}}{\sup}\underset{k\rightarrow\infty}%
{\lim\sup}\int_{\Omega}e^{4\pi u_{k}^{2}}dx=S_{\Omega}^{\delta},
\end{align*}
hence (\ref{lag}) is proved.

Next, we prove (\ref{qq}). Observe that $S_{\Omega}\left(  \lambda,p\right)
\leq$ $S_{\Omega}\left(  0,p\right)  $ for $\lambda>0$, hence%
\[
\underset{\lambda\rightarrow0^{+}}{\lim}S_{\Omega}\left(  \lambda,p\right)
\leq S_{\Omega}\left(  0,p\right)  \text{.}%
\]
On the other hand, let $u_{0}\in H_{0}^{1}\left(  \Omega\right)  $ be an
extremal function of $S_{\Omega}\left(  0,p\right)  $, then%
\[
\underset{\lambda\rightarrow0^{+}}{\lim}S_{\Omega}\left(  \lambda,p\right)
\geq\underset{\lambda\rightarrow0^{+}}{\lim}\int_{\Omega}\left(  e^{4\pi
u_{0}^{2}}-\lambda\left\vert u_{0}\right\vert ^{p}\right)  dx=\int_{\Omega
}e^{4\pi u_{0}^{2}}dx=S_{\Omega}\left(  0,p\right)  \text{.}%
\]
\ Therefore, (\ref{qq}) holds true, and the proof is finished.
\end{proof}

\begin{lemma}
\medskip\label{225} If $S_{\Omega}\left(  \lambda^{\ast},p\right)  =S_{\Omega
}^{\delta}$ holds with some $\lambda^{\ast}>0$, then $S_{\Omega}\left(
\lambda,p\right)  $ is not attained for any $\lambda>\lambda^{\ast}$. If
$S_{\Omega}\left(  \lambda,p\right)  >S_{\Omega}^{\delta}$, then $S_{\Omega
}\left(  \lambda,p\right)  $ is attained by some $u_{\lambda}\in H_{0}%
^{1}\left(  \Omega\right)  $.
\end{lemma}

\begin{proof}
We prove the first conclusion by\ contradiction. If $S_{\Omega}\left(
\lambda,p\right)  $ is attained by some $u_{\lambda}\in$\ $H_{0}^{1}\left(
\Omega\right)  $ for $\lambda>\lambda^{\ast}$, then%
\begin{align*}
S_{\Omega}\left(  \lambda,p\right)   &  =\int_{\Omega}\left(  e^{4\pi
u_{\lambda}^{2}}-\lambda\left\vert u_{\lambda}\right\vert ^{p}\right)  dx\\
&  <\int_{\Omega}\left(  e^{4\pi u_{\lambda}^{2}}-\lambda^{\ast}\left\vert
u_{\lambda}\right\vert ^{p}\right)  dx\leq S_{\Omega}\left(  \lambda^{\ast
},p\right)  =S_{\Omega}^{\delta}\text{,}%
\end{align*}
which contradicts to Lemma \ref{224}.

Next, we prove the second conclusion. Assume that $\{u_{k}\}_{k}\subset
H_{0}^{1}\left(  \Omega\right)  $ is a maximizing sequence of $S_{\Omega
}\left(  \lambda,p\right)  $, namely,%
\[
\left\vert \left\vert \nabla u_{k}\right\vert \right\vert _{L^{2}\left(
\Omega\right)  }\leq1\text{ \ and }\underset{k\rightarrow\infty}{\lim}%
\int_{\Omega}\left(  e^{4\pi u_{k}^{2}}-\lambda\left\vert u_{k}\right\vert
^{p}\right)  dx=S_{\Omega}\left(  \lambda,p\right)  \text{.}%
\]
Then up to a subsequence, there exists some $u_{\lambda}\in H_{0}^{1}\left(
\Omega\right)  $ such that $u_{k}\rightharpoonup u_{\lambda}$ in $H_{0}%
^{1}\left(  \Omega\right)  $ with $\left\vert \left\vert \nabla u_{\lambda
}\right\vert \right\vert _{L^{2}\left(  \Omega\right)  }\leq1$. From the
assumption $S_{\Omega}\left(  \lambda,p\right)  >S_{\Omega}^{\delta}$, we see
that $\{u_{k}\}_{k}$ is not a (NCS) by Proposition \ref{lem2.2}. Using the
Sobolev compact embedding and Concentration-compactness principle associated with
Trudinger-Moser inequalities in \cite{Lions}, we have%
\[
\underset{k\rightarrow\infty}{\lim}\int_{\Omega}\left(  e^{4\pi u_{k}^{2}%
}-\lambda\left\vert u_{k}\right\vert ^{p}\right)  dx=\int_{\Omega}\left(
e^{4\pi u_{\lambda}^{2}}-\lambda\left\vert u_{\lambda}\right\vert ^{p}\right)
dx\text{,}%
\]
and then $S_{\Omega}\left(  \lambda,p\right)  $ is attained. Thus, we finish
the proof.
\end{proof}

Indeed, we can show that if $S_{\Omega}\left(  \lambda,p\right)  $ is attained
by some $u_{\lambda}\in H_{0}^{1}\left(  \Omega\right)  $, then $\left\vert
\left\vert \nabla u_{\lambda}\right\vert \right\vert _{L^{2}\left(
\Omega\right)  }=1$. This is included in the following result.

\begin{lemma}
\label{lem2.1} Let $t\in(0,1)$. Then
\[
\sup_{u\in H_{0}^{1}(\Omega),\Vert\nabla u\Vert_{L^{2}\left(  \Omega\right)
}\leq t}\int_{\Omega}\left(  e^{{4\pi}u^{2}}-\lambda|u|^{p}\right)
dx<S_{\Omega}\left(  \lambda,p\right)  .
\]

\end{lemma}

\begin{proof}
The proof is similar as in \cite[Propisition 2.3]{Has}. For the convenience of
the readers, we give the detailed proof. Set
\[
{S_{\Omega,t}}(\lambda,p)=\sup_{u\in H_{0}^{1}(\Omega),\Vert\nabla
u\Vert_{L^{2}\left(  \Omega\right)  }\leq t}\int_{\Omega}\left(  e^{{4\pi
}u^{2}}-\lambda|u|^{p}\right)  dx.
\]
If the assertion fails, then $S_{\Omega,t}(\lambda,p)=S_{\Omega}\left(
\lambda,p\right)  $ for some $0<t<1$. Taking a maximizing sequence $\left\{
u_{k}\right\}  _{k}$ of $S_{\Omega,t}(\lambda,p)$, then
\[
\left\Vert \nabla u_{k}\right\Vert _{L^{2}\left(  \Omega\right)  }\leq
t,\quad\lim_{k\rightarrow\infty}\int_{\Omega}\left(  e^{{4\pi}u_{k}^{2}%
}-\lambda\left\vert u_{k}\right\vert ^{p}\right)  dx=S_{\Omega,t}%
(\lambda,p)\text{.}%
\]
Up to a subsequence, there exists $u_{\lambda}\in H_{0}^{1}(\Omega)$ such that
$u_{k}\rightharpoonup u_{\lambda}\not \equiv 0$ in $H_{0}^{1}(\Omega)$ and
$\left\Vert \nabla u_{\lambda}\right\Vert _{L^{2}}=\tilde{t}\leq t$. Using the
Concentration-compactness principle associated with Trudinger-Moser
inequalities, we obtain
\[
\int_{\Omega}\left(  e^{{4\pi}u_{\lambda}^{2}}-\lambda\left\vert u_{\lambda
}\right\vert ^{p}\right)  dx=\lim_{k\rightarrow\infty}\int_{\Omega}\left(
e^{{4\pi}u_{k}^{2}}-\lambda\left\vert u_{k}\right\vert ^{p}\right)
dx=S_{\Omega,t}(\lambda,p)\text{,}%
\]
which implies that $u_{\lambda}$ is a maximizer of $S_{\Omega,t}(\lambda,p)$.
It is clear that $u_{\lambda}$ is also a maximizer of
\[
\sup_{u\in H_{0}^{1}(\Omega),\Vert\nabla u\Vert_{L^{2}\left(  \Omega\right)
}=\tilde{t}}\int_{\Omega}\left(  e^{{4\pi}u^{2}}-\lambda|u|^{p}\right)
dx\text{,}%
\]
then there exists some $M$ such that
\begin{equation}
M\int_{\Omega}\nabla u_{\lambda}\nabla\phi dx-\int_{\Omega}\left(  {4\pi
}u_{\lambda}e^{{4\pi}u_{\lambda}^{2}}-\frac{p}{2}\lambda u_{\lambda}%
^{p-1}\right)  \phi dx=0\label{3}%
\end{equation}
for any $\phi\in H_{0}^{1}(\Omega)$. Set
\[
f(s):=\int_{\Omega}\left[  e^{{4\pi}\left(  su_{\lambda}\right)  ^{2}}%
-\lambda\left(  su_{\lambda}\right)  ^{p}\right]  dx
\]
for $s\in\lbrack0,1/\tilde{t}]$. It is easily verify that $\left.  f^{\prime
}(s)\right\vert _{s=1}=0$, then we have
\[
\int_{\Omega}\left(  {8\pi}u_{\lambda}^{2}e^{{4\pi}u_{\lambda}^{2}}-p\lambda
u_{\lambda}^{p}\right)  dx=0\text{.}%
\]
Chose $\phi=u_{\lambda}$ in \eqref{3}, we can obtain $M=0$. Then \eqref{3} can
be rewritten as
\[
\int_{\Omega}{4\pi}u_{\lambda}\left(  e^{{4\pi}u_{\lambda}^{2}}-\frac{p}%
{{8\pi}}\lambda u_{\lambda}^{p-2}\right)  \phi dx=0
\]
for any $\phi\in H_{0}^{1}(\Omega)$. Hence, it holds
\[
e^{{4\pi}u_{\lambda}^{2}(x)}-\frac{p}{{8\pi}}\lambda u_{\lambda}^{p-2}(x)=0
\]
for any $x\in\Omega$. However, this is impossible when $x$ near $\partial
\Omega$ because $\left.  u_{\lambda}\right\vert _{\partial\Omega}$ $=0$.
Consequently, the proof is finished.
\end{proof}

At the end of this section, we show that the perturbed
Trudinger-Moser inequalities can always be attained in the case $p>2$ by the
test function argument.

\begin{proof}
[Proof of Theorem \ref{th1.2}]For any $\varepsilon>0,$ we define the function
$f_{\varepsilon}$ by%
\[
f_{\varepsilon}(t)=%
\begin{cases}
C+C^{-1}\left(  -\frac{1}{4\pi}\ln\left(  1+\pi\varepsilon^{-2}e^{-4\pi
t}\right)  +A\right)  , & \text{{if}}\ t\geq t_{\varepsilon},\\
C^{-1}t, & \text{{if}}\ t<t_{\varepsilon},
\end{cases}
\]
and set%
\[
\phi_{\varepsilon}(x)=f_{\varepsilon}(G_{\Omega,y_{0}}\left(  x\right)
),\text{ }x\in\Omega,
\]
where $t_{\varepsilon}=\frac{1}{2\pi}\ln\frac{1}{R\varepsilon},$
$R=-\ln\varepsilon,$ $y_{0}$ is the harmonic center of $\Omega$, $A$ and $C$
are constants depending only on $\varepsilon$ to be determined such that
$\phi_{\varepsilon}\in H_{0}^{1}(\Omega)$ and $\int_{\Omega}\left\vert
\nabla\phi_{\varepsilon}\right\vert ^{2}dx=1$.

Similar to the discussion in \cite{Yang} (see also \cite{LSXZ}), we can obtain%
\begin{equation}
A=-C^{2}-\frac{1}{2\pi}\ln\varepsilon+\frac{1}{4\pi}\ln\pi+O(R^{-2})
\label{eq5.11}%
\end{equation}
with%
\begin{equation}
C^{2}=-\frac{1}{2\pi}\ln\varepsilon+\frac{1}{4\pi}\ln\pi-\frac{1}{4\pi
}+O(R^{-2}). \label{eq5.13}%
\end{equation}

Using (\ref{eq5.11}) and (\ref{eq5.13}), and by direct\ computation as in
\cite{Yang}, we have%
\[
\int_{\{G<t_{\varepsilon}\}}\left(  e^{4\pi\phi_{\varepsilon}^{2}}-\lambda
\phi_{\varepsilon}^{p}\right)  dx\geq\left\vert \Omega\right\vert +\frac
{4\pi\int_{\Omega}G_{\Omega,y_{0}}^{2}\left(  x\right)  dx}{C^{2}}%
-\frac{\lambda\int_{\Omega}G_{\Omega,y_{0}}^{p}dx}{C^{p}}+O(R^{-2})
\]
and
\[
\int_{\{G>t_{\varepsilon}\}}\left(  e^{4\pi\phi_{\varepsilon}^{2}}-\lambda
\phi_{\varepsilon}^{p}\right)  dx\geq\pi e\sup_{x\in\Omega}r_{\Omega}%
^{2}\left(  x\right)  +O(R^{-2}).
\]
Therefore,
\begin{align}
\int_{\Omega}\left(  e^{4\pi\phi_{\varepsilon}^{2}}-\lambda\phi_{\varepsilon
}^{p}\right)  dx  &  \geq\left\vert \Omega\right\vert +\pi e\sup_{x\in\Omega
}r_{\Omega}^{2}\left(  x\right) \nonumber\\
&  ~~+\frac{4\pi\int_{\Omega}G_{\Omega,y_{0}}^{2}dx}{C^{2}}-\frac{\lambda
\int_{\Omega}G_{\Omega,y_{0}}^{p}dx}{C^{p}}+O(R^{-2}). \label{04}%
\end{align}
Since $C^{2}\sim-\ln\varepsilon=R$ by (\ref{eq5.13}) and $p>2$, we conclude
\[
{S_{\Omega}}(\lambda,p)\geq\int_{\Omega}\left(  e^{4\pi\phi_{\varepsilon}^{2}%
}-\lambda\phi_{\varepsilon}^{p}\right)  dx>\left\vert \Omega\right\vert +\pi
e\sup_{x\in\Omega}r_{\Omega}^{2}\left(  x\right)  =S_{\Omega}^{\delta},
\]
for $\varepsilon>0$ small enough, hence ${S_{\Omega}}(\lambda,p)$ is attained
for any $\lambda\in%
\mathbb{R}
$ by Lemma \ref{225}.
\end{proof}

\begin{remark}\label{small}
When $p=2$, from (\ref{04}), we have%
\begin{equation}
\int_{\Omega}\left(  e^{4\pi\phi_{\varepsilon}^{2}}-\lambda\phi_{\varepsilon
}^{2}\right)  dx\geq\left\vert \Omega\right\vert +\pi e\sup_{x\in\Omega
}r_{\Omega}^{2}\left(  x\right)  +\frac{\left(  4\pi-\lambda\right)
\int_{\Omega}G_{\Omega,y_{0}}^{2}dx}{C^{2}}+O(R^{-2})\text{.} \label{05}%
\end{equation}
If
\[
\frac{\left(  4\pi-\lambda\right)  \int_{\Omega}G_{\Omega,y_{0}}^{2}dx}{C^{2}%
}+O(R^{-2})>0\text{,}%
\]
which is equivalent to $\lambda<4\pi+O\left(  \frac{R^{-1}}{\int_{\Omega
}G_{\Omega,y_{0}}^{2}dx}\right)  $, then we can conclude from (\ref{05}) that%
\[
{S_{\Omega}}(\lambda,p)\geq\int_{\Omega}\left(  e^{4\pi\phi_{\varepsilon}^{2}%
}-\lambda\phi_{\varepsilon}^{p}\right)  dx>\left\vert \Omega\right\vert +\pi
e\sup_{x\in\Omega}r_{\Omega}^{2}\left(  x\right)  =S_{\Omega}^{\delta},
\]
for $\varepsilon>0$ small enough, thus ${S_{\Omega}}(\lambda,p)$ can be
attained provided $\lambda<4\pi+O\left(  \frac{R^{-1}}{\int_{\Omega}G_{\Omega
,y_{0}}^{2}dx}\right)  $. Then it is clear from Lemma \ref{225} that
$$\lambda^{\ast}\left(  2\right)  \geq4\pi+O\left(  \frac{R^{-1}}{\int_{\Omega
}G_{\Omega,y_{0}}^{2}dx}\right).  $$
\end{remark}

\section{ Existence and non-existence of extremals for perturbation
Trudinger-Moser inequalities for $1\leq p\leq2$}

In this section, our focus is on analyzing the existence and non-existence of extremals for the perturbation Trudinger-Moser inequality when $1\leq p\leq2$. In this setting, it can be readily demonstrated that there exists a small $\varepsilon_0>0$, such that the perturbation Trudinger-Moser inequalities always possess extremal functions for $\lambda\in\left(  -\infty
,4\pi+\varepsilon_0\right).$  Leveraging this result along with Lemma \ref{225}, we will establish, by contradiction, the existence of a threshold $\lambda^{*}\in\left( 4\pi
,+\infty\right)$ that determines the existence and non-existence of extremals.

Assume that for
any $\lambda\in%
\mathbb{R}
$, $S_{\Omega}\left(  \lambda,p\right)  $ is always attained by some
$u_{\lambda}\in H_{0}^{1}(\Omega)$ with $\int_{\Omega}|\nabla u_{\lambda}%
|^{2}dx=1$. Without loss of generality, we set a sequence $\{\lambda_{k}\}$ such that
$\lambda_{k}\rightarrow\infty$ as $k\rightarrow\infty$ and let $u_{k}$ be the
corresponding maximizer $u_{\lambda_{k}}$ of $S_{\Omega}\left(  \lambda
_{k},p\right)  $. From Lemma \ref{224}, we have
\[
S_{\Omega}(0,p)-\lambda_{k}\int_{\Omega}\left\vert u_{k}\right\vert ^{p}%
\geq\int_{\Omega}\left(  e^{4\pi u_{k}^{2}}-\lambda_{k}\left\vert
u_{k}\right\vert ^{p}\right)  dx=S_{\Omega}\left(  \lambda_{k},p\right)  \geq
S_{\Omega}^{\delta}\text{,}%
\]
which implies
\[
\lambda_{k}\int_{\Omega}\left\vert u_{k}\right\vert ^{p}dx=O_{k}(1)\text{,}%
\]
then we can obtain $u_{k}\rightharpoonup0$ in $H_{0}^{1}(\Omega)$.

Let $x_{k}$ be the maximum point of $u_{k}$ in $\overline{\Omega}$. Up to a
subsequence, we can always assume that
\[
x_{k}\rightarrow x_{0}\in\overline{\Omega},\text{ as }k\rightarrow\infty,
\]
where $x_{0}$ is called the blow up point.\ \ Denote $c_{k}:=u_{k}\left(
x_{k}\right)  =\underset{\Omega}{\max}u_{k}$ and we claim that
\begin{equation}
c_{k}\rightarrow\infty\text{ as }k\rightarrow\infty\text{.} \label{230}%
\end{equation}
In fact, if $\left\{  c_{k}\right\}  $ is bounded, through the Lebesgue
dominated convergence theorem, then
\[
\underset{k\rightarrow\infty}{\lim\inf}\int_{\Omega}\left(  e^{4\pi u_{k}^{2}%
}-\lambda_{k}u_{k}^{p}\right)  dx\leq\underset{k\rightarrow\infty}{\lim\inf
}\int_{\Omega}e^{4\pi u_{k}^{2}}dx=|\Omega|\text{,}%
\]
which is a contradiction with $S_{\Omega}(\lambda_{k},p)\geq S_{\Omega
}^{\delta}>|\Omega|$. Hence, (\ref{230}) holds true.

Note that $u_{k}$ is the extremal function of the supremum $S_{\Omega}\left(
\lambda_{k},p\right)  $, then $u_{k}$ satisfies the Euler-Lagrange equation%
\begin{equation}%
\begin{cases}
-\Delta{u_{k}}=\frac{{{{4\pi}}}}{{{E_{k}}}}\left(  {{u_{k}}{e^{{{4\pi}}%
u_{k}^{2}}}-}\frac{p}{8\pi}\lambda_{k}{{u_{k}^{p-1}}}\right)  ,\  &
\text{in}\ \Omega\text{,}\\
u_{k}>0,\  & \text{in}\ \Omega\text{,}\\
u_{k}=0, & \text{on}\ \partial\Omega\text{,}%
\end{cases}
\label{eq of un}%
\end{equation}
where%
\[
E_{k}:={4\pi}\int_{\Omega}\left(  {{u_{k}^{2}}{e^{{{4\pi}}u_{k}^{2}}}-}%
\frac{p}{8\pi}\lambda_{k}{{u_{k}^{p}}}\right)  dx\text{.}%
\]
Set $v_{k}:=2\sqrt{\pi}u_{k}$, then by (\ref{eq of un}), $v_{k}$ satisfies%
\begin{equation}%
\begin{cases}
-\Delta{v_{k}}=\frac{{{4\pi}}}{{{E_{k}}}}\left(  {{v_{k}}{e^{v_{k}^{2}}}%
-\frac{p}{{2{{\left(  {4\pi}\right)  ^{\frac{p}{2}}}}}}\lambda_{k}v_{k}^{p-1}%
}\right)  ,\  & \text{in}\ \Omega\text{,}\\
{v_{k}}>0, & \text{in}\ \Omega\text{,}\\
v_{k}=0\text{,} & \text{on}\ \partial\Omega\text{,}%
\end{cases}
\label{eq of vn}%
\end{equation}
and%
\begin{equation*}
\int_{\Omega}|\nabla v_{k}|^{2}dx=4\pi\text{,}\quad E_{k}=\int_{\Omega}\left(
v_{k}^{2}e^{v_{k}^{2}}-\frac{p}{2{{\left(  {4\pi}\right)  ^{\frac{p}{2}}}}%
}{\lambda_{k}}v_{k}^{p}\right)  dx\text{.} %
\end{equation*}
The energy $\left\Vert \nabla v_{k}\right\Vert _{L^{2}\left(  \Omega\right)
}^{2}$ can also be written as
\begin{equation}
\left\Vert \nabla v_{k}\right\Vert _{L^{2}\left(  \Omega\right)  }^{2}%
=\frac{{4\pi}}{E_{k}}\int_{\Omega}\left(  v_{k}^{2}e^{v_{k}^{2}}-\frac
{p}{2{{\left(  {4\pi}\right)  ^{\frac{p}{2}}}}}{\lambda_{k}}v_{k}^{p}\right)
dx:=I_{E}-I_{P}\text{.} \label{14}%
\end{equation}

In the sequel, we will analyze the limit behavior of $\left\{  v_{k}\right\}
$ as $k\rightarrow\infty$, and find the contradiction with the fact
$\left\Vert \nabla v_{k}\right\Vert _{L^{2}\left(  \Omega\right)  }^{2}=1$ by
estimating the integral of the Lebesgue term $I_{P}$ and that of the exponential term $I_{E}$, respectively.

\subsection{ Blow-up analysis and estimates  for the Lebesgue integral $I_{P}$
\label{poly}}

In this subsection, we will study the limit behavior of $\left\{
v_{k}\right\}  $ near and far away from the blow-up point $x_{0}$, and
estimates  for the Lebesgue integral $I_{P}$ in (\ref{14}).

\begin{lemma}
\label{231}It holds%
\[
\underset{k\rightarrow\infty}{\lim\inf}E_{k}>0\text{.}%
\]

\end{lemma}

\begin{proof}
 By Lemma \ref{224}, we have%
\begin{align*}
{{S_{\Omega}^{\delta}}} &  \leq\int_{\Omega}\left(  e^{4\pi u_{k}^{2}}%
-\lambda_{k}u_{k}^{p}\right)  dx\\
&  \leq\int_{\{u_{k}\leq1\}}e^{4\pi u_{k}^{2}}dx+\int_{\{u_{k}>1\}}e^{4\pi
u_{k}^{2}}dx-\frac{p}{8\pi}\lambda_{k}\int_{\Omega}u_{k}^{p}dx\\
&  \leq\left\vert \Omega\right\vert +o_{k}\left(  1\right)  +\int
_{\{u_{k}>1\}}u_{k}^{2}e^{4\pi u_{k}^{2}}dx-\frac{p}{8\pi}\lambda_{k}%
\int_{\Omega}u_{k}^{p}dx\\
&  \leq\left\vert \Omega\right\vert +\frac{E_{k}}{4\pi}+o_{k}\left(  1\right)
\text{.}%
\end{align*}
Hence, we obtain
\[
\underset{k\rightarrow\infty}{\lim\inf}E_{k}\geq4\pi\left(  {{S_{\Omega
}^{\delta}}-\left\vert \Omega\right\vert }\right)  >0.
\]
\end{proof}

\begin{lemma}
\label{232} The sequence $\left\{  u_{k}\right\}  _{k}$ must be a
(NCS) at some interior point ${{y_{0}}}\in\Omega$.
\end{lemma}

\begin{proof}
According to the definition of (NCS), we only need to prove that there exists
some interior point ${{y_{0}\in\Omega}}$ such that
\begin{equation}
\lim\limits_{k\rightarrow\infty}\int_{\Omega\backslash{B_{\varepsilon}}\left(
{{y_{0}}}\right)  }|\nabla u_{k}|^{2}dx=0\label{234}%
\end{equation}
for any $\varepsilon>0$. Suppose not, then for any ${{y\in\Omega}}$, there
exists some $\varepsilon_{0}>0$ such that
\begin{equation}
\lim\limits_{k\rightarrow\infty}\int_{\Omega\cap{B}_{\varepsilon_{0}}\left(
{{y}}\right)  }|\nabla u_{k}|^{2}dx\leq\delta<1.\label{235}%
\end{equation}
Next, we will show that $\left\{  u_{k}\left(  x_{k}\right)  \right\}  $ is
bounded under the assumption (\ref{235}). Now, we introduce a new sequence
$\left\{  \tilde{u}_{k}\right\}  $ with $\tilde{u}_{k}$ satisfies
\begin{equation}%
\begin{cases}
-\Delta{\tilde{u}_{k}}=\frac{4\pi}{E_{k}}u_{k}e^{4\pi u_{k}^{2}},\  &
\text{in}\ \Omega\text{,}\\
\tilde{u}_{k}>0,\  & \text{in}\ \Omega\text{,}\\
\tilde{u}_{k}=0, & \text{on}\ \partial\Omega\text{.}%
\end{cases}
\label{com}%
\end{equation}
From (\ref{eq of un}) and (\ref{com}), we can obtain $u_{k}\leq\tilde{u}_{k}$
by maximum principle. Hence, the boundedness of $\left\{  u_{k}\right\}  $ can
be proved if we can show that $\left\{  \tilde{u}_{k}\right\}  $ is bounded.
For this, we choose a cut-off function $\phi$ such that%
\[
\phi(x)\in C_{c}^{\infty}(\mathbb{R}^{2}),\quad\phi(x)=0\text{ in }%
{\mathbb{R}^{2}\backslash B_{\varepsilon_{0}}(y)},\quad\phi(x)=1\text{ in
}{{B_{\varepsilon_{0}/2}(y)}}.
\]
\ Using (\ref{235}), combining the fact that $u_{k}\rightharpoonup0$ in
$H_{0}^{1}\left(  \Omega\right)  $ and the Sobolev compact imbedding, we have
\[
\lim\limits_{k\rightarrow\infty}\int_{\Omega\cap B_{\varepsilon_{0}}%
(y)}|\nabla(u_{k}\phi)|^{2}dx\leq\delta^{\prime}\text{,}%
\]
\ for some $\delta^{\prime}<1$, then from Trudinger-Moser inequality
(\ref{1}), we see that $\left\{  u_{k}e^{4\pi u_{k}^{2}}\right\}  $ is bounded
in $L^{q}({{B_{\varepsilon_{0}}(y)}})$ for some $q>1$. Since $y$ is an
arbitrary point in $\Omega$, we may cover $\Omega$ by finitely many balls
${{B_{\varepsilon_{0}}(y)}}$ to see that $\left\{  u_{k}e^{4\pi u_{k}^{2}%
}\right\}  $ is bounded in $L^{q}\left(  \Omega\right)  $ for $q>1$. This
together with equation \eqref{com} yields that $\left\{  \tilde{u}%
_{k}\right\}  $ is bounded in $\Omega$ by standard elliptic theory, this
implies $\left\{  u_{k}\left(  x_{k}\right)  \right\}  $ is bounded, which
contradicts with (\ref{230}). Therefore, (\ref{234}) holds true for some
$y_{0}\in\overline{\Omega}$, then $\left\{  u_{k}\right\}  $ is a (NCS) at
$y_{0}\in\overline{\Omega}$.

Now, we show that $y_{0}$ must be an interior point of $\Omega$. In fact, if
$y_{0}\in\partial\Omega$, by Remark \ref{re2.2}, we have%
\begin{equation}
\underset{k\rightarrow\infty}{\lim}\int_{\Omega}e^{4\pi u_{k}^{2}}%
dx\leq\left\vert \Omega\right\vert \text{.}\label{237}%
\end{equation}
Then from Lemma \ref{224}, we can derive that
\[
S_{\Omega}^{\delta}\leq\lim\limits_{k\rightarrow\infty}\int_{\Omega}\left(
e^{4\pi u_{k}^{2}}-\lambda_{k}u_{k}^{p}\right)  dx\leq\lim
\limits_{k\rightarrow\infty}\int_{\Omega}e^{4\pi u_{k}^{2}}dx\text{,}%
\]
which is a contradiction with (\ref{237}). Hence, $y_{0}$ must be an interior
point of $\Omega$, the proof is finished.
\end{proof}

\begin{remark}

Currently, we cannot claim that $y_{0}$ coincides with the blow-up point $x_{0}$; this will be rigorously demonstrated in Remark \ref{corresp}.

\end{remark}

\begin{lemma}
\label{lem12} We have%
\begin{equation}
\lim\limits_{k\rightarrow\infty}\int_{\Omega}e^{v_{k}^{2}}dx=S_{\Omega
}^{\delta}\text{,}\label{vanish pe}%
\end{equation}
and%
\begin{equation}
\lim\limits_{k\rightarrow\infty}\lambda_{k}\int_{\Omega}v_{k}^{p}%
dx=0\text{.}\label{238}%
\end{equation}

\end{lemma}

\begin{proof}
It follows from Lemma \ref{232} and Lemma \ref{224} that
\begin{align*}
S_{\Omega}^{\delta} &  \leq\underset{k\rightarrow\infty}{\lim\inf}\int
_{\Omega}\left(  e^{4\pi u_{k}^{2}}-\lambda_{k}\left\vert u_{k}\right\vert
^{p}\right)  dx\\
&  \leq\underset{k\rightarrow\infty}{\lim\sup}\int_{\Omega}e^{4\pi u_{k}^{2}%
}dx\leq|\Omega|+\pi er_{\Omega}^{2}\left(  y_{0}\right)  \leq S_{\Omega
}^{\delta}\text{.}%
\end{align*}
This implies that the $y_{0}$ must be the maximum point of the harmonic radius
$r_{\Omega}$. Recalling that $v_{k}=2\sqrt{\pi}u_{k}$, we immediately obtain
(\ref{vanish pe}) and (\ref{238}).
\end{proof}

We set $\gamma_k=2\sqrt{\pi}c_k$,
\begin{equation}
r_{k}^{2}:=\frac{E_{k}}{\pi}\gamma_{k}^{-2}e^{-\gamma_{k}^{2}}\text{,}%
\label{3.5}%
\end{equation}
and define
\[
\left\{
\begin{array}
[c]{l}%
\psi_{k}(y):=\gamma_{k}^{-1}v_{k}\left(  x_{k}+r_{k}y\right)  ,\\
\phi_{k}(y):=\gamma_{k}\left(  v_{k}\left(  x_{k}+r_{k}y\right)  -\gamma
_{k}\right)  .
\end{array}
\right.
\]
Then $\psi_{k}$ and $\phi_{k}$ are\ well defined on $\Omega_{k}:=\left\{
y\in\mathbb{R}^{2}:x_{k}+r_{k}y\in\Omega\right\}  $, and\ satisfy the
following equations %

\begin{equation}
-\Delta\psi_{k}=\frac{4}{\gamma_{k}^{2}}\left[  \psi_{k}e^{\gamma_{k}%
^{2}\left(  \psi_{k}^{2}-1\right)  }-\frac{p}{2{{\left(  {4\pi}\right)
^{\frac{p}{2}}}}}\gamma_{k}^{p-2}e^{-\gamma_{k}^{2}}\lambda_{k}\psi_{k}%
^{p-1}\right]  \text{,}\label{19}%
\end{equation}%
\begin{equation}
-\Delta\phi_{k}=4\left[  \psi_{k}e^{\phi_{k}\left(  1+\psi_{k}\right)  }%
-\frac{p}{2{{\left(  {4\pi}\right)  ^{\frac{p}{2}}}}}\gamma_{k}^{p-2}%
e^{-\gamma_{k}^{2}}\lambda_{k}\psi_{k}^{p-1}\right]  \text{. }\label{18}%
\end{equation}
Since $E_{k}\leq\int_{\Omega}v_{k}^{2}e^{v_{k}^{2}}dx\leq\gamma_{k}%
^{2}S_{\Omega}\left(  0,p\right)  $, we can check that $r_{k}=O\left(
e^{-\frac{\gamma_{k}^{2}}{2}}\right)  $.

 Next, we will study the asymptotic behavior of $\psi_{k}$ and
$\phi_{k}$. We first show the boundedness of $\left\{  \frac{2p}{{\left(
{4\pi}\right)  ^{p/2}}}\gamma_{k}^{p-2}e^{-\gamma_{k}^{2}}\lambda_{k}\right\}
$. Multiplying \eqref{19} by $\gamma_{k}^{2}\psi_{k}$ and integrating on
$\Omega_{k}$, we have%

\[
4\pi=4\left[  \int_{\Omega_{k}}\psi_{k}^{2}e^{\gamma_{k}^{2}(\psi_{k}^{2}%
-1)}dy-\frac{p}{2(4\pi)^{\frac{p}{2}}}\gamma_{k}^{p-2}e^{-\gamma_{k}^{2}%
}\lambda_{k}\int_{\Omega_{k}}\psi_{k}^{p}dy\right]  \text{,}%
\]
then%
\begin{equation}
\frac{2p}{(4\pi)^{p/2}}\gamma_{k}^{p-2}e^{-\gamma_{k}^{2}}\lambda_{k}%
=\frac{4\int_{\Omega_{k}}\psi_{k}^{2}e^{\gamma_{k}^{2}(\psi_{k}^{2}-1)}%
dy-4\pi}{\int_{\Omega_{k}}\psi_{k}^{p}dy}\text{.}\label{20}%
\end{equation}
Since $1\leq p\leq2$ and $\psi_{k}\leq1$, we have
\begin{equation*}
\frac{2p}{{\left(  {4\pi}\right)  ^{p/2}}}\gamma_{k}^{p-2}e^{-\gamma_{k}^{2}%
}\lambda_{k}\leq4-\frac{4\pi}{\int_{\Omega_{k}}\psi_{k}^{p}dy}\leq
4\text{.}%
\end{equation*}
This implies that
\[
\frac{4}{\gamma_{k}^{2}}\left[  \psi_{k}e^{\gamma_{k}^{2}\left(  \psi_{k}%
^{2}-1\right)  }-\frac{p}{2{{\left(  {4\pi}\right)  ^{\frac{p}{2}}}}}%
\gamma_{k}^{p-2}e^{-\gamma_{k}^{2}}\lambda_{k}\psi_{k}^{p-1}\right]
=O(1/\gamma_{k}^{2})\text{.}%
\]
Applying the elliptic regularity theory to equation \eqref{19}, we derive that%
\begin{equation}
\psi_{k}\rightarrow\psi\text{ in }C_{loc}^{1}(\mathbb{R}^{2})\text{,}%
\label{247}%
\end{equation}
with $\psi(x)$ satisfying
\[
-\Delta\psi=0,\ \ x\in\mathbb{R}^{2}\text{.}%
\]
Since $\psi_{k}(0)=1$, from Liouville-type theorem we have $\psi=1$.

Next, we claim that
\begin{equation}
\lim_{k\rightarrow\infty}\gamma_{k}^{p-2}e^{-\gamma_{k}^{2}}\lambda
_{k}=0\text{.} \label{22}%
\end{equation}
Indeed, it follows from Lemma \ref{231} and (\ref{238}) that
\begin{align}
4\int_{\Omega_{k}}\psi_{k}^{2}e^{\gamma_{k}^{2}(\psi_{k}^{2}-1)}dy-4\pi &
=\frac{4\pi}{E_{k}}\left(  \int_{\Omega}v_{k}^{2}e^{v_{k}^{2}}dx-E_{k}\right)
\nonumber\\
&  =\frac{\left(  4\pi\right)  ^{1-\frac{p}{2}}}{E_{k}}\frac{p}{2}\lambda
_{k}\int_{\Omega}v_{k}^{p}dx\label{ad2}\\
&  =o_{k}\left(  1\right)  \text{.}\nonumber
\end{align}
On the other hand, since $\psi_{k}\rightarrow1$ in $C_{loc}^{1}(\mathbb{R}%
^{2})$, then%

\begin{equation}
\lim_{k\rightarrow\infty}\int_{\Omega_{k}}\psi_{k}^{p}dx\geq\lim
\limits_{R\rightarrow\infty}\lim\limits_{k\rightarrow\infty}\int_{{B_{R}}}%
\psi_{k}^{p}dx=\infty\text{.}\label{add2}%
\end{equation}
Combining (\ref{20}), (\ref{ad2}) and (\ref{add2}), we conclude that
(\ref{22}) holds true. From (\ref{247}) and (\ref{22}), applying the elliptic
regularity theory to equation (\ref{18}), we obtain that there exists some
$\phi_{\infty}$ such that
\begin{equation}
\phi_{k}\rightarrow\phi_{\infty}\text{ in }C_{loc}^{1}(\mathbb{R}^{2}%
)\text{.}\label{214}%
\end{equation}
From Chen and Li's result (see \cite{Chenli}), $\phi_{\infty}$ must take the
form as
\begin{equation}
\phi_{\infty}(x)=-\log\left(  1+|x|^{2}\right)  \text{,}\quad\label{3.6}%
\end{equation}
and satisfies%
\begin{equation}
-\Delta\phi_{\infty}=4e^{2\phi_{\infty}}\text{ \ in }\mathbb{R}^{2}\text{, and
}\int_{%
\mathbb{R}
^{2}}e^{2\phi_{\infty}}dx=\pi.\label{add690}%
\end{equation}

For any constant $A>1$, we define the truncated function
\[
v_{k,A}:=\min\left\{  \frac{\gamma_{k}}{A},v_{k}\right\}  \text{.}%
\]

\begin{lemma}
\label{le4.4} We have
\[
\lim_{k\rightarrow\infty}\int_{\Omega}\left\vert \nabla v_{k,A}\right\vert
^{2}dx=\frac{4\pi}{A}.
\]

\end{lemma}

\begin{proof}
Testing the equation (\ref{eq of vn}) by $v_{k,A}$ and using Green's formula,
for any $R>0$, one has
\begin{align*}
\int_{\Omega}\left\vert \nabla v_{k,A}\right\vert ^{2}dx &  =\int_{\Omega
}v_{k,A}\frac{{{4\pi}}}{{{E_{k}}}}\left(  {{v_{k}}{e^{v_{k}^{2}}}-\frac
{p}{{2{{\left(  {4\pi}\right)  ^{\frac{p}{2}}}}}}\lambda_{k}v_{k}^{p-1}%
}\right)  dx\\
&  =\int_{\Omega}\frac{{{4\pi}}}{{{E_{k}}}}v_{k,A}{{v_{k}}{e^{v_{k}^{2}}}%
}dx+o_{k}\left(  1\right)  \\
&  \geq\int_{B_{Rr_{k}}\left(  x_{k}\right)  }\frac{{{4\pi}}}{{{E_{k}}}%
}v_{k,A}{{v_{k}}{e^{v_{k}^{2}}}}dx+o_{k}\left(  1\right)  \\
&  =\frac{1}{A}\int_{B_{R}\left(  0\right)  }4e^{2\phi_{\infty}}%
dx+o_{k,R}\left(  1\right)  \text{.}%
\end{align*}
where in the second equality we have used Lemma \ref{231} and (\ref{238}).
Letting $R\rightarrow\infty$, we derive that
\begin{equation}
\lim_{k\rightarrow\infty}\int_{\Omega}\left\vert \nabla v_{k,A}\right\vert
^{2}dx\geq\frac{4\pi}{A}.\label{add6103}%
\end{equation}
Using the same argument above, we can show that%
\[
\lim_{k\rightarrow\infty}{\int_{\Omega}{\left\vert {\nabla\left(  {{v_{k}}%
-}\frac{\gamma_{k}}{A}\right)  }^{+}\right\vert }^{2}}dx\geq4\pi-\frac{{4\pi}%
}{A}.
\]
Note that
\[
{\int_{\Omega}{\left\vert \nabla v_{k,A}\right\vert }^{2}}dx+{\int_{\Omega
}{\left\vert {\nabla\left(  {{v_{k}}-}\frac{\gamma_{k}}{A}\right)  }%
^{+}\right\vert }^{2}}dx=4\pi,
\]
hence we conclude that
\begin{equation}
\lim_{k\rightarrow\infty}\int_{\Omega}\left\vert \nabla v_{k,A}\right\vert
^{2}dx\leq\frac{4\pi}{A}.\label{add6104}%
\end{equation}

Combining (\ref{add6103}) and (\ref{add6104}), we accomplish the proof.
\end{proof}

\begin{lemma}
\label{le4.5} There holds
\[
\lim\limits_{k\rightarrow\infty}\frac{E_{k}}{\gamma_{k}^{2}}=S_{\Omega
}^{\delta}-|\Omega|\text{.}%
\]

\end{lemma}

\begin{proof}
We divide the proof in two steps. First, we prove
\begin{equation}
\underset{k\rightarrow\infty}{\lim\inf}\frac{E_{k}}{\gamma_{k}^{2}}\geq
S_{\Omega}^{\delta}-|\Omega|\text{.}\label{218}%
\end{equation}
Due to (\ref{eq of vn}) and (\ref{238}), it follows that%
\begin{align*}
\underset{k\rightarrow\infty}{\lim}\int_{\Omega}e^{v_{k}^{2}}dx &
=\underset{k\rightarrow\infty}{\lim}\int_{\{v_{k}<\frac{\gamma_{k}}{A}%
\}}e^{v_{k}^{2}}dx+\underset{k\rightarrow\infty}{\lim}\int_{\{v_{k}\geq
\frac{\gamma_{k}}{A}\}}e^{v_{k}^{2}}dx\\
&  \leq\underset{k\rightarrow\infty}{\lim}\int_{\Omega}e^{v_{k,A}^{2}%
}dx+\underset{k\rightarrow\infty}{\lim}\frac{A^{2}}{\gamma_{k}^{2}}%
\int_{\Omega}v_{k}^{2}e^{v_{k}^{2}}dx\\
&  =\left\vert \Omega\right\vert +\underset{k\rightarrow\infty}{\lim}%
\frac{A^{2}}{\gamma_{k}^{2}}\left(  E_{k}+\frac{p}{2\left(  4\pi\right)
^{\frac{p}{2}}}\lambda_{k}\int_{\Omega}v_{k}^{p}dx\right)  \\
&  =\left\vert \Omega\right\vert +\underset{k\rightarrow\infty}{\lim}%
\frac{A^{2}}{\gamma_{k}^{2}}E_{k}\text{.}%
\end{align*}
Letting $A\rightarrow1$, we obtain that (\ref{218}) holds from
(\ref{vanish pe}).

It remains to show that%
\begin{equation}
\underset{k\rightarrow\infty}{\lim\sup}\frac{E_{k}}{\gamma_{k}^{2}}\leq
S_{\Omega}^{\delta}-|\Omega|.\label{222}%
\end{equation}
By a direct computation, we have
\begin{align*}
\underset{k\rightarrow\infty}{\lim}\int_{\Omega}\left(  e^{v_{k}^{2}%
}-1\right)  dx &  \geq\underset{R\rightarrow\infty}{\lim}\underset
{k\rightarrow\infty}{\lim}\int_{B_{Rr_{k}}\left(  x_{k}\right)  }\left(
e^{v_{k}^{2}}-1\right)  dx\\
&  =\underset{R\rightarrow\infty}{\lim}\underset{k\rightarrow\infty}{\lim}%
\int_{B_{Rr_{k}}\left(  x_{k}\right)  }e^{v_{k}^{2}}dx\\
&  =\underset{R\rightarrow\infty}{\lim}\underset{k\rightarrow\infty}{\lim
}\frac{E_{k}}{\gamma_{k}^{2}}\frac{1}{\pi}\int_{B_{R}\left(  0\right)
}e^{2\phi_{\infty}}dx\\
&  =\underset{k\rightarrow\infty}{\lim}\frac{E_{k}}{\gamma_{k}^{2}}\text{,}%
\end{align*}
which implies that (\ref{222}) holds true. Thus we complete the proof.
\end{proof}

\begin{remark}\label{corresp}
Based on the argument for Lemma \ref{le4.5}, we can conclude that the concentration point $y_{0}$ is precisely the blow-up point $x_{0}$.

\end{remark}

\begin{lemma}
\label{le4.6} For any $\eta\in C^{\infty}(\Omega)$, we have
\begin{equation*}
\lim_{k\rightarrow\infty}\frac{1}{E_{k}}\int_{\Omega}\gamma_{k}v_{k}%
e^{v_{k}^{2}}\eta dx=\eta(x_{0}).%
\end{equation*}

\end{lemma}

\begin{proof}
For any \thinspace$A>1$, we divide $\Omega$ into three parts:
\[
\Omega=\left(  {\left\{  {{v_{k}}>\frac{\gamma_{k}}{A}}\right\}
\backslash{B_{R{r}_{k}}}\left(  {{x_{k}}}\right)  }\right)  \cup\left(
{\left\{  {{v_{k}\leq}\frac{\gamma_{k}}{A}}\right\}  \backslash{B_{R{r}_{k}}%
}\left(  {{x_{k}}}\right)  }\right)  \cup{{B_{R{r}_{k}}}\left(  {{x_{k}}%
}\right)  }.
\]
Hence the integral $\frac{1}{E_{k}}\int_{\Omega}\gamma_{k}v_{k}e^{v_{k}^{2}%
}\eta dx$ can be rewritten as%
\begin{align*}
\frac{1}{E_{k}}\int_{\Omega}\gamma_{k}v_{k}e^{v_{k}^{2}}\eta dx &  =\frac
{1}{E_{k}}\int_{{\left\{  {{v_{k}}>\frac{\gamma_{k}}{A}}\right\}
\backslash{B_{R{r}_{k}}}\left(  {{x_{k}}}\right)  }}\gamma_{k}v_{k}%
e^{v_{k}^{2}}\eta dx+\frac{1}{E_{k}}\int_{{{B_{R{r}_{k}}}\left(  {{x_{k}}%
}\right)  }}\gamma_{k}v_{k}e^{v_{k}^{2}}\eta dx\\
&  +\frac{1}{E_{k}}\int_{{\left\{  {{v_{k}}\leq\frac{\gamma_{k}}{A}}\right\}
\backslash{B_{R{r}_{k}}}\left(  {{x_{k}}}\right)  }}\gamma_{k}v_{k}%
e^{v_{k}^{2}}\eta dx\\
&  :=I_{1}^{k,R}+I_{2}^{k,R}+I_{3}^{k,R}\text{.}%
\end{align*}
For $I_{1}^{k,R}$, we derive that%
\begin{align*}
I_{1}^{k,R} &  \leq A\underset{\Omega}{\sup}\left\vert \eta\right\vert
\int_{{\Omega\backslash{B_{R{r}_{k}}}\left(  {{x_{k}}}\right)  }}\frac
{1}{E_{k}}v_{k}^{2}e^{v_{k}^{2}}dx\\
&  \leq A\underset{\Omega}{\sup}\left\vert \eta\right\vert \left(
1-\int_{{{B_{R{r}_{k}}}\left(  {{x_{k}}}\right)  }}\frac{1}{E_{k}}v_{k}%
^{2}e^{v_{k}^{2}}dx\right)  \\
&  \leq A\underset{\Omega}{\sup}\left\vert \eta\right\vert \left(  1-\frac
{1}{\pi}\int_{{{B}}_{R}\left(  0\right)  }e^{2\phi_{\infty}+o_{k}\left(
1\right)  }dy\right)  \text{.}%
\end{align*}
Thus, $\underset{R\rightarrow\infty}{\lim}\underset{k\rightarrow\infty}{\lim
}I_{1}^{k,R}=0$. A careful calculation for $I_{2}^{k,R}$ gives
\begin{align*}
I_{2}^{k,R} &  =\frac{1}{E_{k}}\int_{{{B_{R{r}_{k}}}\left(  {{x_{k}}}\right)
}}\gamma_{k}v_{k}e^{v_{k}^{2}}\eta dx\\
&  =\int_{{{B}}_{R}\left(  0\right)  }\frac{1}{\pi}\eta\left(  x_{k}%
+r_{k}y\right)  e^{2\phi_{\infty}+o_{k}\left(  1\right)  }dy\\
&  =\eta\left(  x_{k}\right)  \int_{{{B}}_{R}\left(  0\right)  }\frac{1}{\pi
}e^{2\phi_{\infty}}dy+o_{k}\left(  1\right)  \text{,}%
\end{align*}
which implies $\underset{R\rightarrow\infty}{\lim}\underset{k\rightarrow
\infty}{\lim}I_{2}^{k,R}=\eta\left(  x_{0}\right)  $. By Lemma \ref{le4.4}, we
have%
\[
\int_{\Omega}e^{qv_{k,A}^{2}}dx<C
\]
for any $1<q<A$. Using Holder's inequality, we have
\begin{align*}
I_{3}^{k,R} &  \leq\underset{\Omega}{\sup}\left\vert \eta\right\vert
\frac{\gamma_{k}}{E_{k}}\int_{{\left\{  {{v_{k}}\leq\frac{\gamma_{k}}{A}%
}\right\}  \backslash{B_{R{r}_{k}}}\left(  {{x_{k}}}\right)  }}v_{k}%
e^{v_{k}^{2}}dx\\
&  \leq\underset{\Omega}{\sup}\left\vert \eta\right\vert \frac{\gamma_{k}%
}{E_{k}}\int_{\Omega}v_{k}e^{v_{k,A}^{2}}dx\leq\underset{\Omega}{\sup
}\left\vert \eta\right\vert \frac{\gamma_{k}}{E_{k}}\left\vert \left\vert
v_{k}\right\vert \right\vert _{L^{\frac{q}{q-1}}\left(  \Omega\right)
}\left\vert \left\vert e^{v_{k,A}^{2}}\right\vert \right\vert _{L^{q}\left(
\Omega\right)  }\rightarrow0\text{,}%
\end{align*}
as $k\rightarrow\infty$, where in the last inequality we have used Lemma
\ref{le4.5}. Then Lemma \ref{le4.6} follows immediately from the estimates for
$I_{1}^{k,R}$, $I_{2}^{k,R}$ and $I_{3}^{k,R}$.
\end{proof}

Below, we can demonstrate the convergence of $\{\gamma_{k}v_{k}\}$. Since this proof of convergence is analogous to that in \cite{Has}, we just present the main results here, the detailed proof provided in the Appendix.

\begin{proposition}
\label{pro4.9} For any $q\in\lbrack1,\infty)$, we have%
\[
\int_{\Omega}\left(  \gamma_{k}v_{k}\right)  ^{q}dx\rightarrow\int_{\Omega
}G_{\Omega,x_{0}}^{q}dx
\]
as $k\rightarrow\infty$, where $G_{\Omega,x_{0}}$ is the Green function of
Laplacian operator with the singularity at $x_{0}$.
\end{proposition}

\begin{remark}
\label{888} From (\ref{238}), it follows that as $k\rightarrow\infty$%
\[
o_{k}(1)=\lambda_{k}\int_{\Omega}v_{k}^pdx=\frac{\lambda_{k}}{\gamma_{k}^p}%
\int_{\Omega}(\gamma_{k}v_{k})^pdx=\frac{\lambda_{k}}{\gamma_{k}^p}\left[
\int_{\Omega}G^p_{\Omega,x_{0}}dx+o_{k}(1)\right]  \text{,}%
\]
which implies
\begin{equation}
\frac{\lambda_{k}}{\gamma_{k}^p}\rightarrow0\text{.}\label{227}%
\end{equation}
Form (\ref{227}) and using the standard elliptic regularity theory, we can
get
\begin{equation}
\gamma_{k}v_{k}\rightarrow G_{\Omega,x_{0}}\text{ in }C_{loc}^{1}\left(
\Omega\backslash x_{0}\right)  .\label{2271}%
\end{equation}

\end{remark}

By Proposition \ref{pro4.9} and Lemma \ref{le4.5}, we have the following
estimates  for Lebesgue perturbation in (\ref{14}).

\begin{proposition}
\label{223}There holds%
\begin{equation*}
\frac{p\lambda_{k}\left(  4\pi\right)  ^{1-\frac{p}{2}}}{2E_{k}}\int_{\Omega
}v_{k}^{p}dx=\frac{p\lambda_{k}\left(  4\pi\right)  ^{1-\frac{p}{2}}}{2\left(
S_{\Omega}^{\delta}-\left\vert \Omega\right\vert \right)  \gamma_{k}^{p+2}%
}\left(  \int_{\Omega}G_{\Omega,x_{0}}^{p}dx+o_{k}\left(  1\right)  \right)
,%
\end{equation*}
as $k\rightarrow\infty$.
\end{proposition}
In the subsequent analysis, our objective is to estimate the integral of the exponential term $I_E$.  To achieve this, we will partition the integral into two distinct regions: a small ball and the exterior region beyond this ball. The difficulty lies in estimating the integral within the
small ball around the blow up point.  Unlike the method employed in \cite{Has}, which benefits from radial symmetrization, our scenario is complicated by the lack of such symmetry. To surmount this challenge, we will adopt the strategy from \cite{Dru2}, which involves establishing a connection between $ v_k $ and certain radial functions that exhibit favorable asymptotic behavior. This strategy is based on the following point-wise estimate on the gradient of $v_k$.

\subsection{Point-wise estimates  on the gradient }
In this subsection, we aim to establish point-wise estimates  on the gradient of $v_k$, which is a solution of ( \ref{eq of vn} ). To this end, we first derive preliminary weak point-wise estimates  for $v_k$.

\begin{lemma}
\label{52}  There exists $C>0$ such that $\ $%
\[
\left\vert x-x_{k}\right\vert ^{2}\left\vert -\Delta v_{k}\right\vert
v_{k}\leq C\text{ in }\tilde{\Omega}_k\text{,}%
\]
where $\tilde{\Omega}_k=\Omega$ if $1\leq p<2$, and $\tilde{\Omega}_k=\{x\in\Omega \mid v_k(x)\geq \log(\gamma_k)\}$ if $p=2$.
\end{lemma}

\begin{proof}
We shall prove this by contradiction for the case \(1 \leq p < 2\). The proof for the case \(p = 2\) follows a similar approach. Assume that there exists some $y_{k}%
\in\Omega$ such that%
\begin{align}
&  \underset{x\in\Omega}{\sup}\left\vert x-x_{k}\right\vert ^{2}\left\vert
-\Delta v_{k}\right\vert v_{k}\nonumber\\
&  =\left\vert y_{k}-x_{k}\right\vert ^{2}\frac{4\pi}{E_{k}}\left\vert
v_{k}^{2}\left(  y_{k}\right)  e^{v_{k}^{2}\left(  y_{k}\right)  }-\frac
{p}{2(4\pi)^{\frac{p}{2}}}\lambda_{k}v_{k}^{p}\left(  y_{k}\right)
\right\vert \rightarrow\infty\text{ as }k\rightarrow\infty\,\text{.}
\label{58}%
\end{align}

We first show $-\Delta v_{k}\left(  y_{k}\right)  \geq0$. Indeed, if $-\Delta
v_{k}\left(  y_{k}\right)  <0$, then from (\ref{58}), there holds
\begin{equation}
\frac{4\pi}{E_{k}}\frac{p}{2(4\pi)^{\frac{p}{2}}}\lambda_{k}v_{k}^{p}\left(
y_{k}\right)  \rightarrow\infty\text{ as }k\rightarrow\infty\text{.}
\label{59}%
\end{equation}
Combining (\ref{59}) with Lemma \ref{le4.5} and (\ref{227}),\ we have%
\[
\frac{v_{k}^{p}\left(  y_{k}\right)  }{\gamma_{k}^{2-p}}\rightarrow\infty\text{ as
}k\rightarrow\infty\text{,}%
\]
which implies
\[
\frac{p\lambda_{k}v_{k}^{p}\left(  y_{k}\right)  }{2(4\pi)^{\frac{p}{2}}%
v_{k}^{2}\left(  y_{k}\right)  e^{v_{k}^{2}\left(  y_{k}\right)  }}%
\rightarrow0,\text{ as }k\rightarrow\infty\text{.}%
\]
Hence, $-\Delta v_{k}\left(  y_{k}\right)  \geq0$. This and (\ref{58}) imply
that%
\begin{equation}
\frac{4\pi}{E_{k}}v_{k}^{2}\left(  y_{k}\right)  e^{v_{k}^{2}\left(
y_{k}\right)  }\rightarrow\infty\text{ as }k\rightarrow\infty\text{.}
\label{60}%
\end{equation}
Using (\ref{60}) and Lemma \ref{le4.5}, it is easy to see that%
\begin{equation}
\hat{\gamma}_{k}:=v_{k}\left(  y_{k}\right)  \rightarrow\infty\text{ as
}k\rightarrow\infty\text{,} \label{add22}%
\end{equation}
and
\begin{equation}
\frac{\hat{\gamma}_{k}^pe^{\frac{p}{2}\hat{\gamma}_{k}^{2}}}{\gamma_{k}^p%
}\rightarrow\infty\text{ as }k\rightarrow\infty\text{.} \label{add1}%
\end{equation}
Hence, for any $\alpha>\frac{p}{2}\geq\frac{1}{2},$ we get%
\begin{equation}
\frac{e^{\alpha\hat{\gamma}_{k}^{2}}}{\gamma_{k}^p}\geq\frac{\hat{\gamma}^p%
_{k}e^{\frac{p}{2}\hat{\gamma}_{k}^{2}}}{\gamma_{k}^p}\rightarrow\infty\text{ as
}k\rightarrow\infty\text{.} \label{63}%
\end{equation}

Using (\ref{63}) and (\ref{227}), we can obtain the following relationship
between $\lambda_{k}$ and $\hat{\gamma}_{k}:$
\begin{equation}
\frac{\lambda_{k}}{e^{\alpha\hat{\gamma}_{k}^{2}}}=\frac{\lambda_{k}}%
{\gamma_{k}^p}\frac{\gamma_{k}^p}{e^{\alpha\hat{\gamma}_{k}^{2}}}\rightarrow
0\text{ as }k\rightarrow\infty\text{,} \label{64}%
\end{equation}
for any $\alpha>\frac{1}{2}$. Set%
\begin{equation}
\hat{r}_{k}^{2}=\frac{E_{k}}{\pi}\hat{\gamma}_{k}^{-2}e^{-\hat{\gamma}_{k}%
^{2}}\text{.} \label{add221}%
\end{equation}
From (\ref{add1}) and Lemma \ref{le4.5}, it is easy to check $\hat{r}%
_{k}\rightarrow0$ as $k\rightarrow\infty$. By (\ref{58}), (\ref{add22}),
(\ref{64}) and (\ref{add221}), we can get%
\begin{equation}
\underset{x\in\Omega}{\sup}\left\vert x-x_{k}\right\vert ^{2}\left\vert
-\Delta v_{k}\right\vert v_{k}=4\left\vert \frac{y_{k}-x_{k}}{\hat{r}_{k}%
}\right\vert ^{2}\left(  1+o_{k}\left(  1\right)  \right)  \rightarrow
\infty\text{ as }k\rightarrow\infty\,\text{.} \label{65}%
\end{equation}
Since $r_{k}^{-2}>\hat{r}_{k}^{-2}$ where $r_{k}$ is defined as (\ref{3.5}),
by (\ref{65}), we derive that
\begin{equation}
\frac{\left\vert y_{k}-x_{k}\right\vert }{r_{k}}\rightarrow\infty\text{ \ as
}k\rightarrow\infty\text{.} \label{above}%
\end{equation}

For $R>0$, we set $\Omega_{R,k}=B_{R\hat{r}_{k}}\left(  y_{k}\right)
\cap\Omega$ and $\tilde{\Omega}_{R,k}=\left(  \Omega_{R,k}-y_{k}\right)
/\hat{r}_{k}$. Let $\hat{v}_{k}$ be given by%
\[
\hat{v}_{k}\left(  x\right)  =\hat{\gamma}_{k}\left(  v_{k}\left(  y_{k}%
+\hat{r}_{k}x\right)  -\hat{\gamma}_{k}\right)
\]
for $x\in\tilde{\Omega}_{R,k}$.

Now we claim that for all $R>0$,%
\begin{equation}
\underset{k\rightarrow\infty}{\lim\sup}\underset{x\in\tilde{\Omega}_{R,k}%
}{\sup}\hat{v}_{k}\left(  x\right)  \leq0\text{.}\label{67}%
\end{equation}
We prove this by contradiction. If (\ref{67}) does not hold true, then for
some $R>0$, we may assume that there exists $z_{k}\in\tilde{\Omega}_{R,k}$
such that
\begin{equation}
\hat{v}_{k}\left(  z_{k}\right)  =\hat{\gamma}_{k}\left(  v_{k}\left(
y_{k}+\hat{r}_{k}z_{k}\right)  -\hat{\gamma}_{k}\right)  \rightarrow\alpha
_{0}\in\left(  0,+\infty\right)  \label{77}%
\end{equation}
as $k\rightarrow\infty$, thus, there exists some $k_{0}>0$ such that
\begin{equation}
v_{k}\left(  y_{k}+\hat{r}_{k}z_{k}\right)  >v_{k}\left(  y_{k}\right)
\label{68}%
\end{equation}
for any $k\geq$ $k_{0}$. From the definition of $y_{k}$, we have
\begin{equation}
\left.  \left\vert \cdot-x_{k}\right\vert ^{2}\left\vert -\Delta v_{k}\left(
\cdot\right)  \right\vert v_{k}\left(  \cdot\right)  \right\vert _{y_{k}%
+\hat{r}_{k}z_{k}}\leq\left\vert y_{k}-x_{k}\right\vert ^{2}\left\vert -\Delta
v_{k}\left(  y_{k}\right)  \right\vert v_{k}\left(  y_{k}\right)
\text{.}\label{adds1}%
\end{equation}
By (\ref{65}) and the assumption on $z_{k}$, we have
\begin{equation}
\left\vert \left(  y_{k}+\hat{r}_{k}z_{k}\right)  -x_{k}\right\vert
=\left\vert y_{k}-x_{k}\right\vert \left(  1+o_{k}\left(  1\right)  \right)
\text{, as }k\rightarrow\infty.\label{adds2}%
\end{equation}
Combining (\ref{adds1}) and (\ref{adds2}), we get
\begin{equation}
v_{k}\left(  y_{k}+\hat{r}_{k}z_{k}\right)  \left\vert -\Delta v_{k}\left(
y_{k}+\hat{r}_{k}z_{k}\right)  \right\vert \leq\left(  1+o_{k}\left(
1\right)  \right)  v_{k}\left(  y_{k}\right)  \left\vert -\Delta v_{k}\left(
y_{k}\right)  \right\vert \text{,}\label{73}%
\end{equation}
as $k\rightarrow\infty.$

From (\ref{64}), we have%
\begin{equation}
-\Delta v_{k}\left(  y_{k}\right)  =\frac{4\pi}{E_{k}}\left(  1+o_{k}\left(
1\right)  \right)  v_{k}\left(  y_{k}\right)  e^{v_{k}^{2}\left(
y_{k}\right)  } \label{74}%
\end{equation}
as $k\rightarrow\infty$. Due to (\ref{64}) and (\ref{68}), then%
\begin{equation}
-\Delta v_{k}\left(  y_{k}+\hat{r}_{k}z_{k}\right)  =\frac{4\pi}{E_{k}}\left(
1+o_{k}\left(  1\right)  \right)  v_{k}\left(  y_{k}+\hat{r}_{k}z_{k}\right)
e^{v_{k}^{2}\left(  y_{k}+\hat{r}_{k}z_{k}\right)  } \label{75}%
\end{equation}
as $k\rightarrow\infty$. Using (\ref{68}), (\ref{73}), (\ref{74}) and (\ref{75}), we can
obtain%
\[
\exp\left(  v_{k}^{2}\left(  y_{k}+\hat{r}_{k}z_{k}\right)  \right)
\leq\left(  1+o_{k}\left(  1\right)  \right)  \exp\left(  v_{k}^{2}\left(
y_{k}\right)  \right)  \text{,}%
\]
namely,%
\begin{equation}
\exp\left(  v_{k}^{2}\left(  y_{k}+\hat{r}_{k}z_{k}\right)  -v_{k}^{2}\left(
y_{k}\right)  \right)  \leq\left(  1+o_{k}\left(  1\right)  \right)  \text{.}
\label{76}%
\end{equation}
On the other hand, by (\ref{77}) and (\ref{68}), it holds that%
\[
\exp\left(  v_{k}^{2}\left(  y_{k}+\hat{r}_{k}z_{k}\right)  -v_{k}^{2}\left(
y_{k}\right)  \right)  \geq\exp\left(  2\alpha_{0}\left(  1+o_{k}\left(
1\right)  \right)  \right)  >1\text{,}%
\]
as $k\rightarrow\infty$, which contradicts with (\ref{76}), and the claim
(\ref{67}) is proved. \vskip0.1cm

By (\ref{67}), for any $\varepsilon>0$, there exists some $k_{1}>0$, such that
when $k\geq k_{1}$,
\[
\hat{v}_{k}\left(  x\right)  \leq\varepsilon\text{, for all }x\in\tilde
{\Omega}_{R,k},
\]
namely,
\begin{equation}
v_{k}\left(  y_{k}+\hat{r}_{k}x\right)  \leq\hat{\gamma}_{k}+\frac
{\varepsilon}{\hat{\gamma}_{k}}. \label{adds3}%
\end{equation}
Then there exists some $C>0$ such that%
\begin{align}
-\Delta\hat{v}_{k}  &  =4\hat{\gamma}_{k}^{-1}e^{-\hat{\gamma}_{k}^{2}}\left(
v_{k}\left(  y_{k}+\hat{r}_{k}x\right)  e^{v_{k}^{2}\left(  y_{k}+\hat{r}%
_{k}x\right)  }-\frac{p}{2\left(  4\pi\right)  ^{\frac{p}{2}}}\lambda_{k}%
v_{k}^{p-1}\left(  y_{k}+\hat{r}_{k}x\right)  \right) \nonumber\\
&  \leq4\frac{v_{k}\left(  y_{k}+\hat{r}_{k}x\right)  e^{v_{k}^{2}\left(
y_{k}+\hat{r}_{k}x\right)  }}{\hat{\gamma}_{k}e^{\hat{\gamma}_{k}^{2}}}\leq
C\text{,} \label{899}%
\end{align}
and%
\begin{equation}
-\Delta\hat{v}_{k}\geq-4\hat{\gamma}_{k}^{-1}e^{-\hat{\gamma}_{k}^{2}}\frac
{p}{2\left(  4\pi\right)  ^{\frac{p}{2}}}\lambda_{k}v_{k}^{p-1}\left(
y_{k}+\hat{r}_{k}x\right)  \rightarrow0\text{,} \label{adds4}%
\end{equation}
as $k\rightarrow\infty$ by (\ref{64}) and (\ref{adds3}). Hence, combining
(\ref{899}) and (\ref{adds4}), we have
\begin{equation}
\left\vert \left\vert -\Delta\hat{v}_{k}\right\vert \right\vert _{L^{\infty
}\left(  \tilde{\Omega}_{R,k}\right)  }\leq C_{1}\text{,} \label{199}%
\end{equation}
where\ $C_{1}$ is independent on $k$.

Since $\hat{r}_{k}\rightarrow0$ as $k\rightarrow\infty$, either $\tilde
{\Omega}_{R,k}\rightarrow\mathbb{R}^{2}$ as $k\rightarrow\infty$ or
$\tilde{\Omega}_{R,k}\rightarrow\left(  -\infty,R_{0}\right)  \times
\mathbb{R}$ as $k\rightarrow\infty$ for some $R_{0}>0$.

Now, we show the second case will not occur by contradiction. Assume that
$\tilde{\Omega}_{R,k}\rightarrow\left(  -\infty,R_{0}\right)  \times
\mathbb{R}$ as $k\rightarrow\infty$ for some $R_{0}>0,$ set
\[
\hat{v}_{k}=\varphi_{k}+\psi_{k}\text{ in }\tilde{\Omega}_{R,k}\cap
B_{2R_{0}+1}\left(  0\right)  \text{,}%
\]
with%
\[
-\Delta\varphi_{k}=-\Delta\hat{v}_{k}\text{ in }\tilde{\Omega}_{R,k}\cap
B_{2R_{0}+1}\left(  0\right)  \text{\ and\ }\psi_{k}=\hat{v}_{k}\text{ on
}\partial\left(  \tilde{\Omega}_{R,k}\cap B_{2R_{0}+1}\left(  0\right)
\right)  \text{.}%
\]
According to (\ref{199}), since $\varphi_{k}=0$ on $\partial\left(
\tilde{\Omega}_{R,k}\cap B_{2R_{0}+1}\left(  0\right)  \right)  $, by standard
elliptic theory that $\left(  \varphi_{k}\right)  $ is uniformly bounded in
$\tilde{\Omega}_{R,k}\cap B_{2R_{0}+1}\left(  0\right)  $. Then $\psi_{k}$ is
some harmonic function in $\tilde{\Omega}_{R,k}\cap B_{2R_{0}+1}\left(
0\right)  $ which is bounded from above thanks to (\ref{67}) and which
verifies that $\psi_{k}=-\hat{\gamma}_{k}^{2}\rightarrow-\infty$ on
$\partial\tilde{\Omega}_{R,k}\cap B_{2R_{0}+1}\left(  0\right)  $. Since
$\partial\tilde{\Omega}_{R,k}\cap B_{2R_{0}+1}\left(  0\right)  \rightarrow
\left(  \{R_{0}\}\times R\right)  \cap B_{2R_{0}+1}\left(  0\right)  $ as
$k\rightarrow\infty$, one easily gets that%
\[
\psi_{k}\left(  0\right)  \rightarrow-\infty\text{ \ as }k\rightarrow
\infty\text{.}%
\]
This is a contradiction, since $\hat{v}_{k}\left(  0\right)  =0$, $\psi
_{k}\left(  0\right)  =-\varphi_{k}\left(  0\right)  $ and $\left(
\varphi_{k}\left(  0\right)  \right)  $ is bounded. Thus, we proved%
\[
\tilde{\Omega}_{R,k}\rightarrow\mathbb{R}^{2}\text{.}%
\]
Moreover, using $\hat{v}_{k}\left(  0\right)  =0$, (\ref{67}), (\ref{199}) and
the standard elliptic theory, we get%
\begin{equation}
\hat{v}_{k}\rightarrow\phi_{\infty}\text{ in }C_{loc}^{1}\left(
\mathbb{R}^{2}\right)  \text{,}\label{gives}%
\end{equation}
as $k\rightarrow\infty$, where $\phi_{\infty}$ is defined in (\ref{214}). \vskip0.1cm

Denote $f_{k}=\frac{4\pi}{E_{k}}(v_{k}e^{v_{k}^{2}}-\frac{p}{2(4\pi)^{\frac
{p}{2}}}\lambda_{k}v_{k}^{p-1})$, then by (\ref{eq of vn}), we have%
\begin{equation}
\int_{\Omega}\left\vert \nabla v_{k}\right\vert ^{2}dx=\int_{\Omega}f_{k}%
v_{k}dx=4\pi\text{.} \label{78}%
\end{equation}
It follows from (\ref{65}) and (\ref{above}) that $B_{Rr_{k}}\left(
x_{k}\right)  \cap B_{R\hat{r}_{k}}\left(  y_{k}\right)  =\emptyset$ as
$k\rightarrow\infty$, then%
\begin{align*}
\underset{k\rightarrow\infty}{\lim}\int_{\Omega}f_{k}v_{k}dx  &
=\underset{k\rightarrow\infty}{\lim}\int_{B_{Rr_{k}}\left(  x_{k}\right)
}f_{k}v_{k}dx+\underset{k\rightarrow\infty}{\lim}\int_{B_{R\hat{r}_{k}}\left(
y_{k}\right)  }f_{k}v_{k}dx\\
&  +\underset{k\rightarrow\infty}{\lim}\int_{\Omega\backslash\{B_{Rr_{k}%
}\left(  x_{k}\right)  \cup B_{R\hat{r}_{k}}\left(  y_{k}\right)  \}}%
f_{k}v_{k}dx\\
&  :=L_{1}+L_{2}+L_{3}\text{.}%
\end{align*}
An direct calculation gives that
\begin{align}
L_{1}  &  =\underset{k\rightarrow\infty}{\lim}\int_{B_{Rr_{k}}\left(
x_{k}\right)  }f_{k}v_{k}dx\nonumber\\
&  =\underset{k\rightarrow\infty}{\lim}\int_{B_{R}\left(  0\right)  }4\left[
\left(  \frac{v_{k}\left(  x_{k}+r_{k}x\right)  }{\gamma_{k}}\right)
^{2}e^{\left(  1+\frac{v_{k}\left(  x_{k}+r_{k}x\right)  }{\gamma_{k}}\right)
\phi_{k}}\right. \nonumber\\
&  \left.  -\frac{p}{2(4\pi)^{\frac{p}{2}}}\lambda_{k}\gamma_{k}%
^{p-2}e^{-\gamma_{k}^{2}}\left\vert \frac{v_{k}\left(  x_{k}+r_{k}x\right)
}{\gamma_{k}}\right\vert ^{p}\right]  dx\nonumber\\
&  =4\int_{B_{R}\left(  0\right)  }e^{2\phi_{\infty}}dx\rightarrow
4\int_{\mathbb{R}^{2}}e^{2\phi_{\infty}}dx=4\pi\text{,} \label{79}%
\end{align}
as $R\rightarrow\infty$, where $\phi_{k}\left(  x\right)  $ satisfies equation
(\ref{18}). Similarly, using (\ref{64}), (\ref{gives}) and (\ref{67}) we have%
\begin{equation}
L_{2}=\underset{k\rightarrow\infty}{\lim}\int_{B_{R\hat{r}_{k}}\left(
y_{k}\right)  }f_{k}v_{k}dx=4\pi. \label{80}%
\end{equation}
And for $L_{3}$, we have%
\begin{align}
L_{3}  &  =\underset{k\rightarrow\infty}{\lim}\int_{\Omega\backslash
\{B_{Rr_{k}}\left(  x_{k}\right)  \cup B_{R\hat{r}_{k}}\left(  y_{k}\right)
\}}f_{k}v_{k}dx\nonumber\\
&  =\underset{k\rightarrow\infty}{\lim}\frac{4\pi}{E_{k}}\int_{\Omega
\backslash\{B_{Rr_{k}}\left(  x_{k}\right)  \cup B_{R\hat{r}_{k}}\left(
y_{k}\right)  \}}\left(  v_{k}^{2}e^{v_{k}^{2}}-\frac{p}{2(4\pi)^{\frac{p}{2}%
}}\lambda_{k}v_{k}^{p}\right)  dx\nonumber\\
&  \geq\underset{k\rightarrow\infty}{\lim}\frac{4\pi}{E_{k}}\int_{\Omega
}-\frac{p}{2(4\pi)^{\frac{p}{2}}}\lambda_{k}v_{k}^{p}dx=0\text{,} \label{81}%
\end{align}
where the last inequality holds due to (\ref{238}). Combining (\ref{79}),
(\ref{80}) and (\ref{81}), we get%
\[
\underset{k\rightarrow\infty}{\lim}\int_{\Omega}f_{k}v_{k}dx\geq8\pi\text{,}%
\]
which contradicts to with (\ref{78}), and the proof is finished.
\end{proof}
In the following, we use the weak point-wise estimates of $\Delta v_{k}$ from Lemma \ref{52} to derive the point-wise estimates for the gradient of $v_{k}$. This allows comparison of $v_{k}$ with its mean value on spheres and is crucial for proving the comparison principle of the equation in (\ref{eq of vn}).

\begin{lemma}
\label{85} There exists $C>0$ such that%
\[
\left\vert x-x_{k}\right\vert \left\vert \nabla v_{k}\right\vert v_{k}\leq
C\text{ in }\tilde{\Omega}_k\text{,}%
\]
where $\tilde{\Omega}_k=\Omega$ if $1\leq p<2$, and $\tilde{\Omega}_k=\{x\in\Omega \mid v_k(x)\geq \log(\gamma_k)\}$ if $p=2$.
\end{lemma}

\begin{proof}
We only prove this by contradiction for the case \(1 \leq p < 2\), the proof for the case \(p = 2\) is similar. We assume that $z_{k}\in\Omega$ such that%
\begin{equation}
\underset{x\in\Omega}{\max}\left\vert x-x_{k}\right\vert \left\vert \nabla
v_{k}(x)\right\vert v_{k}(x)=\left\vert z_{k}-x_{k}\right\vert \left\vert
\nabla v_{k}(z_{k})\right\vert v_{k}(z_{k}):=C_{k}\rightarrow\infty\label{max}%
\end{equation}
as $k\rightarrow\infty$. Let $\mu_{k}>0$ be given by%
\[
\mu_{k}=\left\vert z_{k}-x_{k}\right\vert \text{.}%
\]
By (\ref{2271}) we immediate get $\mu_{k}\rightarrow0$, as $k\rightarrow
\infty$.

For simplicity, and with a slight abuse of notation, we set $\Omega
_{R,k}=B_{R\mu_{k}}\left(  z_{k}\right)  \cap\Omega$, $\tilde{\Omega}%
_{R,k}=\left(  \Omega_{R,k}-z_{k}\right)  /\mu_{k}$, and
\[
\omega_{k}(x)=v_{k}(z_{k}+\mu_{k}x)\text{ for }x\in\tilde{\Omega}%
_{R,k}\text{.}%
\]
It is clearly that%
\[
\tilde{x}_{k}:=\frac{x_{k}-z_{k}}{\mu_{k}}\in\tilde{\Omega}_{R,k}\text{,}%
\]
with
\[
\left\vert \tilde{x}_{k}\right\vert =1\text{ for all }k>0\text{.}%
\]
Hence, (\ref{max}) can be rewritten as
\begin{equation}
\left\vert x-\tilde{x}_{k}\right\vert \omega_{k}(x)\left\vert \nabla\omega
_{k}(x)\right\vert \leq\omega_{k}(0)\left\vert \nabla\omega_{k}(0)\right\vert
\rightarrow\infty\text{ as }k\rightarrow\infty\text{.} \label{54}%
\end{equation}
for all $x\in\tilde{\Omega}_{R,k}\backslash\{\tilde{x}_{k}\}$.

Now, we claim that for any $R>1,$
\begin{equation}
\left\vert \left\vert -\Delta\omega_{k}\right\vert \right\vert _{L^{\infty
}\left(  \tilde{\Omega}_{R,k}\backslash B_{R^{-1}}\left(  \tilde{x}%
_{k}\right)  \right)  }\rightarrow0,\text{as }k\rightarrow\infty. \label{200}%
\end{equation}
For any sequence $\left(  y_{k}\right)  _{k}$ such that $y_{k}\in\Omega
_{R,k}\backslash\{x_{k}\}$, namely, $\tilde{y}_{k}=\frac{y_{k}-z_{k}}{\mu_{k}%
}\in\tilde{\Omega}_{R,k}\backslash\{\tilde{x}_{k}\}$. If $v_{k}\left(
y_{k}\right)  $ is bounded, due to (\ref{227}) and Lemma \ref{le4.5}, we have
\begin{align}
\left\vert -\Delta\omega_{k}\left(  \tilde{y}_{k}\right)  \right\vert  &
=\mu_{k}^{2}\left\vert -\Delta v_{k}\left(  y_{k}\right)  \right\vert
\nonumber\\
&  =\mu_{k}^{2}\left\vert \frac{4\pi}{E_{k}}\left(  v_{k}\left(  y_{k}\right)
e^{v_{k}^{2}\left(  y_{k}\right)  }-\frac{p}{2(4\pi)^{\frac{p}{2}}}\lambda
_{k}v_{k}^{p-1}\left(  y_{k}\right)  \right)  \right\vert \rightarrow0
\label{72}%
\end{align}
as $k\rightarrow\infty$. While if $v_{k}\left(  y_{k}\right)  \rightarrow
\infty$ as $k\rightarrow\infty$, by Lemma \ref{52}, for $y_{k}\in
\tilde{\Omega}_{R,k}\backslash B_{R^{-1}}\left(  \tilde{x}_{k}\right)  $, we
have
\begin{equation}
\left\vert -\Delta\omega_{k}\left(  \tilde{y}_{k}\right)  \right\vert =\mu
_{k}^{2}\left\vert -\Delta v_{k}\left(  y_{k}\right)  \right\vert \leq\frac
{C}{v_{k}\left(  y_{k}\right)  \left\vert \tilde{y}_{k}-\tilde{x}%
_{k}\right\vert ^{2}}\rightarrow0, \label{71}%
\end{equation}
as $k\rightarrow\infty$. Combining (\ref{71}) and (\ref{72}), the claim
(\ref{200}) is proved.

Since $v_{k}$ is bounded in $H_{0}^{1}(\Omega)$, we immediately have%
\begin{equation}
\left\vert \left\vert \nabla\omega_{k}\right\vert \right\vert _{L^{2}\left(
\tilde{\Omega}_{R,k}\right)  }=O_{k}\left(  1\right)  . \label{gradient}%
\end{equation}

Next, we show that up to a subsequence%
\begin{equation}
\text{ }\frac{d\left(  z_{k},\partial\Omega\right)  }{\mu_{k}}\rightarrow
\infty\text{ as }k\rightarrow\infty. \label{49}%
\end{equation}
We prove this by contradiction, and assume that (\ref{49}) does not hold, then
for all $R$ sufficiently large, from (\ref{54}), (\ref{200}) and the fact that
$\omega_{k}=0$ on some point of $\tilde{\Omega}_{R,k}$, we can get
\[
\frac{\omega_{k}}{\sqrt{\omega_{k}(0)\left\vert \nabla\omega_{k}(0)\right\vert
}}\rightarrow0\text{ in }C_{loc}^{1}\left(  \tilde{\Omega}_{R,k}\backslash
B_{R^{-1}}\left(  \tilde{x}_{k}\right)  \right)  \text{ as }k\rightarrow
\infty\text{,}%
\]
by using the standard elliptic regularity theory, but this contradicts to%
\[
\frac{\omega_{k}\left(  0\right)  }{\sqrt{\omega_{k}(0)\left\vert \nabla
\omega_{k}(0)\right\vert }}\left\vert \nabla\left(  \frac{\omega_{k}}%
{\sqrt{\omega_{k}(0)\left\vert \nabla\omega_{k}(0)\right\vert }}\right)
\left(  0\right)  \right\vert =1\text{,}%
\]
thus (\ref{49}) is proved. We remark that this also implies that $\mu
_{k}\rightarrow0$ as $k\rightarrow\infty$.

Up to a subsequence, we may assume
\[
\tilde{x}_{k}\rightarrow\tilde{x}\text{ as }k\rightarrow\infty\text{.}%
\]
Now, we claim that
\begin{equation}
\omega_{k}\left(  0\right)  =v_{k}\left(  z_{k}\right)  \rightarrow
\infty\text{ as }k\rightarrow\infty\text{.} \label{201}%
\end{equation}

We prove this claim\ by contradiction, and assume that $\omega_{k}\left(
0\right)  =O_{k}\left(  1\right)  ,$ as $k\rightarrow\infty.$ \ Fix $R>2$, and
define $K_{R}=B_{R}\left(  0\right)  \backslash B_{R^{-1}}\left(  \tilde
{x}\right)  $, due to (\ref{200}), we have
\begin{equation}
\left\vert \left\vert -\Delta\omega_{k}\right\vert \right\vert _{L^{\infty
}\left(  K_{R}\right)  }\rightarrow0\text{ as }k\rightarrow\infty
\text{.}\label{202}%
\end{equation}
Since $\omega_{k}\geq0$, and $\omega_{k}\left(  0\right)  =O_{k}\left(
1\right)  $, then by the elliptic regularity theory we know that $\left(
\omega_{k}\right)  $ is uniformly bounded in $C^{1}\left(  K_{R}\right)  $,
which contradicts to (\ref{54}), thus the claim (\ref{201}) is proved.

By (\ref{gradient}), (\ref{49}) (\ref{201}), (\ref{202}) and the fact that
$\omega_{k}\geq0$, using the elliptic regularity theory again we get%
\begin{equation}
\frac{\omega_{k}}{\omega_{k}\left(  0\right)  }\rightarrow1\text{ in }%
C_{loc}^{1}\left(  \mathbb{R}^{2}\backslash\{\tilde{x}\}\right)  ,\text{ as
}k\rightarrow\infty\text{.} \label{rato}%
\end{equation}

For any $x\in\tilde{\Omega}_{R,k}\backslash B_{R^{-1}}\left(  \tilde
{x}\right)  $, we set%
\[
\tilde{\omega}_{k}\left(  x\right)  =\frac{\omega_{k}\left(  x\right)
-\omega_{k}\left(  0\right)  }{\left\vert \nabla\omega_{k}\left(  0\right)
\right\vert }.
\]
Inequality (\ref{54}) together with (\ref{rato}) gives that%
\begin{equation}
\left\vert \nabla\tilde{\omega}_{k}\left(  x\right)  \right\vert \leq
\frac{1+o_{k}\left(  1\right)  }{\left\vert x-\tilde{x}_{k}\right\vert
}\text{.} \label{change gradient}%
\end{equation}
Given $R>1$ and any sequence $\left(  y_{k}\right)  _{k}$ such that $\tilde
{y}_{k}:=\frac{y_{k}-z_{k}}{\mu_{k}}\in\tilde{\Omega}_{R,k}\backslash
B_{R^{-1}}\left(  \tilde{x}\right)  $, from (\ref{rato}) and (\ref{71}), we
get
\[
\left\vert -\Delta\tilde{\omega}_{k}\left(  \tilde{y}_{k}\right)  \right\vert
=\frac{\left\vert -\Delta\omega_{k}\left(  \tilde{y}_{k}\right)  \right\vert
}{\left\vert \nabla\omega_{k}\left(  0\right)  \right\vert }=\frac{\omega
_{k}\left(  0\right)  }{C_{k}}\left\vert -\Delta\omega_{k}\left(  \tilde
{y}_{k}\right)  \right\vert \leq\frac{C\left(  1+o_{k}\left(  1\right)
\right)  }{C_{k}\left\vert \tilde{y}_{k}-\tilde{x}_{k}\right\vert ^{2}%
}\text{,}%
\]
as $k\rightarrow\infty$, where $C_{k}$ is defined in (\ref{max}). Hence,we
have
\begin{equation}
\left\vert \left\vert -\Delta\tilde{\omega}_{k}\right\vert \right\vert
_{L^{\infty}\left(  \tilde{\Omega}_{R,k}\backslash B_{R^{-1}}\left(  \tilde
{x}\right)  \right)  }\rightarrow0\text{ as }k\text{ }\rightarrow
\infty\text{.} \label{AB}%
\end{equation}
From (\ref{change gradient}), (\ref{AB}) and $\tilde{\omega}_{k}\left(
0\right)  =0$, using the standard elliptic theory with the fact (\ref{49}), we
have
\begin{equation}
\tilde{\omega}_{k}\rightarrow H\text{ in }C_{loc}^{1}\left(  \mathbb{R}%
^{2}\backslash\{\tilde{x}\}\right)  \text{ as }k\rightarrow\infty\text{,}
\label{203}%
\end{equation}
where $H\in C^{\infty}\left(  \mathbb{R}^{2}\backslash\{\tilde{x}\}\right)  $
satisfies%
\[
-\Delta H=0\text{ in }\mathbb{R}^{2}\backslash\{\tilde{x}\}\text{,\ }H\left(
0\right)  =0\text{,\ }\left\vert \nabla H\left(  0\right)  \right\vert
=1\text{,}%
\]
with%
\begin{equation}
\left\vert \nabla H\left(  x\right)  \right\vert \leq\frac{1}{\left\vert
x-\tilde{x}\right\vert }\text{ for all }x\in\mathbb{R}^{2}\backslash
\{\tilde{x}\}. \label{2041}%
\end{equation}

Let $\eta>0$ be small enough, we have
\begin{align*}
\left\vert \int_{B_{\eta}\left(  \tilde{x}\right)  }\omega_{k}\left(
-\Delta\omega_{k}\right)  dx\right\vert  &  =\left\vert \int_{B_{\mu_{k}\eta
}\left(  z_{k}+\mu_{k}\tilde{x}\right)  }v_{k}\left(  -\Delta v_{k}\right)
dx\right\vert \nonumber\\
&  \leq\int_{B_{\mu_{k}\eta}\left(  z_{k}+\mu_{k}\tilde{x}\right)  }%
v_{k}\left(  -\Delta v_{k}\right)  ^{+}dx+\int_{B_{\mu_{k}\eta}\left(
z_{k}+\mu_{k}\tilde{x}\right)  }v_{k}\left(  -\Delta v_{k}\right)
^{-}dx\text{.} %
\end{align*}
Since $\left(  -\Delta v_{k}\right)  ^{-}=O_{k}\left(  \frac{\lambda_{k}%
v_{k}^{p-1}}{E_{k}}\right)  $ by (\ref{eq of vn}), then from (\ref{238}), we
have%
\[
\int_{B_{\mu_{k}\eta}\left(  z_{k}+\mu_{k}\tilde{x}\right)  }v_{k}\left(
-\Delta v_{k}\right)  ^{-}dx\leq C\int_{\Omega}\frac{\lambda_{k}v_{k}^{p}%
}{E_{k}}dx\rightarrow0
\]
as $k\rightarrow\infty$. Hence,%
\begin{align}
\left\vert \int_{B_{\eta}\left(  \tilde{x}\right)  }\omega_{k}\left(
-\Delta\omega_{k}\right)  dx\right\vert  &  \leq\int_{B_{\mu_{k}\eta}\left(
y_{k}+\mu_{k}\tilde{x}\right)  }v_{k}\left(  -\Delta v_{k}\right)
^{+}dx+o_{k}\left(  1\right) \nonumber\\
&  \leq\int_{\Omega}v_{k}\left(  -\Delta v_{k}\right)  ^{+}dx+o_{k}\left(
1\right) \nonumber\\
&  =\int_{\Omega}v_{k}\left(  -\Delta v_{k}\right)  dx+\int_{\Omega}%
v_{k}\left(  -\Delta v_{k}\right)  ^{-}dx+o_{k}\left(  1\right) \nonumber\\
&  =4\pi+o_{k}\left(  1\right)  \text{,} \label{209}%
\end{align}
where in the last step we have used the fact that $\int_{\Omega}v_{k}\left(
-\Delta v_{k}\right)  dx=4\pi$.

From (\ref{gradient}) and (\ref{209}), using the Green's formula, we get
\begin{equation}
\int_{\partial B_{\eta}\left(  \tilde{x}\right)  }\omega_{k}\partial_{\nu
}\omega_{k}d\sigma=\int_{B_{\eta}\left(  \tilde{x}\right)  }\left\vert
\nabla\omega_{k}\right\vert ^{2}dx-\int_{B_{\eta}\left(  \tilde{x}\right)
}\omega_{k}\left(  -\Delta\omega_{k}\right)  dx=O_{k}\left(  1\right)  .
\label{addss1}%
\end{equation}
On the other hand, using (\ref{rato}) and (\ref{203}), we can get%
\[
\int_{\partial B_{\eta}\left(  \tilde{x}\right)  }\omega_{k}\partial_{\nu
}\omega_{k}d\sigma=\omega_{k}\left(  0\right)  \left\vert \nabla\omega
_{k}\left(  0\right)  \right\vert \left(  \int_{\partial B_{\eta}\left(
\tilde{x}\right)  }\partial_{\nu}Hd\sigma+o_{k}\left(  1\right)  \right)  .
\]
This together with (\ref{54}) and (\ref{addss1}) yields%
\[
\int_{\partial B_{\eta}\left(  \tilde{x}\right)  }\partial_{\nu}%
Hd\sigma=0\text{.}%
\]
Therefore, $H$ must be bounded around $\tilde{x}$ and then the singularity at
$\tilde{x}$ is removable. Hence, we can deduce
\[
\text{ }-\Delta H=0\text{ in }\mathbb{R}^{2}\text{,}%
\]
and then
\begin{equation}
\int_{\partial B_{R}\left(  0\right)  }Hd\sigma=0\text{ for all }R>0\text{,}
\label{207}%
\end{equation}
due to $H\left(  0\right)  =0$.

Let $R>2$ and $x\in\partial B_{R}\left(  0\right)  $, by (\ref{2041}), we have%
\begin{equation}
\underset{\partial B_{R}\left(  0\right)  }{\sup}\left\vert \nabla
H\right\vert \leq\frac{2}{R}\text{.} \label{208}%
\end{equation}
Due to (\ref{207}), there exists some $y\in\partial B_{R}\left(  0\right)  $
such that $H\left(  y\right)  =0$, and by (\ref{208}), we obtain
\[
\left\vert H\left(  x\right)  \right\vert \leq\left(  \underset{\partial
B_{R}\left(  0\right)  }{\max}\left\vert \nabla H\right\vert \right)  \pi
R\leq2\pi\text{,}%
\]
which means $H\in L^{\infty}\left(  \mathbb{R}^{2}\right)  $. By the
Liouville's theorem and the fact $H\left(  0\right)  =0$, we obtain that
$H\equiv0$. This is clearly a contradiction with the fact that $\left\vert
\nabla H\left(  0\right)  \right\vert =1$, hence (\ref{max}) does not hold,
and the proof is finished.
\end{proof}

\medskip

\subsection{ Blow-up behavior up to higher order precision for
radially symmetric solutions}

In this subsection, we present radially symmetric solutions, defined on the entire space  $\mathbb{R}^2$ and centered at $x_{k}$, and analyze their asymptotic behavior within a small ball around $x_{k}$.

  Let $V_{k}$ be the radially symmetric solution centered at $x_{k}$ of%
\begin{equation}
\left\{
\begin{array}
[c]{l}%
-\Delta V_{k}=\frac{{{4\pi}}}{{{E_{k}}}}\left(  V_{k}{{e^{V_{k}^{2}}}-}%
\frac{p}{2\left(  4\pi\right)  ^{\frac{p}{2}}}\lambda_{k}\left\vert V{_{k}%
}\right\vert ^{p-1}\right)  ,\\
V_{k}\left(  x_{k}\right)  ={\gamma}_{k}.
\end{array}
\right.  \label{3.10}%
\end{equation}
and $t_{k}$ be given by%
\[
t_{k}\left(  y\right)  =-\phi_{\infty}\left(  \frac{y-x_{k}}{r_{k}}\right)
=\log\left(  1+\frac{\left\vert y-x_{k}\right\vert ^{2}}{r_{k}^{2}}\right)
\text{,}%
\]
where $\phi_{\infty}$ is defined in (\ref{3.6}). We will write $V_{k}\left(
r\right)  $ or $t_{k}\left(  r\right)  $ instead of $V_{k}\left(  y\right)  $
or $t_{k}\left(  y\right)  $ for $\left\vert y-x_{k}\right\vert =r,$
respectively. For any $\delta\in\left(  0,1\right)  $, we let $r_{k,\delta}>0$
be given as the follows%
\begin{equation}
t_{k}\left(  r_{k,\delta}\right)  =\delta\gamma_{k}^{2}\text{.} \label{55}%
\end{equation}
Observe that (\ref{55}) implies%
\begin{equation}
r_{k,\delta}^{2}=r_{k}^{2}\exp\left(  \delta\gamma_{k}^{2}+o_{k}\left(
1\right)  \right)  \text{,} \label{212}%
\end{equation}
and
\[
\frac{r_{k,\delta}}{r_{k}}\rightarrow\infty,\text{ as }k\rightarrow\infty.
\]

Now, we study the blow-up behavior of radially symmetric solutions
$\left(  V_{k}\right)  $ up to higher order precision in the ball
$B_{r_{k,\delta}}\left(  x_{k}\right)  .$ Let $S_{k}$ be given by%
\[
S_{k}\left(  z\right)  =S_{0}\left(  \frac{z-x_{k}}{r_{k}}\right)  \text{.}%
\]
where $S_{0}$ is the radial solution around $0\in\mathbb{R}^{2}$ of%
\begin{equation}
-\Delta S_{0}-8\exp\left(  2\phi_{\infty}\right)  S_{0}=4\exp\left(
2\phi_{\infty}\right)  \left(  \phi_{\infty}^{2}+\phi_{\infty}\right)  ,\text{
in }\mathbb{R}^{2}\label{57}%
\end{equation}
with $S_{0}\left(  0\right)  =0$.

From \cite{Malchiodi}, we know that $S_{0}$ has the following explicit
formula:
\[
S_{0}\left(  r\right)  =\phi_{\infty}\left(  r\right)  +\frac{2r^{2}}{1+r^{2}%
}-\frac{1}{2}\phi_{\infty}^{2}\left(  r\right)  +\frac{1-r^{2}}{1+r^{2}}%
\int_{1}^{1+r^{2}}\frac{\log t}{1-t}dt\text{.}%
\]
In particular
\begin{equation}\label{s0}
S_{0}\left(  r\right)  =\frac{A_{0}}{4\pi}\log\frac{1}{r^{2}}+B_{0}+O\left(
\log\left(  r^{2}\right)  r^{-2}\right)  \text{ as }r\rightarrow\infty\text{,}%
\end{equation}
where
\begin{equation}
A_{0}=\int_{\mathbb{R}^{2}}\left(  -\Delta S_{0}\right)  dy=4\pi\text{ and
}B_{0}=\frac{\pi^{2}}{6}+2. \label{56}%
\end{equation}

\begin{lemma}
\label{15}It holds%
\begin{equation}
V_{k}\left(  z\right)  =\gamma_{k}-\frac{t_{k}\left(  z\right)  }{\gamma_{k}%
}+\frac{S_{k}\left(  z\right)  }{\gamma_{k}^{3}}+O\left(  \frac{1+t_{k}\left(
z\right)  }{\gamma_{k}^{5}}\right)  \text{,} \label{510}%
\end{equation}
for any $z\in B_{r_{k,\delta}}\left(  x_{k}\right)  $
\end{lemma}

\begin{proof}
Let $\xi_{k}$ be given by
\begin{equation}
V_{k}=\gamma_{k}-\frac{t_{k}}{\gamma_{k}}+\frac{\xi_{k}}{\gamma_{k}^{3}%
}\text{,} \label{B}%
\end{equation}
and $\rho_{k}$ be defined as%
\begin{equation}
\rho_{k}=\sup\left\{  r\in(0,r_{k,\delta}]:\left\vert S_{k}-\xi_{k}\right\vert
\leq1+t_{k}\text{ in }[0,r]\right\}  \text{.} \label{P}%
\end{equation}
Notice that (\ref{s0}) and (\ref{P}) imply $\xi_{k}=O\left(  1+t_{k}\right)  $
in $B_{\rho_{k}}\left(  x_{k}\right)  $. In particular, from (\ref{B}) we get%
\begin{equation}
V_{k}=\gamma_{k}-\frac{t_{k}}{\gamma_{k}}+\frac{O\left(  1+t_{k}\right)
}{\gamma_{k}^{3}}\text{,} \label{extend}%
\end{equation}
and
\begin{equation}
V_{k}^{2}=\gamma_{k}^{2}-2t_{k}+\frac{t_{k}^{2}+2\xi_{k}}{\gamma_{k}^{2}%
}+O\left(  \frac{1+t_{k}^{2}}{\gamma_{k}^{4}}\right)  \label{squel}%
\end{equation}
in $B_{\rho_{k}}\left(  x_{k}\right)  $. Then, we have%
\begin{equation}
\exp\left(  \frac{t_{k}^{2}+2\xi_{k}}{\gamma_{k}^{2}}+O\left(  \frac
{1+t_{k}^{2}}{\gamma_{k}^{4}}\right)  \right)  =1+\frac{t_{k}^{2}+2\xi_{k}%
}{\gamma_{k}^{2}}+O\left(  \frac{\left(  1+t_{k}^{4}\right)  \exp\left(
t_{k}^{2}/\gamma_{k}^{2}\right)  }{\gamma_{k}^{4}}\right)  \label{exp1}%
\end{equation}
in $B_{\rho_{k}}\left(  x_{k}\right)  $. By using (\ref{3.5}), (\ref{extend}),
(\ref{squel}) and (\ref{exp1}), we have%
\begin{equation}
\frac{4\pi}{E_{k}}V_{k}\exp\left(  V_{k}^{2}\right)  =\frac{4\exp\left(
-2t_{k}\right)  }{r_{k}^{2}\gamma_{k}}\left[  1+\frac{2\xi_{k}+t_{k}^{2}%
-t_{k}}{\gamma_{k}^{2}}+O\left(  \frac{\left(  1+t_{k}^{4}\right)  \exp\left(
t_{k}^{2}/\gamma_{k}^{2}\right)  }{\gamma_{k}^{4}}\right)  \right]
\label{main}%
\end{equation}
in $B_{\rho_{k}}\left(  x_{k}\right)  $. Now, we claim that
\begin{equation}
\frac{4\pi}{E_{k}}\frac{p}{2\left(  4\pi\right)  ^{\frac{p}{2}}}\lambda
_{k}V_{k}^{p-1}=o\left(  \frac{\exp\left(  -2t_{k}\right)  \exp\left(
t_{k}^{2}/\gamma_{k}^{2}\right)  }{\gamma_{k}^{5}r_{k}^{2}}\right)  \text{.}
\label{o}%
\end{equation}
Indeed,
\[
\frac{\exp\left(  t_{k}\left(  -2+t_{k}/\gamma_{k}^{2}\right)  \right)
}{r_{k}^{2}}=\exp\left(  \gamma_{k}-\frac{t_{k}}{\gamma_{k}}\right)  ^{2}%
\frac{\pi}{E_{k}}\gamma_{k}^{2}\text{.}%
\]
Since $t_{k}\left(  z\right)  \leq\delta\gamma_{k}^{2}$ for any $z\in
B_{r_{k,\delta}}\left(  x_{k}\right)  $, by (\ref{227}) we have
\[
\frac{\frac{p}{2\left(  4\pi\right)  ^{\frac{p}{2}}}\lambda_{k}\left(
\gamma_{k}-\frac{t_{k}}{\gamma_{k}}+O\left(  \frac{1+t_{k}}{\gamma_{k}^{3}%
}\right)  \right)  ^{p-1}\frac{4\pi}{E_{k}}\gamma_{k}^{5}}{\exp\left(
\gamma_{k}-\frac{t_{k}}{\gamma_{k}}\right)  ^{2}\frac{\pi}{E_{k}}\gamma
_{k}^{2}}\leq o\left(  \frac{\gamma_{k}^{2p+2}}{\exp\left(  \left(
1-\delta\right)  ^{2}\gamma_{k}^{2}\right)  }\right)  \rightarrow0,
\]
as $k\rightarrow\infty$. Thus, the claim (\ref{o}) is proved.

Combining (\ref{o}), (\ref{3.10}), (\ref{B}) and (\ref{main}), we get%
\begin{equation}
-\Delta\xi_{k}=\frac{4\exp\left(  -2t_{k}\right)  }{r_{k}^{2}}\left(  2\xi
_{k}+t_{k}^{2}-t_{k}+O\left(  \frac{\left(  1+t_{k}^{4}\right)  \exp\left(
t_{k}^{2}/\gamma_{k}^{2}\right)  }{\gamma_{k}^{2}}\right)  \right)
\label{get}%
\end{equation}
in $B_{\rho_{k}}\left(  x_{k}\right)  $.

Next, we estimate the function $\left\vert \xi_{k}-S_{k}\right\vert $. From
(\ref{57}) and (\ref{get}), we have%
\begin{equation}
-\Delta\left(  \xi_{k}-S_{k}\right)  =\frac{8\exp\left(  -2t_{k}\right)
}{r_{k}^{2}}\left[  \left(  \xi_{k}-S_{k}\right)  +O\left(  \frac{\left(
1+t_{k}^{4}\right)  \exp\left(  t_{k}^{2}/\gamma_{k}^{2}\right)  }{\gamma
_{k}^{2}}\right)  \right]  \label{505}%
\end{equation}
for all $0\leq r\leq$ $\rho_{k}$. Observe that
\begin{equation}
\int_{B_{r}\left(  x_{k}\right)  }\left(  -\Delta\left(  \xi_{k}-S_{k}\right)
\right)  dy=-2\pi r\left(  \xi_{k}-S_{k}\right)  ^{\prime}\left(  r\right)
\text{.} \label{observe}%
\end{equation}

Now, we estimate the integral of the functions on the right-hand side of the
equations (\ref{505}) over the balls $B_{\rho_{k}}\left(  x_{k}\right)  $.
Since $2-\frac{t_{k}}{\gamma_{k}^{2}}\geq2-\delta>1$ by (\ref{55}), there
exists some $\alpha>1$ and $C>0$ such that
\begin{equation}
\left(  1+t_{k}^{4}\right)  \exp\left(  t_{k}\left(  -2+t_{k}/\gamma_{k}%
^{2}\right)  \right)  \leq C\exp\left(  -\alpha t_{k}\right)  \label{T}%
\end{equation}
in $B_{\rho_{k}}\left(  x_{k}\right)  $. Using (\ref{T}), we can get%
\begin{equation}
\int_{B_{r}\left(  x_{k}\right)  }\frac{8\left(  1+t_{k}^{4}\right)
\exp\left(  t_{k}\left(  -2+t_{k}/\gamma_{k}^{2}\right)  \right)  }{r_{k}^{2}%
}dy\leq C_{\alpha}\left(  1-\left(  1+\left(  \frac{r}{r_{k}}\right)
^{2}\right)  ^{1-\alpha}\right)  \text{.} \label{2}%
\end{equation}
It remains to compute the integral of first term in (\ref{505}). Since
$$\left\vert \left(  \xi_{k}-S_{k}\right)  \left(  r\right)  \right\vert
\leq\left\Vert \left(  \xi_{k}-S_{k}\right)  ^{\prime}\right\Vert _{L^{\infty
}\left(  \left[  0,\rho_{k}\right]  \right)  }r,$$ then%
\begin{equation}
\int_{B_{r}\left(  x_{k}\right)  }\frac{8\exp\left(  -2t_{k}\right)  }%
{r_{k}^{2}}\left\vert \xi_{k}-S_{k}\right\vert dy\leq8\pi\left\Vert \left(
\xi_{k}-S_{k}\right)  ^{\prime}\right\Vert _{L^{\infty}\left(  \left[
0,\rho_{k}\right]  \right)  }r_{k}\left(  \arctan\left(  \frac{r}{r_{k}%
}\right)  -\frac{\frac{r}{r_{k}}}{1+\left(  \frac{r}{r_{k}}\right)  ^{2}%
}\right)  \text{.} \label{506}%
\end{equation}
Then by (\ref{observe}), (\ref{505}), (\ref{2}) and (\ref{506}), there exists
a constant $C^{^{\prime}}>1$ such that%
\begin{equation}
\frac{r\left\vert \left(  \xi_{k}-S_{k}\right)  ^{\prime}\left(  r\right)
\right\vert }{C^{^{\prime}}}\leq\frac{\left(  \frac{r}{r_{k}}\right)  ^{2}%
}{\gamma_{k}^{2}\left(  1+\left(  \frac{r}{r_{k}}\right)  ^{2}\right)  }%
+\frac{r_{k}\left\Vert \left(  \xi_{k}-S_{k}\right)  ^{\prime}\right\Vert
_{L^{\infty}\left(  \left[  0,\rho_{k}\right]  \right)  }\left(  \frac
{r}{r_{k}}\right)  ^{3}}{1+\left(  \frac{r}{r_{k}}\right)  ^{3}} \label{6}%
\end{equation}
for all $0\leq r\leq$ $\rho_{k}$. Now we show that%
\begin{equation}
r_{k}\left\Vert \left(  \xi_{k}-S_{k}\right)  ^{\prime}\right\Vert
_{L^{\infty}\left(  \left[  0,\rho_{k}\right]  \right)  }=O\left(  \frac
{1}{\gamma_{k}^{2}}\right)  \text{.} \label{4}%
\end{equation}
We prove this by contradiction and assume that
\begin{equation}
\gamma_{k}^{2}r_{k}\left\Vert \left(  \xi_{k}-S_{k}\right)  ^{\prime
}\right\Vert _{L^{\infty}\left(  \left[  0,\rho_{k}\right]  \right)  }%
=\gamma_{k}^{2}r_{k}\left\vert \left(  \xi_{k}-S_{k}\right)  ^{\prime}\left(
s_{k}\right)  \right\vert \rightarrow\infty,\text{ as }k\rightarrow
\infty\label{5}%
\end{equation}
for some $s_{k}\in\left[  0,\rho_{k}\right]  $. Indeed, from (\ref{6}) and
(\ref{5}), we immediately obtain
\begin{equation}
s_{k}=O\left(  r_{k}\right)  \text{ and }r_{k}=O\left(  s_{k}\right)  \text{,}
\label{300}%
\end{equation}
this implies that there exists $\alpha_{0}\in(0,+\infty]$ such that
$\frac{\rho_{k}}{r_{k}}\rightarrow$ $\alpha_{0}$ as $k\rightarrow\infty$.

Let $\tilde{\xi}_{k}$ be given by
\[
\tilde{\xi}_{k}\left(  s\right)  =\frac{\left(  \xi_{k}-S_{k}\right)  \left(
r_{k}s\right)  }{r_{k}\left\Vert \left(  \xi_{k}-S_{k}\right)  ^{\prime
}\right\Vert _{L^{\infty}\left(  \left[  0,\rho_{k}\right]  \right)  }%
}\text{.}%
\]
Then by (\ref{6}) and (\ref{5}), there exists a constant $C_{1}>0$ such that%
\begin{equation}
\left\vert \tilde{\xi}_{k}^{\prime}\left(  s\right)  \right\vert \leq
\frac{C_{1}}{1+s}\text{ in }\left[  0,\rho_{k}/r_{k}\right]  \text{.}\label{7}%
\end{equation}
Combining (\ref{505}), (\ref{T}) and (\ref{7}), using the standard elliptic
regularity theory, we get that
\begin{equation}
\tilde{\xi}_{k}\rightarrow\tilde{\xi}\text{ in }C_{loc}^{1}\left(
B_{\alpha_{0}}\left(  0\right)  \right)  \text{ as }k\rightarrow\infty
\text{,}\label{10}%
\end{equation}
for some $\tilde{\xi}\in C_{loc}^{1}\left(  B_{\alpha_{0}}\left(  0\right)
\right)  $ satisfying%
\begin{equation*}
\left\{
\begin{array}
[c]{c}%
-\Delta\tilde{\xi}=8\exp\left(  2\phi_{\infty}\right)  \tilde{\xi}\text{ in
}B_{\alpha_{0}}\left(  0\right)  \text{,}\\
\tilde{\xi}\left(  0\right)  =0\text{,}\\
\tilde{\xi}\text{ radially symmetric around }0\in\mathbb{R}^{2}\text{.}%
\end{array}
\right.  %
\end{equation*}
From \cite[Lemma C.1.]{PLA}, one can obtain
\begin{equation}
\tilde{\xi}\equiv0\text{ in }B_{\alpha_{0}}\left(  0\right)  \text{.}\label{9}%
\end{equation}

Now, we can improve the estimates  in (\ref{6}) from (\ref{9}). Indeed, using
(\ref{7}), (\ref{10}), (\ref{9}) and the dominated convergence theorem, we can
obtain%
\begin{equation}
\int_{B_{\rho_{k}}\left(  x_{k}\right)  }\frac{\exp\left(  -2t_{k}\right)
}{r_{k}^{2}}\left\vert \xi_{k}-S_{k}\right\vert dy=o\left(  r_{k}\left\Vert
\left(  \xi_{k}-S_{k}\right)  ^{\prime}\right\Vert _{L^{\infty}\left(  \left[
0,\rho_{k}\right]  \right)  }\right)  \text{.} \label{11}%
\end{equation}
By repeating the argument for (\ref{6}), replacing (\ref{506}) with (\ref{11}),
and utilizing (\ref{5}), we get%
\begin{equation}
r\left\vert \left(  \xi_{k}-S_{k}\right)  ^{\prime}\left(  r\right)
\right\vert =o\left(  r_{k}\left\vert \left\vert \left(  \xi_{k}-S_{k}\right)
^{\prime}\right\vert \right\vert _{L^{\infty}\left(  \left[  0,\rho
_{k}\right]  \right)  }\right)  , \label{12}%
\end{equation}
for all $0\leq r\leq$ $\rho_{k},$ as $k\rightarrow\infty$. If we choose
$r=s_{k}$ in (\ref{12}), then we immediately have $s_{k}=o\left(
r_{k}\right)  $, but this contradicts to (\ref{300}). Hence, (\ref{4}) is proved.

Now, plugging (\ref{4}) into (\ref{6}), using the fact $\xi_{k}\left(
0\right)  =S_{k}\left(  0\right)  =0$ and the fundamental theorem of calculus,
we obtain that%
\[
\left\vert \left\vert \left(  \xi_{k}-S_{k}\right)  \right\vert \right\vert
_{L^{\infty}\left(  \left[  0,\rho_{k}\right]  \right)  }=O\left(
\frac{1+t_{k}}{\gamma_{k}^{2}}\right)
\]
as $k\rightarrow\infty$, which together with (\ref{P}) yields $\rho
_{k}=r_{k,\delta\text{ }}$and we can conclude the proof.
\end{proof}

\subsection{A compare principle}

Now, we define $\bar{v}_{k}$ by%
\begin{equation}
\bar{v}_{k}\left(  z\right)  =\frac{1}{2\pi\left\vert x_{k}-z\right\vert }%
\int_{\partial B_{\left\vert x_{k}-z\right\vert }\left(  x_{k}\right)  }%
v_{k}\left(  y\right)  d\sigma\left(  y\right)  \label{246}%
\end{equation}
for all $z\in\Omega\backslash\left\{  x_{k}\right\}  $ and $\bar{v}_{k}\left(
x_{k}\right)  =v_{k}\left(  x_{k}\right)  $, and set
\begin{equation}
R_{k\text{ }}:=\sup\left\{  r\in(0,r_{k,\delta}]:\left\vert \bar{v}_{k}%
-V_{k}\right\vert \leq\frac{\eta}{\gamma_{k}}\text{ in }B_{r}\left(
x_{k}\right)  \text{ for }0<\eta<1\right\}  \text{,} \label{213}%
\end{equation}

First, we estimate the value of $v_{k}$ on $B_{R_{k}}\left(  x_{k}\right)  $.

\begin{lemma}
\label{90}It holds that%
\[
\underset{B_{R_{k}}\left(  x_{k}\right)  }{\min}v_{k}\geq\left(
1-\delta+o_{k}\left(  1\right)  \right)  \gamma_{k}\text{.}%
\]

\end{lemma}

\begin{proof}
Decomposing $v_{k}$ into
\begin{equation}
v_{k}=V_{k}+\omega_{k}\text{.} \label{87}%
\end{equation}
Since $-\Delta V_{k}\geq0$ in $\left[  0,r_{k,\delta}\right]  $ from
 (\ref{main}) and (\ref{o}), then $V_{k}$ is radially decreasing
in $B_{r_{k,\delta}}\left(  x_{k}\right)  $. Using (\ref{510}), we can get
\begin{equation}
\underset{B_{R_{k}}\left(  x_{k}\right)  }{\min}V_{k}\geq\left(
1-\delta+o_{k}\left(  1\right)  \right)  \gamma_{k}\text{.} \label{yy}%
\end{equation}
We define a radius $\tilde{R}_{k}$ by%
\begin{equation}
\tilde{R}_{k}=\sup\left\{  r\in(0,R_{k\text{ }}]:\frac{1}{2}\left(
1-\delta\right)  \gamma_{k}\left\vert x-x_{k}\right\vert \left\vert \nabla
v_{k}\left(  x\right)  \right\vert \leq C\text{ in }B_{r}\left(  x_{k}\right)
\right\}  \label{86}%
\end{equation}
where $C$ is given in Lemma \ref{85}.

Now, we claim that $\tilde{R}_{k}=R_{k}$. By employing (\ref{86}), we can get%
\begin{equation}
\left\vert \bar{v}_{k}-v_{k}\right\vert \leq\frac{2C}{\left(  1-\delta\right)
\gamma_{k}}\pi\text{ in }B_{\tilde{R}_{k}}\left(  x_{k}\right)  . \label{af}%
\end{equation}
Using the definition of $R_{k}$, from (\ref{af}) and (\ref{87}), we can derive
the following estimates:%
\[
\left\vert \omega_{k}\right\vert =\left\vert v_{k}-V_{k}\right\vert
\leq\left(  \eta+\frac{2C\pi}{1-\delta}\right)  \gamma_{k}^{-1}\text{ in
}B_{\tilde{R}_{k}}\left(  x_{k}\right)  ,
\]
hence by (\ref{yy}) we can obtain a lower bound of $v_{k}$:%
\[
\underset{B_{\tilde{R}_{k}}\left(  x_{k}\right)  }{\min}v_{k}\geq\left(
1-\delta+o_{k}\left(  1\right)  \right)  \gamma_{k}\text{,}%
\]
which together with Lemma \ref{85} imply%
\begin{equation}
\left(  1-\delta+o_{k}\left(  1\right)  \right)  \gamma_{k}\left\vert
x-x_{k}\right\vert \left\vert \nabla v_{k}\left(  x\right)  \right\vert \leq
C\text{ in }B_{\tilde{R}_{k}}\left(  x_{k}\right)  \text{.} \label{ag}%
\end{equation}
If $\tilde{R}_{k}<R_{k}$, from (\ref{ag}) we can find some $r=\left\vert
x-x_{k}\right\vert >\tilde{R}_{k}$ such that%
\[
\frac{1}{2}\left(  1-\delta\right)  \gamma_{k}\left\vert x-x_{k}\right\vert
\left\vert \nabla v_{k}\left(  x\right)  \right\vert \leq C\text{,}%
\]
which is a contradiction with (\ref{86}), hence $\tilde{R}_{k}=R_{k}$, and the
proof is finished.
\end{proof}

\begin{lemma}
[Comparison principle]\label{507}Let $\bar{v}_{k}$ be defined in (\ref{246}).
Then
\[
\left\vert \left\vert \bar{v}_{k}-V_{k}\right\vert \right\vert _{L^{\infty
}\left(  B_{r_{k,\delta}}\left(  x_{k}\right)  \right)  }=o\left(  \frac
{1}{\gamma_{k}}\right)  \text{.}%
\]
Moreover, we have
\[
\left\vert \left\vert \nabla\left(  v_{k}-V_{k}\right)  \right\vert
\right\vert _{_{L^{\infty}\left(  B_{r_{k,\delta}}\left(  x_{k}\right)
\right)  }}=O\left(  \frac{1}{\gamma_{k}r_{k,\delta\text{ }}}\right)  \text{,}%
\]
and there exists a constant $C>0$ such that%
\begin{equation*}
\left\vert v_{k}-V_{k}\right\vert \leq C\frac{\left\vert \cdot-x_{k}%
\right\vert }{\gamma_{k}r_{k,\delta\text{ }}}\text{ in }B_{r_{k,\delta}%
}\left(  x_{k}\right)  \text{.} %
\end{equation*}

\end{lemma}

\begin{proof}
For any fixed $\eta\in\left(  0,1\right)  $, let $R_{k\text{ }}$ be given by
(\ref{213}), from Lemma \ref{15} and (\ref{214}), one can easily verify
\begin{equation}
\frac{R_{k\text{ }}}{r_{k}}\rightarrow\infty\text{ as }k\rightarrow
\infty\text{.} \label{43}%
\end{equation}
From (\ref{213}), we have
\begin{equation}
\left\vert \bar{v}_{k}\left(  r\right)  -V_{k}\left(  r\right)  \right\vert
\leq\frac{\eta}{\gamma_{k}}\text{ for all }0\leq r\leq R_{k\text{ }}\text{,}
\label{500}%
\end{equation}
and%
\begin{equation}
\left\vert \bar{v}_{k}\left(  R_{k}\right)  -V_{k}\left(  R_{k}\right)
\right\vert =\frac{\eta}{\gamma_{k}}\text{ if }R_{k\text{ }}<r_{k,\delta
}\text{.} \label{47}%
\end{equation}
Combining Lemmas \ref{85} and \ref{90}, we get%
\begin{equation}
\left\vert \cdot-x_{k}\right\vert \left\vert \nabla v_{k}\right\vert \leq
\frac{C}{\gamma_{k}\left(  1-\delta+o_{k}\left(  1\right)  \right)  }\text{ in
}B_{R_{k}}\left(  x_{k}\right)  \text{,} \label{16}%
\end{equation}
which together with (\ref{500}) imply that%
\begin{equation}
\left\vert \omega_{k}\right\vert \leq\frac{\eta+C\pi}{\gamma_{k}}\text{ in
}B_{R_{k}}\left(  x_{k}\right)  \text{.} \label{504}%
\end{equation}
In $B_{R_{k}}\left(  x_{k}\right)  $, using Lemma \ref{15}, (\ref{504})  and (\ref{227}), we
can write%
\begin{align}
\left\vert -\Delta\omega_{k}\right\vert  &  =\left\vert \frac{4\pi}{E_{k}%
}\left(  v_{k}e^{v_{k}^{2}}-\frac{p}{2\left(  4\pi\right)  ^{\frac{p}{2}}%
}\lambda_{k}v_{k}^{p-1}\right)  -\frac{4\pi}{E_{k}}\left(  V_{k}e^{V_{k}^{2}%
}-\frac{p}{2\left(  4\pi\right)  ^{\frac{p}{2}}}\lambda_{k}V_{k}^{p-1}\right)
\right\vert \nonumber\\
&  =\left\vert \frac{4\pi}{E_{k}}e^{V_{k}^{2}}\left(  \left(  V_{k}+\omega
_{k}\right)  e^{\omega_{k}^{2}+2V_{k}\omega_{k}}-V_{k}\right)  -\frac{4\pi
}{E_{k}}\frac{p}{2\left(  4\pi\right)  ^{\frac{p}{2}}}\lambda_{k}\left(
v_{k}^{p-1}-V_{k}^{p-1}\right)  \right\vert \nonumber\\
&  \leq\frac{C}{E_{k}}e^{V_{k}^{2}}\left(  1+2V_{k}^{2}\right)  \left(
1+o_{k}\left(  1\right)  \right)  \left\vert \omega_{k}\right\vert +\frac
{C}{E_{k}}\lambda_{k}\left(  1+o_{k}\left(  1\right)  \right)  V_{k}%
^{p-2}\left\vert \omega_{k}\right\vert \label{21}\\
&  \leq\frac{C}{E_{k}}e^{V_{k}^{2}}\left(  1+2V_{k}^{2}\right)  \left(
1+o_{k}\left(  1\right)  \right)  \left\vert \omega_{k}\right\vert
\text{.}\nonumber
\end{align}

Now, we demonstrate that the functions $\left(  \omega_{k}\right)  $ can be
approximated by some harmonic functions $\left(  \phi_{k}\right)  $ that
coincides with $\left(  \omega_{k}\right)  $ on the boundary $\partial
B_{R_{k}}\left(  x_{k}\right)  $, that is,
\begin{equation}
\left\vert \left\vert \nabla\left(  \omega_{k}-\phi_{k}\right)  \right\vert
\right\vert _{L^{\infty}\left(  B_{R_{k}}\left(  x_{k}\right)  \right)
}=o\left(  \gamma_{k}^{-1}R_{k}^{-1}\right)  \text{,} \label{44}%
\end{equation}%
\begin{equation}
\left\vert \left\vert \omega_{k}-\phi_{k}\right\vert \right\vert _{L^{\infty
}\left(  B_{R_{k}}\left(  x_{k}\right)  \right)  }=o\left(  \gamma_{k}%
^{-1}\right)  \text{,} \label{45}%
\end{equation}
and
\begin{equation}
\left\vert \left\vert \nabla\phi_{k}\right\vert \right\vert _{L^{\infty
}\left(  B_{R_{k}}\left(  x_{k}\right)  \right)  }=O\left(  \gamma_{k}%
^{-1}R_{k\text{ }}^{-1}\right)  \text{.} \label{30}%
\end{equation}
Since the proofs for (\ref{44}), (\ref{45}) and (\ref{30}) are lengthy and
highly technical, we will defer their presentation until the end of the proof.

From (\ref{44}) and (\ref{45}), and the fact $\omega_{k}\left(  x_{k}\right)
=0$ and $\nabla\omega_{k}\left(  x_{k}\right)  =0$, we get%
\begin{equation}
\phi_{k}\left(  x_{k}\right)  =o\left(  \gamma_{k}^{-1}\right)  \text{,
\ }\nabla\phi_{k}\left(  x_{k}\right)  =o\left(  \gamma_{k}^{-1}R_{k}%
^{-1}\right)  \text{.} \label{46}%
\end{equation}
Since $\phi_{k}$ is harmonic, (\ref{46}) gives
\[
\gamma_{k}\phi_{k}\left(  x_{k}\right)  =\frac{\gamma_{k}}{2\pi R_{k}}%
\int_{\partial B_{R_{k}}\left(  x_{k}\right)  }\phi_{k}d\sigma\rightarrow
0\text{ as }k\rightarrow\infty\text{.}%
\]
Due to $v_{k}-V_{k}=\omega_{k}=\phi_{k}$ on $\partial B_{R_{k}}\left(
x_{k}\right)  $, this leads to $\gamma_{k}\left\vert \bar{v}_{k}\left(
R_{k}\right)  -V_{k}\left(  R_{k}\right)  \right\vert \rightarrow0$ as
$k\rightarrow\infty$, which is impossible if $R_{k}<r_{k,\delta\text{ }}$by
(\ref{47}). Thus, we have proved $R_{k}=r_{k,\delta}$ and%
\[
\left\vert \left\vert \bar{v}_{k}-V_{k}\right\vert \right\vert _{L^{\infty
}\left(  B_{r_{k,\delta}}\left(  x_{k}\right)  \right)  }=o\left(  \gamma
_{k}^{-1}\right)  .
\]
It follows from (\ref{30}) and (\ref{44}) that%
\begin{equation}
\left\vert \left\vert \nabla\left(  v_{k}-V_{k}\right)  \right\vert
\right\vert _{L^{\infty}\left(  B_{r_{k,\delta}}\left(  x_{k}\right)  \right)
}=O\left(  \gamma_{k}^{-1}r_{k,\delta}^{-1}\right)  \text{.} \label{48}%
\end{equation}
Since $V_{k}\left(  x_{k}\right)  =v_{k}\left(  x_{k}\right)  =\gamma_{k}$,
(\ref{48}) obviously implies
\[
\left\vert v_{k}-V_{k}\right\vert \leq C\frac{\left\vert \cdot-x_{k}%
\right\vert }{\gamma_{k}r_{k,\delta\text{ }}}\text{ in }B_{r_{k,\delta}%
}\left(  x_{k}\right)  \text{. }%
\]

In the following, we will finished the proof of this lemma by proving
(\ref{44}), (\ref{45}) and (\ref{30}). Let $\phi_{k}$ be the harmonic function
satisfying%
\begin{equation}
-\Delta\phi_{k}=0\text{ in }B_{R_{k}}\left(  x_{k}\right)  ,\ \phi_{k}%
=\omega_{k}\text{ on }\partial B_{R_{k}}\left(  x_{k}\right)  \text{.}%
\label{502}%
\end{equation}
From (\ref{yy}) and (\ref{227}), we can rewrite equation (\ref{3.10}) as%
\begin{align*}
e^{t}\left(  \left(  1-e^{-t}\right)  V_{k}^{\prime}\right)  ^{\prime} &
=-\frac{1}{\gamma_{k}^{2}}\left(  V_{k}e^{V_{k}^{2}-\gamma_{k}^{2}+2t}%
-\frac{p}{2\left(  4\pi\right)  ^{\frac{p}{2}}}\lambda_{k}\left\vert
V_{k}\right\vert ^{p-1}e^{-\gamma_{k}^{2}+2t}\right)  \\
&  =-\frac{V_{k}}{\gamma_{k}^{2}}\left(  1+o_{k}\left(  1\right)  \right)
e^{V_{k}^{2}-\gamma_{k}^{2}+2t}%
\end{align*}
in $B_{R_{k}}\left(  x_{k}\right)  $, by introducing the change of variable
\[
t=\ln\left(  1+\frac{r^{2}}{r_{k}^{2}}\right)  \text{.}%
\]
Then by employing a similar approach as discussed in \cite[claim 5.3]{Dru2},
we can arrive at the following conclusion%
\begin{equation}
\left\vert \nabla V_{k}\left(  x\right)  -\gamma_{k}^{-1}\frac{2\left(
x-x_{k}\right)  }{\left\vert x-x_{k}\right\vert ^{2}+r_{k}^{2}}\right\vert
\leq C\gamma_{k}^{-2}\frac{\left\vert x-x_{k}\right\vert }{\left\vert
x-x_{k}\right\vert ^{2}+r_{k}^{2}}\text{\ \ for }t_{k}\left(  x\right)
\leq\delta\gamma_{k}^{2}\text{.}\label{17}%
\end{equation}
By (\ref{16}), (\ref{17}) and (\ref{43}), we can get $\left\vert \nabla
\omega_{k}\right\vert \leq C\gamma_{k}^{-1}R_{k\text{ }}^{-1}$ on $\partial
B_{R_{k}}\left(  x_{k}\right)  $ for some $C>0$. Through (\ref{504}) and using
the standard elliptic estimates for (\ref{502}), we can obtain (\ref{30}).
Moreover, since $\phi_{k}$ is harmonic function in $B_{R_{k}}\left(
x_{k}\right)  $, then
\[
\left\vert \left\vert \phi_{k}\right\vert \right\vert _{L^{\infty}\left(
B_{R_{k}}\left(  x_{k}\right)  \right)  }\leq\left\vert \left\vert \phi
_{k}\right\vert \right\vert _{L^{\infty}\left(  \partial B_{R_{k}}\left(
x_{k}\right)  \right)  }\leq\frac{C}{\gamma_{k}},
\]
namely, $\left\vert \left\vert \gamma_{k}\phi_{k}\right\vert \right\vert
_{L^{\infty}\left(  B_{R_{k}}\left(  x_{k}\right)  \right)  }\leq C$. Hence
using the standard elliptic theory for (\ref{502}) again, there exists
$\phi_{0}$ such that
\begin{equation}
\gamma_{k}\phi_{k}\left(  x_{k}+R_{k\text{ }}x\right)  \rightarrow\phi
_{0}\left(  x\right)  \text{ in }C_{loc}^{2}\left(  B_{1}\left(  0\right)
\right)  \text{ as }k\rightarrow\infty\text{.}\label{503}%
\end{equation}

In the following, we will give the proof for (\ref{44}) and (\ref{45}). First,
we prove the following assertion: for all $y\in B_{R_{k}}\left(  x_{k}\right)
$,
\begin{equation}
\left\vert \nabla\left(  \omega_{k}-\phi_{k}\right)  \left(  y\right)
\right\vert \leq C\left\vert \left\vert \nabla\omega_{k}\right\vert
\right\vert _{L^{\infty}\left(  B_{R_{k}}\left(  x_{k}\right)  \right)
}\left(  \frac{r_{k}}{r_{k}+\left\vert y-x_{k}\right\vert }+o\left(  \frac
{1}{\gamma_{k}}\right)  \right)  \text{.} \label{29}%
\end{equation}
Let $y_{k}\in B_{R_{k}}\left(  x_{k}\right)  $, using the Green representation
formula and (\ref{21}), for $k$ large enough,\ we get
\begin{align}
\left\vert \nabla\left(  \omega_{k}-\phi_{k}\right)  \left(  y_{k}\right)
\right\vert  &  \leq\frac{C}{E_{k}}\int_{B_{R_{k}}\left(  x_{k}\right)  }%
\frac{1}{\left\vert x-y_{k}\right\vert }e^{V_{k}^{2}}\left(  1+2V_{k}%
^{2}\right)  \left\vert \omega_{k}\right\vert dx\nonumber\\
&  =\frac{C}{E_{k}}\left(  \int_{\Omega_{0,k}}+\int_{\Omega_{1,k}}\right)
\frac{1}{\left\vert x-y_{k}\right\vert }e^{V_{k}^{2}}\left(  1+2V_{k}%
^{2}\right)  \left\vert \omega_{k}\right\vert dx\nonumber\\
&  :=I_{0,k}+I_{1,k}, \label{23}%
\end{align}
where
\[
\Omega_{0,k}=B_{R_{k}}\left(  x_{k}\right)  \cap\left\{  t_{k}\left(
x\right)  \leq t_{1,k}=\frac{1}{4}\gamma_{k}^{2}\right\}  \text{,}%
\]
and%
\[
\Omega_{1,k}=B_{R_{k}}\left(  x_{k}\right)  \cap\left\{  t_{1,k}\leq
t_{k}\left(  x\right)  \leq t_{2,k}=\gamma_{k}^{2}-\gamma_{k}\right\}
\text{.}%
\]
It is worth emphasizing that the decomposition of $B_{R_{k}}\left(
x_{k}\right)  $ mentioned above is reasonable, since from (\ref{55}), it is
clear that $B_{R_{k}}\left(  x_{k}\right)  =\Omega_{0,k}\cup\Omega_{1,k}$ for
$k$ large enough.

We now estimate $I_{i,k}$, for $i=0,1$. Using Lemma \ref{15} and the fact
$\omega_{k}\left(  x_{k}\right)  =0$, we obtain
\begin{align*}
I_{i,k}  &  \leq\frac{C}{E_{k}}\int_{\Omega_{i,k}}\frac{1}{\left\vert
x-y_{k}\right\vert }\gamma_{k}^{2}e^{\gamma_{k}^{2}}e^{\frac{t_{k}^{2}}%
{\gamma_{k}^{2}}-2t_{k}}\left\vert \omega_{k}\left(  x\right)  -\omega
_{k}\left(  x_{k}\right)  \right\vert dx\\
&  \leq Cr_{k}^{-2}\left\vert \left\vert \nabla\omega_{k}\right\vert
\right\vert _{L^{\infty}\left(  \Omega_{i,k}\right)  }\int_{\Omega_{i,k}}%
\frac{\left\vert x-x_{k}\right\vert }{\left\vert x-y_{k}\right\vert }%
e^{\frac{t_{k}^{2}}{\gamma_{k}^{2}}-2t_{k}}dx\text{, for }i=0,1.
\end{align*}
To estimate the integrals above, we distinguish two cases: $\left\vert
y_{k}-x_{k}\right\vert =O\left(  r_{k}\right)  $ and $\frac{\left\vert
y_{k}-x_{k}\right\vert }{r_{k}}\rightarrow\infty$, as $k\rightarrow\infty.$

Case 1: $\left\vert y_{k}-x_{k}\right\vert =O\left(  r_{k}\right)  $, as
$k\rightarrow\infty$. Since $t_{k}^{2}\left(  x\right)  /\gamma_{k}^{2}%
-2t_{k}\left(  x\right)  \leq-\frac{7}{4}t_{k}\left(  x\right)  $ in
$\Omega_{0,k}$, we have
\begin{align*}
I_{0,k}  &  \leq Cr_{k}^{-2}\left\vert \left\vert \nabla\omega_{k}\right\vert
\right\vert _{L^{\infty}\left(  \Omega_{0,k}\right)  }\int_{\Omega_{0,k}}%
\frac{\left\vert x-x_{k}\right\vert }{\left\vert x-y_{k}\right\vert }\left(
1+\frac{\left\vert x-x_{k}\right\vert ^{2}}{r_{k}^{2}}\right)  ^{-\frac{7}{4}%
}dx\\
&  \leq C\left\vert \left\vert \nabla\omega_{k}\right\vert \right\vert
_{L^{\infty}\left(  \Omega_{0,k}\right)  }\int_{\mathbb{R}^{2}}\frac
{\left\vert y\right\vert }{\left\vert y-\frac{y_{k}-x_{k}}{r_{k}}\right\vert
}\left(  1+\left\vert y\right\vert ^{2}\right)  ^{-\frac{7}{4}}dy\text{.}%
\end{align*}
For the case $\left\vert y_{k}-x_{k}\right\vert =O\left(  r_{k}\right)  $, we
obtain%
\begin{equation}
I_{0,k}=O\left(  \left\vert \left\vert \nabla\omega_{k}\right\vert \right\vert
_{L^{\infty}\left(  \Omega_{0,k}\right)  }\right)  \text{.} \label{24}%
\end{equation}
In $\Omega_{1,k}$, it is clear that $\left\vert x-x_{k}\right\vert \leq\left(
1+o_{k}\left(  1\right)  \right)  \left\vert x-y_{k}\right\vert $ by the fact
$\left\vert y_{k}-x_{k}\right\vert =O\left(  r_{k}\right)  $, as above, we
have%
\begin{align}
I_{1,k}  &  \leq Cr_{k}^{-2}\left\vert \left\vert \nabla\omega_{k}\right\vert
\right\vert _{L^{\infty}\left(  \Omega_{1,k}\right)  }\int_{\Omega_{1,k}%
}e^{\frac{t_{k}^{2}}{\gamma_{k}^{2}}-2t_{k}}dx\label{06}\\
&  \leq C\left\vert \left\vert \nabla\omega_{k}\right\vert \right\vert
_{L^{\infty}\left(  \Omega_{1,k}\right)  }\int_{t_{1,k}}^{t_{2,k}}%
e^{\frac{t^{2}}{\gamma_{k}^{2}}-t}dt\text{.}\nonumber
\end{align}
Since $\frac{1}{4}\gamma_{k}^{2}=t_{1,k}\leq t\leq t_{2,k}=\gamma_{k}%
^{2}-\gamma_{k}$, then
\begin{equation}
\frac{t^{2}}{\gamma_{k}^{2}}-t\leq-t/\gamma_{k}\leq-\frac{1}{4}\gamma
_{k}\text{,} \label{add25}%
\end{equation}
and we have%
\begin{equation}
I_{1,k}\leq C\left\vert \left\vert \nabla\omega_{k}\right\vert \right\vert
_{L^{\infty}\left(  \Omega_{1,k}\right)  }\gamma_{k}^{2}e^{-\frac{1}{4}%
\gamma_{k}}\text{.} \label{25}%
\end{equation}
Coming back to (\ref{23}) with (\ref{24}) and (\ref{25}), we obtain%
\begin{align*}
\left\vert \nabla\left(  \omega_{k}-\phi_{k}\right)  \left(  y_{k}\right)
\right\vert  &  \leq C  \left\vert \left\vert \nabla\omega
_{k}\right\vert \right\vert _{L^{\infty}\left(  \Omega_{0,k}\right)  }
+C\left\vert \left\vert \nabla\omega_{k}\right\vert \right\vert _{L^{\infty
}\left(  \Omega_{0,k}\right)  }\gamma_{k}^{2}e^{-\frac{1}{4}\gamma_{k}}\\
&  \leq C\left\vert \left\vert \nabla\omega_{k}\right\vert \right\vert
_{L^{\infty}\left(  B_{R_{k}}\left(  x_{k}\right)  \right)  }\text{,}%
\end{align*}
which implies assertion (\ref{29}) holds true in the case $\left\vert
y_{k}-x_{k}\right\vert =O\left(  r_{k}\right)  $.

Case 2: $\left\vert y_{k}-x_{k}\right\vert /r_{k}\rightarrow\infty$ as
$k\rightarrow\infty$. We follow the lines in the case 1 and emphasize the
differences. Direct calculation gives that
\begin{align*}
I_{0,k}  &  \leq Cr_{k}^{-2}\left\vert \left\vert \nabla\omega_{k}\right\vert
\right\vert _{L^{\infty}\left(  \Omega_{0,k}\right)  }\int_{\Omega_{0,k}}%
\frac{\left\vert x-x_{k}\right\vert }{\left\vert x-y_{k}\right\vert }%
e^{\frac{t_{k}^{2}}{\gamma_{k}^{2}}-2t_{k}}dx\\
&  \leq C\left\vert \left\vert \nabla\omega_{k}\right\vert \right\vert
_{L^{\infty}\left(  \Omega_{0,k}\right)  }\int_{\mathbb{R}^{2}}\frac
{\left\vert y\right\vert }{\left\vert y-\frac{y_{k}-x_{k}}{r_{k}}\right\vert
}\left(  1+\left\vert y\right\vert ^{2}\right)  ^{-\frac{7}{4}}dy\text{.}%
\end{align*}
Then we can estimate%
\begin{align*}
&  \int_{\mathbb{R}^{2}}\frac{\left\vert y\right\vert }{\left\vert
y-\frac{y_{k}-x_{k}}{r_{k}}\right\vert }\left(  1+\left\vert y\right\vert
^{2}\right)  ^{-\frac{7}{4}}dy\\
&  =\left(  \frac{\left\vert y_{k}-x_{k}\right\vert }{r_{k}}\right)
^{-\frac{3}{2}}\int_{\mathbb{R}^{2}}\frac{\left\vert x\right\vert }{\left\vert
x-\frac{y_{k}-x_{k}}{\left\vert y_{k}-x_{k}\right\vert }\right\vert }\left(
\frac{r_{k}^{2}}{\left\vert x_{k}-y_{k}\right\vert ^{2}}+\left\vert
x\right\vert ^{2}\right)  ^{-\frac{7}{4}}dx\\
&  \leq C\left(  \frac{\left\vert y_{k}-x_{k}\right\vert }{r_{k}}\right)
^{-\frac{3}{2}}+2\left(  \frac{\left\vert y_{k}-x_{k}\right\vert }{r_{k}%
}\right)  ^{-\frac{3}{2}}\int_{B_{\frac{1}{2}}\left(  0\right)  }\left\vert
x\right\vert \left(  \frac{r_{k}^{2}}{\left\vert x_{k}-y_{k}\right\vert ^{2}%
}+\left\vert x\right\vert ^{2}\right)  ^{-\frac{7}{4}}dx\\
&  \leq C\left(  \frac{\left\vert y_{k}-x_{k}\right\vert }{r_{k}}\right)
^{-\frac{3}{2}}+2\frac{r_{k}}{\left\vert y_{k}-x_{k}\right\vert }%
\int_{\mathbb{R}^{2}}\left\vert y\right\vert \left(  1+\left\vert y\right\vert
^{2}\right)  ^{-\frac{7}{4}}dy\text{,}%
\end{align*}
hence, we obtained%
\begin{equation}
I_{0,k}\leq C\left\vert \left\vert \nabla\omega_{k}\right\vert \right\vert
_{L^{\infty}\left(  \Omega_{0,k}\right)  }\frac{r_{k}}{\left\vert y_{k}%
-x_{k}\right\vert }\text{.} \label{27}%
\end{equation}
Direct calculation gives
\[
I_{1,k}\leq Cr_{k}^{-2}\left\vert \left\vert \nabla\omega_{k}\right\vert
\right\vert _{L^{\infty}\left(  \Omega_{1,k}\right)  }\int_{\Omega_{1,k}}%
\frac{\left\vert x-x_{k}\right\vert }{\left\vert x-y_{k}\right\vert }%
e^{\frac{t_{k}^{2}}{\gamma_{k}^{2}}-2t_{k}}dx\text{.}%
\]
We rewrite the integral above as
\[
r_{k}^{-2}\left(  \int_{\Omega_{1,k}\backslash B_{\frac{1}{2}\left\vert
y_{k}-x_{k}\right\vert }\left(  y_{k}\right)  }+\int_{\Omega_{1,k}\cap
B_{\frac{1}{2}\left\vert y_{k}-x_{k}\right\vert }\left(  y_{k}\right)
}\right)  \frac{\left\vert x-x_{k}\right\vert }{\left\vert x-y_{k}\right\vert
}e^{\frac{t_{k}^{2}}{\gamma_{k}^{2}}-2t_{k}}dx:=I_{k}+II_{k}.
\]
Using the triangle inequality and the similar calculations as for
(\ref{06})~in Case 1, we derive that
\[
I_{k}\leq3r_{k}^{-2}\int_{\Omega_{1,k}}e^{\frac{t_{k}^{2}}{\gamma_{k}^{2}%
}-2t_{k}}dx\leq C\gamma_{k}^{2}e^{-\frac{1}{4}\gamma_{k}}\text{.}%
\]
We obverse that if
\begin{equation*}
t_{k}\left(  \frac{3}{2}\left\vert y_{k}-x_{k}\right\vert \right)
=\log\left(  1+\frac{9\left\vert y_{k}-x_{k}\right\vert ^{2}}{4r_{k}^{2}%
}\right)  \leq t_{1,k}=\frac{1}{4}\gamma_{k}^{2}\text{,} %
\end{equation*}
then $\Omega_{1,k}\cap B_{\frac{1}{2}\left\vert y_{k}-x_{k}\right\vert
}\left(  y_{k}\right)  =\emptyset$. Hence, if (\ref{06}) holds, then the
estimates  of $I_{1,k}$ is finished. Otherwise, if $\log\left(  1+\frac
{9\left\vert y_{k}-x_{k}\right\vert ^{2}}{4r_{k}^{2}}\right)  >\frac{1}%
{4}\gamma_{k}^{2}$, it follows that%
\[
s_{k}=t_{k}\left(  \frac{\left\vert y_{k}-x_{k}\right\vert }{2}\right)
=\log\left(  1+\frac{\left\vert y_{k}-x_{k}\right\vert ^{2}}{4r_{k}^{2}%
}\right)  >\frac{1}{4}\gamma_{k}^{2}-\log9\text{.}%
\]
When $\Omega_{1,k}\cap B_{\frac{1}{2}\left\vert y_{k}-x_{k}\right\vert
}\left(  y_{k}\right)  \neq\emptyset$, observe that $\frac{y_{k}-x_{k}}{2}$ is
the closest point to the point $x_{k}$~in $B_{\frac{1}{2}\left\vert
y_{k}-x_{k}\right\vert }\left(  y_{k}\right)$, therefore, we
have $s_{k}\leq t_{2,k}=\gamma_{k}^{2}-\gamma_{k}$. Then similar as
(\ref{add25}), we obtain $\frac{s_{k}^{2}}{\gamma_{k}^{2}}-s_{k}\leq-\frac
{1}{4}\gamma_{k}$.

Since $\frac{t^{2}}{\gamma_{k}^{2}}-2t$ is decreasing in $\left[  \frac{1}%
{4}\gamma_{k}^{2},\gamma_{k}^{2}-\gamma_{k}\right]  $, we have
\begin{align*}
&  r_{k}^{-2}\int_{\Omega_{1,k}\cap B_{\frac{1}{2}\left\vert y_{k}%
-x_{k}\right\vert }\left(  y_{k}\right)  }\frac{\left\vert x-x_{k}\right\vert
}{\left\vert x-y_{k}\right\vert }e^{\frac{t_{k}^{2}}{\gamma_{k}^{2}}-2t_{k}%
}dx\label{xx}\\
&  \leq\frac{3}{2}r_{k}^{-2}\left\vert y_{k}-x_{k}\right\vert e^{\frac
{s_{k}^{2}}{\gamma_{k}^{2}}-2s_{k}}\int_{B_{\frac{1}{2}\left\vert y_{k}%
-x_{k}\right\vert }\left(  y_{k}\right)  }\frac{1}{\left\vert x-y_{k}%
\right\vert }dx\nonumber\\
&  \leq C\frac{\left\vert y_{k}-x_{k}\right\vert ^{2}}{r_{k}^{2}}%
e^{\frac{s_{k}^{2}}{\gamma_{k}^{2}}-2s_{k}}\nonumber\\
&  \leq C\frac{\left\vert y_{k}-x_{k}\right\vert ^{2}}{r_{k}^{2}}e^{-\frac
{1}{4}\gamma_{k}}e^{-s_{k}}\nonumber\\
&  \leq C\frac{\left\vert y_{k}-x_{k}\right\vert ^{2}}{r_{k}^{2}}e^{-\frac
{1}{4}\gamma_{k}}\left(  1+\frac{\left\vert y_{k}-x_{k}\right\vert ^{2}%
}{4r_{k}^{2}}\right)  ^{-1}\leq Ce^{-\frac{1}{4}\gamma_{k}}\text{.}\nonumber
\end{align*}
Thus, we have%
\begin{equation}
I_{1,k}\leq C\gamma_{k}^{2}e^{-\frac{1}{4}\gamma_{k}}\left\vert \left\vert
\nabla\omega_{k}\right\vert \right\vert _{L^{\infty}\left(  \Omega
_{1,k}\right)  }\text{.} \label{28}%
\end{equation}
Hence, it follows from (\ref{27}) and (\ref{28}) that%
\begin{align*}
\left\vert \nabla\left(  \omega_{k}-\phi_{k}\right)  \left(  y_{k}\right)
\right\vert  &  \leq C\left\vert \left\vert \nabla\omega_{k}\right\vert
\right\vert _{L^{\infty}\left(  \Omega_{0,k}\right)  }\frac{r_{k}}{\left\vert
y_{k}-x_{k}\right\vert }+C\gamma_{k}^{2}e^{-\frac{1}{4}\gamma_{k}}\left\vert
\left\vert \nabla\omega_{k}\right\vert \right\vert _{L^{\infty}\left(
\Omega_{1,k}\right)  }\\
&  \leq C\left\vert \left\vert \nabla\omega_{k}\right\vert \right\vert
_{L^{\infty}\left(  B_{R_{k}}\left(  x_{k}\right)  \right)  }\left(
\frac{r_{k}}{\left\vert y_{k}-x_{k}\right\vert }+\gamma_{k}^{2}e^{-\frac{1}%
{4}\gamma_{k}}\right)  \text{.}%
\end{align*}
We can prove the assertion (\ref{29}) for the case $\left\vert y_{k}%
-x_{k}\right\vert /r_{k}\rightarrow\infty$ as $k\rightarrow\infty$.

Next, we proceed to present the proofs for (\ref{44}) and (\ref{45}), which
offer a more refined estimates  of the assertion (\ref{29}). We will prove
(\ref{44}) by contradiction. Assume $y_{k}\in B_{R_{k}}\left(  x_{k}\right)  $
is defined by
\begin{equation}
\left\vert \nabla\left(  \omega_{k}-\phi_{k}\right)  \left(  y_{k}\right)
\right\vert =\left\vert \left\vert \nabla\left(  \omega_{k}-\phi_{k}\right)
\right\vert \right\vert _{L^{\infty}\left(  B_{R_{k}}\left(  x_{k}\right)
\right)  }\text{,}\label{31}%
\end{equation}
and satisfies%
\begin{equation*}
\alpha_{k}:=\left\vert \nabla\left(  \omega_{k}-\phi_{k}\right)  \left(
y_{k}\right)  \right\vert \geq\frac{L}{\gamma_{k}R_{k}}%
\end{equation*}
for some $L>0$. From (\ref{30}), we have
\begin{equation}
\left\vert \left\vert \nabla\omega_{k}\right\vert \right\vert _{L^{\infty
}\left(  B_{R_{k}}\left(  x_{k}\right)  \right)  }\leq\alpha_{k}+CR_{k}%
^{-1}\gamma_{k}^{-1}\leq\alpha_{k}\left(  1+\frac{C}{L}\right)  \text{.}%
\label{33}%
\end{equation}
Applying (\ref{33}) to (\ref{29}), we obtain that there exists $C_{L}%
\ $depending on $L$ such that
\[
\alpha_{k}\leq C_{L}\alpha_{k}\left(  \frac{r_{k}}{r_{k}+\left\vert
y_{k}-x_{k}\right\vert }+o\left(  \frac{1}{\gamma_{k}}\right)  \right)
\text{,}%
\]
which indicates that $\left(  \frac{\left\vert y_{k}-x_{k}\right\vert }{r_{k}%
}\right)  $ is bounded, and then
\begin{equation}
\frac{y_{k}-x_{k}}{r_{k}}\rightarrow y_{0}\in\mathbb{R}^{2}\text{ as
}k\rightarrow\infty\text{.}\label{37}%
\end{equation}
Setting
\begin{equation}
\tilde{\omega}_{k}\left(  x\right)  :=\frac{1}{r_{k}\alpha_{k}}\omega
_{k}\left(  x_{k}+r_{k}x\right)  \text{, \ \ \ }\tilde{\phi}_{k}\left(
x\right)  :=\frac{1}{r_{k}\alpha_{k}}\phi_{k}\left(  x_{k}+r_{k}x\right)
\text{.}\label{39}%
\end{equation}
According to (\ref{29}) and (\ref{33}), we get%
\begin{equation}
\left\vert \nabla\left(  \tilde{\omega}_{k}-\tilde{\phi}_{k}\right)  \left(
x\right)  \right\vert \leq\frac{C_{L}}{1+\left\vert x\right\vert }%
+o_{k}\left(  1\right)  \text{ for all }x\in\mathbb{R}^{2}\text{.}\label{34}%
\end{equation}
It is trivial to see that%
\begin{equation}
\tilde{\omega}_{k}\left(  0\right)  =0\text{, \ \ }\nabla\tilde{\omega}%
_{k}\left(  0\right)  =0\text{, \ \ }\left\vert \nabla\left(  \tilde{\omega
}_{k}-\tilde{\phi}_{k}\right)  \left(  \frac{y_{k}-x_{k}}{r_{k}}\right)
\right\vert =1\text{.}\label{35}%
\end{equation}
Employing (\ref{503}) and (\ref{43}), passing to a subsequence, we have
\begin{equation}
\nabla\tilde{\phi}_{k}\rightarrow\vec{A}\text{ in }C_{1oc}^{1}\left(
\mathbb{R}^{2}\right)  \text{ as }k\rightarrow\infty\text{,}\label{36}%
\end{equation}
where $\vec{A}$ $:=\left(  \underset{k\rightarrow\infty}{\lim}\frac{1}%
{\gamma_{k}R_{k}\alpha_{k}}\right)  \nabla\phi_{0}\left(  0\right)  $. Using
(\ref{21}), we obtain%
\begin{equation}
\left\vert -\Delta\tilde{\omega}_{k}\right\vert \leq\frac{C}{E_{k}}r_{k}%
^{2}\left(  1+2V_{k}^{2}\left(  x_{k}+r_{k}x\right)  \right)  e^{V_{k}%
^{2}\left(  x_{k}+r_{k}x\right)  }\left\vert \tilde{\omega}_{k}\right\vert
\text{.}\label{add681}%
\end{equation}
By (\ref{34}), (\ref{35}) and (\ref{36}), we can estimate%
\begin{align}
\left\vert \tilde{\omega}_{k}\left(  x\right)  \right\vert  &  =\left\vert
\int_{0}^{1}\left(  \tilde{\omega}_{k}\left(  rx\right)  -\tilde{\phi}%
_{k}\left(  rx\right)  \right)  ^{\prime}dr+\tilde{\phi}_{k}\left(  x\right)
-\tilde{\phi}_{k}\left(  0\right)  \right\vert \nonumber\\
&  \leq\int_{0}^{1}\left\vert \nabla\left(  \tilde{\omega}_{k}-\tilde{\phi
}_{k}\right)  \left(  rx\right)  \cdot x\right\vert dr+\left\vert \tilde{\phi
}_{k}\left(  x\right)  -\tilde{\phi}_{k}\left(  0\right)  \right\vert
\nonumber\\
&  \leq\int_{0}^{1}\frac{C_{L}\left\vert x\right\vert }{1+\left\vert
rx\right\vert }dr+\left\vert \nabla\tilde{\phi}_{k}\right\vert \left\vert
x\right\vert \nonumber\\
&  =C_{L}\log\left(  1+\left\vert x\right\vert \right)  +\left\vert \vec
{A}\right\vert \left\vert x\right\vert +o_{k}\left(  \left\vert x\right\vert
\right)  \text{,}\label{add682}%
\end{align}
which is uniformly bounded on any compact subset of $\mathbb{R}^{2}$.
Combining (\ref{add681}) and (\ref{add682}), we see that $\left(
-\Delta\tilde{\omega}_{k}\right)  $ is uniformly bounded in any compact subset
of $\mathbb{R}^{2}$ by the definition of $r_{k}$. Thus, using the standard
elliptic theory, there exists some $\tilde{\omega}_{0}\in C_{loc}^{1}\left(
\mathbb{R}^{2}\right)  $, such that
\begin{equation}
\tilde{\omega}_{k}\rightarrow\tilde{\omega}_{0}\text{ in }C_{loc}^{1}\left(
\mathbb{R}^{2}\right)  \text{ as }k\rightarrow\infty\text{.}\label{38}%
\end{equation}
Furthermore, combining (\ref{37}), (\ref{34}), (\ref{35}),(\ref{36}) with (\ref{38}), we
have%
\begin{equation}
\tilde{\omega}_{0}\left(  0\right)  =0\text{, }\nabla\tilde{\omega}_{0}\left(
0\right)  =0\text{, }\left\vert \nabla\tilde{\omega}_{0}\left(  y_{0}\right)
-\vec{A}\right\vert =1\text{, }\left\vert \nabla\tilde{\omega}_{0}\left(
x\right)  -\vec{A}\right\vert \leq\frac{C_{L}}{1+\left\vert x\right\vert
}\text{ in }\mathbb{R}^{2}\text{ .}\label{41}%
\end{equation}
Thus $\tilde{\omega}_{0}\not \equiv 0$. This together with the fact that
$\gamma_{k}\omega_{k}\left(  x_{k}+r_{k}x\right)  \rightarrow0$ in
$C_{loc}^{1}\left(  \mathbb{R}^{2}\right)  $ as $k\rightarrow\infty$ (from
(\ref{214})) imply that%
\[
\gamma_{k}\alpha_{k}r_{k}\rightarrow0\text{ as }k\rightarrow\infty\text{.}%
\]
From (\ref{21}), (\ref{87}) and (\ref{39}), we can write%
\begin{align}
-\Delta\tilde{\omega}_{k}\left(  x\right)   &  =\frac{r_{k}^{2}}{\alpha
_{k}r_{k}}\frac{4\pi}{E_{k}}\left[  V_{k}e^{V_{k}^{2}}\left(  e^{2V_{k}%
\omega_{k}+\omega_{k}^{2}}-1\right)  +\omega_{k}e^{\left(  V_{k}+\omega
_{k}\right)  ^{2}}\right.  \nonumber\\
&  \text{ \ \ }\left.  -\frac{p}{2\left(  4\pi\right)  ^{\frac{p}{2}}}\left(
1+o_{k}\left(  1\right)  \right)  \lambda_{k}V_{k}^{p-2}\omega_{k}\right]
\nonumber\\
&  =\frac{4}{\alpha_{k}r_{k}\gamma_{k}}\frac{V_{k}}{\gamma_{k}}e^{V_{k}%
^{2}-\gamma_{k}^{2}}\left(  e^{2V_{k}\omega_{k}+\omega_{k}^{2}}-1\right)
+4\gamma_{k}^{-2}\tilde{\omega}_{k}e^{V_{k}^{2}-\gamma_{k}^{2}+2V_{k}%
\omega_{k}+\omega_{k}^{2}}\label{40}\\
&  -\frac{p}{2\left(  4\pi\right)  ^{\frac{p}{2}}}\left(  1+o_{k}\left(
1\right)  \right)  4\gamma_{k}^{-2}\tilde{\omega}_{k}\lambda_{k}V_{k}%
^{p-2}e^{-\gamma_{k}^{2}}\text{.}\nonumber
\end{align}
It is clear that%
\begin{equation}
\gamma_{k}^{-1}V_{k}\left(  x_{k}+r_{k}x\right)  \rightarrow1\text{ in
}C_{loc}^{0}\left(  \mathbb{R}^{2}\right)  \text{ as }k\rightarrow
\infty\text{,}\label{215}%
\end{equation}%
\begin{equation}
V_{k}^{2}\left(  x_{k}+r_{k}x\right)  -\gamma_{k}^{2}\rightarrow2\phi_{\infty
}\left(  x\right)  \text{ in }C_{loc}^{0}\left(  \mathbb{R}^{2}\right)  \text{
as }k\rightarrow\infty\text{,}\label{216}%
\end{equation}
and
\begin{equation}
2V_{k}\left(  x_{k}+r_{k}x\right)  \omega_{k}\left(  x_{k}+r_{k}x\right)
+\omega_{k}\left(  x_{k}+r_{k}x\right)  ^{2}=2\gamma_{k}\alpha_{k}r_{k}%
\tilde{\omega}_{0}\left(  1+o_{k}\left(  1\right)  \right)  =o_{k}\left(
1\right)  \text{.}\label{217}%
\end{equation}
Plugging (\ref{215}), (\ref{216}) and (\ref{217}) back into (\ref{40}), we
have%
\[
-\Delta\tilde{\omega}_{k}=8e^{2\phi_{\infty}}\tilde{\omega}_{0}+o_{k}\left(
1\right)  \text{, in }\Omega_{k},
\]
Hence,%
\[
-\Delta\tilde{\omega}_{0}=8e^{2\phi_{\infty}}\tilde{\omega}_{0}\text{ in
}\mathbb{R}^{2}\text{.}%
\]
 From \cite[Lemma C.1.]{PLA}, we know that the only solution of the above
equation satisfying (\ref{41}) is $\tilde{\omega}_{0}\left(  x\right)
\equiv0$, which is a contradiction. Hence, we must have $\alpha_{k}=o\left(
\gamma_{k}^{-1}R_{k}^{-1}\right)  $, namely, (\ref{44}) holds true. Combining
(\ref{44}) with (\ref{30}), we have
\begin{align*}
\left\vert \left\vert \nabla\omega_{k}\right\vert \right\vert _{L^{\infty
}\left(  B_{R_{k}}\left(  x_{k}\right)  \right)  } &  \leq\left\vert
\left\vert \nabla\left(  \omega_{k}-\phi_{k}\right)  \right\vert \right\vert
_{L^{\infty}\left(  B_{R_{k}}\left(  x_{k}\right)  \right)  }+\left\vert
\left\vert \nabla\phi_{k}\right\vert \right\vert _{L^{\infty}\left(  B_{R_{k}%
}\left(  x_{k}\right)  \right)  }\nonumber\\
&  =O\left(  \gamma_{k}^{-1}R_{k}^{-1}\right)  \text{,}%
\end{align*}
which together with (\ref{29}) imply that%
\[
\left\vert \nabla\left(  \omega_{k}-\phi_{k}\right)  \left(  y\right)
\right\vert \leq C\gamma_{k}^{-1}R_{k}^{-1}\left(  \frac{r_{k}}{r_{k}%
+\left\vert y-x_{k}\right\vert }+o\left(  \gamma_{k}^{-1}\right)  \right)
\text{.}%
\]
Since $\omega_{k}-\phi_{k}=0$ on $\partial B_{x_{k}}\left(  R_{k}\right)  $,
we obtain
\[
\left\vert \left(  \omega_{k}-\phi_{k}\right)  \left(  y\right)  \right\vert
\leq C\gamma_{k}^{-1}R_{k}^{-1}r_{k}\log\frac{r_{k}+R_{k}}{r_{k}+\left\vert
y-x_{k}\right\vert }+o\left(  \gamma_{k}^{-2}\right)
\]
for all $y\in B_{R_{k}}\left(  x_{k}\right)  $. Therefore, we have
\begin{align*}
\left\vert \left\vert \omega_{k}-\phi_{k}\right\vert \right\vert _{L^{\infty
}\left(  B_{R_{k}}\left(  x_{k}\right)  \right)  } &  =O\left(  \gamma
_{k}^{-1}\frac{r_{k}}{R_{k}}\log\frac{r_{k}+R_{k}}{r_{k}+\left\vert
y-x_{k}\right\vert }\right)  +o\left(  \gamma_{k}^{-2}\right)  \\
&  =O\left(  \gamma_{k}^{-1}\frac{r_{k}}{R_{k}}\log\left(  1+\frac{R_{k}%
}{r_{k}}\right)  \right)  +o\left(  \gamma_{k}^{-2}\right)  \\
&  =o\left(  \gamma_{k}^{-1}\right)  \text{,}%
\end{align*}
where in the last equality\ we have used the fact (\ref{43}). Hence (\ref{45})
is proved. In summary, we finish the proofs of (\ref{30}), \eqref{44} and
\eqref{45}, and the proof of this lemma is finished.
\end{proof}

\subsection{ Estimates of the exponential term and proof of Theorem \ref{th1.1}}

In this subsection, we will finish estimating the Dirichlet energy of $\left(
v_{k}\right)  $. From (\ref{14}) and Subsection \ref{poly}, we only need to
compute the integral of the exponential type term $I_{E}=\frac{{4\pi}}{E_{k}%
}\int_{\Omega}v_{k}^{2}e^{v_{k}^{2}}dx$.

From Lemma \ref{90} and Lemma \ref{507}, we see that
\begin{equation}
v_{k}\geq\left(  1-\delta+o_{k}\left(  1\right)  \right)  \gamma_{k}\text{ in
}B_{r_{k,\delta}}\left(  x_{k}\right)  ,\label{508}%
\end{equation}
In view of (\ref{508}), we rewrite integral $I_{E}$ as
\[
\frac{{4\pi}}{E_{k}}\left(  \int_{B_{r_{k,\delta}}\left(  x_{k}\right)  }%
+\int_{\Omega\backslash B_{r_{k,\delta}}\left(  x_{k}\right)  }\right)
v_{k}^{2}e^{v_{k}^{2}}dx:=I_{k}+II_{k}.
\]

We first estimate the integral $II_{k}$, for this, we introduce the following
truncated function of $v_{k}$: let ${0<{\delta^{\prime}}<\delta<1}$ and
define
\[
{v}_{k}^{{{\delta^{\prime}}}}{=}\left\{
\begin{array}
[c]{cc}%
v_{k} & \text{in }\Omega\backslash B_{r_{k,\delta}}\left(  x_{k}\right)
\text{,}\\
{\min\left\{  {{v_{k},}\left(  {1-{\delta^{\prime}}}\right)  {\gamma_{k}}%
}\right\}  } & \text{in }B_{r_{k,\delta}}\left(  x_{k}\right)  \text{.}%
\end{array}
\right.
\]

\begin{lemma}
\label{le3.63}It holds that
\[
{\int_{\Omega}{\left\vert {\nabla  {v}_{k}^{{{\delta^{\prime}}}}         }\right\vert }^{2}}dx\leq
4\pi\left(  {1-{\delta^{\prime}+o}}_{k}\left(  1\right)  \right)  .
\]

\end{lemma}

\begin{proof}
According to the proof of Lemma \ref{15}, we have $-\Delta V_{k}\geq0$ in
$B_{r_{k,\delta}}\left(  x_{k}\right)  $, then $V_{k}$ is radially decreasing
in $B_{r_{k,\delta}}\left(  x_{k}\right)  $. Through (\ref{55}) and
\eqref{510}, we obtain $V_{k}=\left(  1-\delta+o_{k}\left(  1\right)  \right)
\gamma_{k}$ on $\partial B_{r_{k,\delta}}\left(  x_{k}\right)  $. Hence, from
Lemma \ref{507}, we get
\begin{equation}
{{v_{k}<}\left(  {1-{\delta^{\prime}}}\right)  {\gamma_{k}}}\text{ on
}\partial B_{r_{k,\delta}}\left(  x_{k}\right)  \text{.}\label{03}%
\end{equation}
By the definition of ${v}_{k}^{{{\delta^{\prime}}}}$, we have
\[
v_{k}={v}_{k}^{{{\delta^{\prime}}}}+\tilde{v}_{k}^{{{\delta^{\prime}}}%
}\text{,}%
\]
where $\tilde{v}_{k}^{{{\delta^{\prime}}}}:=\left(  v_{k}-\left(
1-\delta^{^{\prime}}\right)  \gamma_{k}\right)  ^{+}\chi_{B_{r_{k,\delta}%
}\left(  x_{k}\right)  }$ vanishes on $\partial B_{r_{k,\delta}}\left(
x_{k}\right)  $ by (\ref{03}).

Using the Green's formula, for any given $R>0$ one has%
\begin{align*}
\int_{B_{r_{k,\delta}}\left(  x_{k}\right)  }\left\vert \nabla\tilde{v}%
_{k}^{{{\delta^{\prime}}}}\right\vert ^{2}dx  &  =\int_{B_{r_{k,\delta}%
}\left(  x_{k}\right)  }\nabla\tilde{v}_{k}^{{{\delta^{\prime}}}}\cdot\nabla
v_{k}dx\\
&  =\int_{B_{r_{k,\delta}}\left(  x_{k}\right)  }-\Delta v_{k}\tilde{v}%
_{k}^{{{\delta^{\prime}}}}dx\\
&  =\int_{B_{r_{k,\delta}}\left(  x_{k}\right)  }\frac{4\pi}{E_{k}}\left(
v_{k}e^{v_{k}^{2}}-\frac{p}{2\left(  4\pi\right)  ^{\frac{p}{2}}}\lambda
_{k}v_{k}^{p-1}\right)  \tilde{v}%
_{k}^{{{\delta^{\prime}}}}dx\\
&  \text{(using (\ref{227}) and (\ref{508}))}\\
&  \geq\int_{B_{Rr_{k}}\left(  x_{k}\right)  }\frac{4\pi}{E_{k}}v_{k}%
e^{v_{k}^{2}}\left(  1+o_{k}\left(  1\right)  \right)  \left(  v_{k}-\left(
1-\delta^{^{\prime}}\right)  \gamma_{k}\right)  dx\\
&  \geq\delta^{^{\prime}}\left(  1+o_{k}\left(  1\right)  \right)  \int
_{B_{R}\left(  0\right)  }4e^{2\phi_{\infty}}dy\text{,}%
\end{align*}
letting $R\rightarrow\infty$, we obtain
\[
\left\Vert \tilde{v}_{k}^{{{\delta^{\prime}}}}\right\Vert _{H_{0}^{1}\left(
\Omega\right)  }^{2}\geq4\pi{\delta^{\prime}}\left(  {1+o}_{k}{\left(
1\right)  }\right)  .
\]
Using the fact that $\left\Vert {v}_{k}^{{{\delta^{\prime}}}}+\tilde{v}%
_{k}^{{{\delta^{\prime}}}}\right\Vert _{H_{0}^{1}\left(  \Omega\right)  }%
^{2}=4\pi$, and ${v}_{k}^{{{\delta^{\prime}}}}\ $ and $\tilde{v}_{k}%
^{{{\delta^{\prime}}}}$ are $H_{0}^{1}$-orthogonal, we have
\[
\left\Vert {\nabla v}_{k}^{{{\delta^{\prime}}}}\right\Vert _{{L^{2}}\left(
\Omega\right)  }^{2}\leq4\pi\left(  {1-{\delta^{\prime}+o}}_{k}\left(
1\right)  \right)  .
\]
Then we get the desired result.
\end{proof}

\begin{proposition}
\label{pro4.10} It holds
\begin{equation}
\int_{\Omega\backslash B_{r_{k,\delta}}{\left(  {{x_{k}}}\right)  }}%
{\frac{{{4\pi}}}{{{E_{k}}}}v_{k}^{2}{e^{v_{k}^{2}}}dx\leq\frac{C}{{\gamma
_{k}^{4}}},} \label{out}%
\end{equation}
as $k\rightarrow\infty$.
\end{proposition}

\begin{proof}
Using H\"{o}lder's inequality, it follows from Lemma \ref{le4.5} and the
definition of ${v}_{k}^{{{\delta^{\prime}}}}$ that
\begin{align*}
\int_{\Omega\backslash B_{r_{k,\delta}}\left(  x_{k}\right)  }\frac{4\pi
}{E_{k}}v_{k}^{2}e^{v_{k}^{2}}dx &  =\frac{4\pi}{E_{k}\gamma_{k}^{2}}%
\int_{\Omega\backslash B_{r_{k,\delta}}\left(  x_{k}\right)  }\gamma_{k}%
^{2}v_{k}^{2}e^{v_{k}^{2}}dx\\
&  =\frac{4\pi\left(  1+o_{k}\left(  1\right)  \right)  }{\left(  S_{\Omega
}^{\delta}-\left\vert \Omega\right\vert \right)  \gamma_{k}^{4}}\int
_{\Omega\backslash B_{r_{k,\delta}}\left(  x_{k}\right)  }\gamma_{k}^{2}%
v_{k}^{2}e^{v_{k}^{2}}dx\\
&  \leq\frac{4\pi\left(  1+o_{k}\left(  1\right)  \right)  }{\left(
S_{\Omega}^{\delta}-\left\vert \Omega\right\vert \right)  \gamma_{k}^{4}%
}\left(  \int_{\Omega\backslash B_{r_{k,\delta}}\left(  x_{k}\right)  }\left(
\gamma_{k}v_{k}\right)  ^{2q}dx\right)  ^{\frac{1}{q}}\left(  \int_{\Omega
}e^{q^{\prime}\left(  {v}_{k}^{{{\delta^{\prime}}}}\right)  ^{2}}dx\right)
^{\frac{1}{q^{\prime}}}\text{,}%
\end{align*}
where $\frac{1}{q}+\frac{1}{q^{\prime}}=1$. Combining Lemma \ref{le3.63}, the
Trudinger-Moser inequality (\ref{1}) and (\ref{limfun}) in the Appendix,
(\ref{out}) is proved provided $q^{\prime}$ sufficiently close to$\ 1$.
\end{proof}

Now, it remains to estimate the integral of the exponential term in
${B_{{r_{k,\delta}}}}\left(  {{x_{k}}}\right)  $. As a direct consequence of
Lemma \ref{15} and Lemma \ref{507}, for $x\in$ $B_{r_{k,\delta}}\left(
x_{k}\right)  $, we have%
\begin{align*}
\frac{4\pi}{E_{k}}v_{k}e^{v_{k}^{2}}  &  =\frac{4\pi}{E_{k}}\left(
V_{k}+O\left(  \frac{\left\vert \cdot-x_{k}\right\vert }{\gamma_{k}%
r_{k,\delta\text{ }}}\right)  \right)  \exp\left(  V_{k}+O\left(
\frac{\left\vert \cdot-x_{k}\right\vert }{\gamma_{k}r_{k,\delta\text{ }}%
}\right)  \right)  ^{2}\\
&  =\frac{4\pi}{E_{k}}\left(  V_{k}+O\left(  \frac{\left\vert \cdot
-x_{k}\right\vert }{\gamma_{k}r_{k,\delta\text{ }}}\right)  \right)
\exp\left(  V_{k}^{2}+O\left(  \frac{\left\vert \cdot-x_{k}\right\vert
}{r_{k,\delta\text{ }}}\right)  \right) \\
&  =\frac{4\pi}{E_{k}}V_{k}\exp\left(  V_{k}^{2}\right)  \left(  1+O\left(
\frac{\left\vert \cdot-x_{k}\right\vert }{r_{k,\delta\text{ }}}\right)
\right)  \text{.}%
\end{align*}
In view of (\ref{main}), we have%
\begin{align}
\frac{4\pi}{E_{k}}v_{k}e^{v_{k}^{2}}  &  =\frac{4\exp\left(  -2t_{k}\right)
}{r_{k}^{2}\gamma_{k}}\left[  1+\frac{2S_{k}+t_{k}^{2}-t_{k}}{\gamma_{k}^{2}%
}\right. \nonumber\\
&  \left.  +O\left(  \frac{(1+t_{k}^{4})\exp(t_{k}^{2}/\gamma_{k}^{2})}%
{\gamma_{k}^{4}}+\left(  1+\frac{t_{k}^{2}}{\gamma_{k}^{2}}\right)
\frac{\left\vert \cdot-x_{k}\right\vert }{r_{k,\delta\text{ }}}\right)
\right]  \label{210}%
\end{align}
in $B_{r_{k,\delta}}\left(  x_{k}\right)  $. Thanks to (\ref{T}), we can
estimate the first error term in (\ref{210}). Specifically, we can find
$\alpha>1$ such that%
\begin{align*}
\frac{4\pi}{E_{k}}v_{k}e^{v_{k}^{2}}  &  =\frac{4\exp\left(  -2t_{k}\right)
}{r_{k}^{2}\gamma_{k}}\left[  1+\frac{2S_{k}+t_{k}^{2}-t_{k}}{\gamma_{k}^{2}%
}+O\left(  1+\frac{t_{k}^{2}}{\gamma_{k}^{2}}\right)  \frac{\left\vert
\cdot-x_{k}\right\vert }{r_{k,\delta\text{ }}}\right] \\
&  +O(\frac{\exp(-\alpha t_{k})}{r_{k}^{2}\gamma_{k}^{5}})\text{.}%
\end{align*}
Using Lemma \ref{15} and Lemma \ref{507} again, then for $x\in$
$B_{r_{k,\delta}}\left(  x_{k}\right)  $, we can get
\begin{align}
\frac{4\pi}{E_{k}}v_{k}^{2}e^{v_{k}^{2}}  &  =\frac{4\exp\left(
-2t_{k}\right)  }{r_{k}^{2}\gamma_{k}}\left(  V_{k}+O\left(  \frac{\left\vert
\cdot-x_{k}\right\vert }{\gamma_{k}r_{k,\delta\text{ }}}\right)  \right)
\left[  1+\frac{2S_{k}+t_{k}^{2}-t_{k}}{\gamma_{k}^{2}}\right. \nonumber\\
&  \left.  +O\left(  1+\frac{t_{k}^{2}}{\gamma_{k}^{2}}\right)  \frac
{\left\vert \cdot-x_{k}\right\vert }{r_{k,\delta\text{ }}}\right]  +O\left(
\frac{\exp(-\alpha t_{k})}{r_{k}^{2}\gamma_{k}^{5}}\right)  \left(
V_{k}+O\left(  \frac{\left\vert \cdot-x_{k}\right\vert }{\gamma_{k}%
r_{k,\delta\text{ }}}\right)  \right) \nonumber\\
&  =\frac{4\exp\left(  -2t_{k}\right)  }{r_{k}^{2}\gamma_{k}}\left(
\gamma_{k}-\frac{t_{k}}{\gamma_{k}}+o\left(  \frac{t_{k}}{\gamma_{k}}\right)
+O\left(  \frac{\left\vert \cdot-x_{k}\right\vert }{\gamma_{k}r_{k,\delta
\text{ }}}\right)  \right) \nonumber\\
&  \cdot\left[  1+\frac{2S_{k}+t_{k}^{2}-t_{k}}{\gamma_{k}^{2}}+O\left(
1+\frac{t_{k}^{2}}{\gamma_{k}^{2}}\right)  \frac{\left\vert \cdot
-x_{k}\right\vert }{r_{k,\delta\text{ }}}\right]  +O\left(  \frac{\exp(-\alpha
t_{k})}{r_{k}^{2}\gamma_{k}^{4}}\right) \label{211}\\
&  =\frac{4\exp\left(  -2t_{k}\right)  }{r_{k}^{2}}\left(  1-\frac{t_{k}%
}{\gamma_{k}^{2}}+o\left(  \frac{t_{k}}{\gamma_{k}^{2}}\right)  \right)
\left[  1+\frac{2S_{k}+t_{k}^{2}-t_{k}}{\gamma_{k}^{2}}\right. \nonumber\\
&  \left.  +O\left(  1+\frac{t_{k}^{2}}{\gamma_{k}^{2}}\right)  \frac
{\left\vert \cdot-x_{k}\right\vert }{r_{k,\delta\text{ }}}\right]  +O\left(
\frac{\exp(-\alpha t_{k})}{r_{k}^{2}\gamma_{k}^{4}}\right) \nonumber\\
&  =\frac{4\exp\left(  -2t_{k}\right)  }{r_{k}^{2}}\left[  1+\frac
{2S_{k}+t_{k}^{2}-2t_{k}}{\gamma_{k}^{2}}+O\left(  1+\frac{t_{k}^{2}}%
{\gamma_{k}^{2}}\right)  \frac{\left\vert \cdot-x_{k}\right\vert }%
{r_{k,\delta\text{ }}}\right] \nonumber\\
&  +O\left(  \frac{\exp(-\alpha t_{k})}{r_{k}^{2}\gamma_{k}^{4}}\right)
\text{.}\nonumber
\end{align}

From the above analysis, we can obtain the integral estimates  on the small ball
$B_{r_{k,\delta}}\left(  x_{k}\right)  $.
\begin{proposition}
\label{pro3.1}It holds that
\[
\frac{{4\pi}}{{{E_{k}}}}\int_{B_{r_{k,\delta}}\left(  x_{k}\right)  }%
{v_{k}^{2}}\exp\left(  {v_{k}^{2}}\right)  dy=4\pi+O\left(  {\frac{1}%
{{\gamma_{k}^{4}}}}\right)  .
\]

\end{proposition}

\begin{proof}
It follows from (\ref{211}) that
\begin{align*}
&  \frac{{4\pi}}{{{E_{k}}}}\int_{B_{r_{k,\delta}}\left(  x_{k}\right)  }%
{v_{k}^{2}}\exp\left(  {v_{k}^{2}}\right)  dy\\
&  =\int_{B_{r_{k,\delta}}\left(  x_{k}\right)  }\frac{4\exp\left(
-2t_{k}\right)  }{r_{k}^{2}}\left(  1+\frac{2S_{k}+t_{k}^{2}-2t_{k}}%
{\gamma_{k}^{2}}\right)  dy+O\left(  \int_{B_{r_{k,\delta}}\left(
x_{k}\right)  }\frac{4\exp\left(  -2t_{k}\right)  }{r_{k}^{2}}\left(
1+\frac{t_{k}^{2}}{\gamma_{k}^{2}}\right)  \frac{\left\vert \cdot
-x_{k}\right\vert }{r_{k,\delta\text{ }}}dy\right) \\
&  +O\left(  \int_{B_{r_{k,\delta}}\left(  x_{k}\right)  }\frac{\exp(-\alpha
t_{k})}{r_{k}^{2}\gamma_{k}^{4}}dy\right) \\
&  =:L_{1}+{L_{2}+L_{3}}\text{.}%
\end{align*}
From (\ref{212}) and by a direct calculation, we have
\begin{align*}
{L_{2}}  &  {=O}\left(  \int_{{B_{{{r_{k,\delta}/r}}_{k}}}\left(  {0}\right)
}\left(  \frac{1}{1+\left\vert x\right\vert ^{2}}\right)  ^{2}\left(
1+\frac{\log^{2}\left(  1+\left\vert x\right\vert ^{2}\right)  }{\gamma
_{k}^{2}}\right)  \frac{\left\vert x\right\vert r_{k}}{r_{k,\delta\text{ }}%
}dx\right) \\
&  =O\left(  \frac{r_{k}}{r_{k,\delta\text{ }}}\int_{{B_{{{r_{k,\delta}/r}%
}_{k}}}\left(  {0}\right)  }\left(  \frac{1}{1+\left\vert x\right\vert ^{2}%
}\right)  ^{2}\left\vert x\right\vert dx\right)  =O\left(  \frac{r_{k}%
}{r_{k,\delta\text{ }}}\right)  \text{.}%
\end{align*}
Since $\alpha>1$, the calculation of $L_{3}$ follows in a similar manner,
\[
L{_{3}}=\frac{1}{{\gamma_{k}^{4}}}{\int_{{B_{{{r_{k,\delta}/r}}_{k}}}\left(
{0}\right)  }}{\left(  {\frac{1}{{1+}\left\vert x\right\vert {{^{2}}}}%
}\right)  ^{\alpha}}dx=O\left(  {\frac{1}{{\gamma_{k}^{4}}}}\right)  \text{.}%
\]
Hence
\begin{align}
&  \frac{{4\pi}}{{{E_{k}}}}\int_{B_{r_{k,\delta\text{ }}}\left(  x_{k}\right)
}{v_{k}^{2}}\exp\left(  {v_{k}^{2}}\right)  dy\nonumber\\
&  =4\int_{{B_{{{r_{k,\delta}/r}}_{k}}}\left(  {0}\right)  }\exp\left(
2\phi_{\infty}\right)  \left(  1+\frac{2S_{0}+\phi_{\infty}^{2}+2\phi_{\infty
}}{\gamma_{k}^{2}}\right)  dx+O\left(  {\frac{1}{{\gamma_{k}^{4}}}+}%
\frac{r_{k}}{r_{k,\delta\text{ }}}\right)  \text{.} \label{3.97}%
\end{align}
\ By \eqref{56} and a direct calculation, we get
\[
\int_{\mathbb{R}^{2}}\left(  -\Delta S_{0}\right)  dx=4\pi=-\int
_{\mathbb{R}^{2}}4\exp\left(  2\phi_{\infty}\right)  \phi_{\infty}dx\text{,}%
\]
combining this with \eqref{57}, we can obtain
\begin{equation}
\int_{\mathbb{R}^{2}}\exp\left(  2\phi_{\infty}\right)  \left(  2S_{0}%
+\phi_{\infty}^{2}+2\phi_{\infty}\right)  dx=0. \label{add60}%
\end{equation}
Now, using (\ref{212}) again and by a direct calculation, for any $0<\sigma
<1$, we have
\begin{equation}
\int_{%
\mathbb{R}
^{2}\backslash{B_{{{r_{k,\delta}/r}}_{k}}}\left(  {0}\right)  }\exp\left(
2\phi_{\infty}\right)  \left(  1+\phi_{\infty}^{2}+S_{0}\right)  dx=O\left(
\left(  \frac{r_{k}}{r_{k,\delta\text{ }}}\right)  ^{\sigma}\right)  ,
\label{add69}%
\end{equation}
Combining (\ref{3.97}), (\ref{add60}), (\ref{add69}) and (\ref{add690}), we
get%
\[
\frac{{4\pi}}{{{E_{k}}}}\int_{B_{r_{k,\delta}}\left(  x_{k}\right)  }%
{v_{k}^{2}}\exp\left(  {v_{k}^{2}}\right)  dy=4\pi+O\left(  {\frac{1}%
{{\gamma_{k}^{4}}}}\right)  \text{.}%
\]
Then the proof is completed.
\end{proof}

In summary, we have obtained the Lebesgue perturbation term $I_{P}$\ (see
Proposition \ref{223}) and exponential term $I_{E}$ (see Propositions
\ref{pro4.10} and \ref{pro3.1}), respectively. Now, we can obtain the
following Dirichlet energy expansion formula for $\left\Vert {\nabla}u{{_{k}}%
}\right\Vert _{{L^{2}}}^{2}$. \

\begin{proposition}
\label{pro3.2} Let $p\in\lbrack1,2]$, there exist positive constants $C_{1},$
$C_{2}$ such that as $\lambda_{k}\rightarrow\infty$
\[
\left\Vert {\nabla{u_{k}}}\right\Vert _{{L^{2}}}^{2}\leq1-\lambda_{k}%
\frac{{{C_{1}}}}{{\gamma}_{k}^{2+p}}+\frac{{{C_{2}}}}{{\gamma}_{k}^{4}%
}+o\left(  {\gamma_{k}^{-4}}\right)  \text{.}%
\]

\end{proposition}

\begin{proof}
Using \eqref{14}, we can get from Proposition \ref{223}, Proposition
\ref{pro4.10} and Proposition \ref{pro3.1} that
\begin{align*}
\left\vert \left\vert \nabla u_{k}\right\vert \right\vert _{L^{2}}^{2}  &
=\frac{1}{4\pi}\left\vert \left\vert \nabla v_{k}\right\vert \right\vert
_{L^{2}}^{2}\\
&  =\frac{1}{4\pi}\left[  \frac{4\pi}{E_{k}}\int_{\Omega}\left(  v_{k}%
^{2}e^{v_{k}^{2}}-\frac{p}{2\left(  4\pi\right)  ^{\frac{p}{2}}}\lambda
_{k}v_{k}^{p}\right)  dx\right] \\
&  =\frac{1}{4\pi}\left[  \frac{4\pi}{E_{k}}\left(  \int_{\Omega\backslash
B_{r_{k,\delta\text{ }}}\left(  x_{k}\right)  }v_{k}^{2}e^{v_{k}^{2}}%
dx+\int_{B_{r_{k,\delta\text{ }}}\left(  x_{k}\right)  }v_{k}^{2}e^{v_{k}^{2}%
}dx\right)  \right. \\
&  \text{ \ \ \ \ }\left.  -\frac{4\pi}{E_{k}}\int_{\Omega}\frac{p}{2\left(
4\pi\right)  ^{\frac{p}{2}}}\lambda_{k}v_{k}^{p}dx\right] \\
&  \leq1-\lambda_{k}\frac{C_{1}}{\gamma_{k}^{2+p}}+\frac{C_{2}}{\gamma_{k}%
^{4}}+o\left(  \gamma_{k}^{-4}\right)  \text{,}%
\end{align*}
this completes the proof.
\end{proof}

Now, we are in position to give the proof of Theorem \ref{th1.1}. \medskip

\begin{proof}
[Proof of Theorem \ref{th1.1}]Recall that we have assumed that for any
$\lambda>0$, $S_{\Omega}(\lambda,p)$ could be achieved by some $u_{\lambda}%
\ $satisfying $\Vert\nabla u_{\lambda}\Vert_{L^{2}}=1$. In view of Proposition
\ref{pro3.2}, there exits subsequence $\left(  u_{k}\right)  $ of $\left(
u_{\lambda}\right)  $ such that $\left\vert \left\vert \nabla u_{k}\right\vert
\right\vert _{L^{2}}^{2}$ is strictly smaller than $1$, provided $\lambda_{k}$
large enough, this arrives at a contradiction with $\Vert\nabla u_{\lambda
}\Vert_{L^{2}}=1$. Hence, we conclude that there exists sufficiently large
$\lambda>0$ such that $S_{\Omega}(\lambda,p)$ could not be achieved. Define
$\ $
\[
\lambda^{\ast}(p):=\sup\left\{  \lambda>0\ |\ S_{\Omega}\left(  \lambda
,p\right)  \text{ }\mathrm{could\ be\ achieved}\right\}  <+\infty.
\]

We claim that
\begin{equation}
S_{\Omega}\left(  \lambda^{\ast}(p),p\right)  =S_{\Omega}^{\delta
}.\label{add6101}%
\end{equation}
In fact, if $S_{\Omega}\left(  \lambda^{\ast}(p),p\right)  >S_{\Omega}%
^{\delta}$, then $S_{\Omega}\left(  \lambda^{\ast}(p),p\right)  $ is achieved
by some $u_{\lambda^{\ast}}$ through Lemma \ref{225}. For $\lambda
>\lambda^{\ast}(p)$ and sufficiently close to $\lambda^{\ast}(p)$, we have
\[
S_{\Omega}(\lambda,p)\geq\int_{\Omega}\left(  e^{{4\pi}u_{\lambda^{\ast}}^{2}%
}-\lambda|u_{\lambda^{\ast}}|^{p}\right)  dx>S_{\Omega}^{\delta}.
\]
By Lemma \ref{225}, we see that $S_{\Omega}(\lambda,p)$ is achieved for
$\lambda$ slightly bigger than $\lambda^{\ast}(p)$, which arrives at a
contradiction with the definition of $\lambda^{\ast}(p)$. Hence the claim is
proved. From (\ref{add6101}) and using Lemma \ref{225} again, we can
accomplish the proof of Theorem \ref{th1.1}.
\end{proof}

\medskip

\section{ Appendix}

In this section, we give the proof of Proposition \ref{pro4.9}. \medskip

\begin{proof}
[Proof of Proposition \ref{pro4.9}]From the equation \eqref{eq of vn}, we see
that
\begin{equation}%
\begin{cases}
-\Delta\left(  \gamma_{k}v_{k}\right)  =\frac{4\pi}{E_{k}}\left(  \gamma
_{k}v_{k}\right)  e^{v_{k}^{2}}-\frac{4\pi}{E_{k}}\frac{p}{2{\left(  {4\pi
}\right)  ^{p/2}}}\lambda_{k}\gamma_{k}v_{k}^{p-1}\text{,} & \text{ in }%
\Omega\text{,}\\
v_{k}=0\text{,} & \text{ on }\partial\Omega\text{.}%
\end{cases}
\label{242}%
\end{equation}
Using the similar argument as in \cite[Proposition $3.7$]{liruf} (see also
\cite{Str}), we know that
\[
\left\Vert \nabla\left(  \gamma_{k}v_{k}\right)  \right\Vert _{L^{\gamma
}\left(  \Omega\right)  }\leq C(\gamma)
\]
for any $\gamma\in(1,2)$. Thus, after passing to a subsequence, there exists
some $\omega\in W_{0}^{1,\gamma}(\Omega)$ such that
\begin{equation}
\gamma_{k}v_{k}\rightharpoonup\omega\text{ in }W_{0}^{1,\gamma}(\Omega
)\text{,}\label{239}%
\end{equation}
and by the compactness of the Sobolev embedding, we get
\begin{equation}
\int_{\Omega}\left(  \gamma_{k}v_{k}\right)  ^{q}dx\rightarrow\int_{\Omega
}\omega^{q} dx\text{.}\label{limfun}%
\end{equation}
for any $1\leq q<\infty$. \vskip0.1cm

We next claim that $\omega\not \equiv 0$. We distinguish two cases: $p=1$ and
$p\in(1,2]$.

In the case of $p=1$, we suppose that $\omega=0$, it is observed from Lemma
\ref{le4.6} that%
\begin{align*}
0 &  =\underset{k\rightarrow\infty}{\lim}\int_{\Omega}\nabla\left(  \gamma
_{k}v_{k}\right)  \nabla\eta dx\\
&  =\underset{k\rightarrow\infty}{\lim}\frac{4\pi}{E_{k}}\int_{\Omega}\left[
\left(  \gamma_{k}v_{k}\right)  e^{v_{k}^{2}}\eta-\frac{1}{2{\left(  {4\pi
}\right)  ^{1/2}}}\lambda_{k}\gamma_{k}\eta\right]  dx\\
&  =4\pi\eta\left(  x_{0}\right)  -\underset{k\rightarrow\infty}{\lim}%
\frac{{\left(  {4\pi}\right)  ^{1/2}}}{E_{k}}\frac{\lambda_{k}\gamma_{k}}%
{2}\int_{\Omega}\eta dx
\end{align*}
for any $\eta\in C_{0}^{\infty}\left(  \Omega\right)  $, which is impossible,
hence, $\omega\not \equiv 0$. Using (\ref{238}), we get%
\[
o_{k}\left(  1\right)  =\frac{\lambda_{k}}{\gamma_{k}}\int_{\Omega}\gamma
_{k}v_{k}dx=\frac{\lambda_{k}}{\gamma_{k}}\left(  \int_{\Omega}\omega
dx+o_{k}\left(  1\right)  \right)  ,\text{as }k\rightarrow\infty,
\]
which means $\frac{\lambda_{k}}{\gamma_{k}}=o_{k}\left(  1\right)  $ due to
$\omega\not \equiv 0$. This together with Lemma \ref{le4.5} imply
\begin{equation}
\frac{\lambda_{k}\gamma_{k}}{E_{k}}\rightarrow0,\text{ as }k\rightarrow
\infty.\label{add610}%
\end{equation}

Combining Lemma \ref{le4.6}, (\ref{239}) and (\ref{add610}), the function
$\omega$ satisfies the limit equation of (\ref{242}) as follows:%
\begin{equation}
\left\{
\begin{array}
[c]{cc}%
-\Delta\omega=4\pi\delta_{x_{0}}\text{,} & \text{in }\Omega\text{,}\\
\omega=0\text{,} & \text{on }\partial\Omega\text{.}%
\end{array}
\right.  \label{02}%
\end{equation}
Hence we have $\omega=G_{\Omega,x_{0}}$.

As for the case of $p\in(1,2]$, we first claim that
\begin{equation}
\frac{\lambda_{k}}{E_{k}}\int_{\Omega}\gamma_{k}v_{k}^{p-1}dx=o_{k}\left(
1\right)  \label{01}%
\end{equation}
as $k\rightarrow\infty$. Indeed, we have proved Proposition \ref{pro4.9} holds
for $p=1$, then Theorem \ref{th1.1} holds for $p=1$. Then there exists a
positive constant ${\lambda}^{\ast}\left(  1\right)  $ such that $S_{\Omega
}\left(  \lambda,1\right)  $ is not attained for any $\lambda>{\lambda}^{\ast
}\left(  1\right)  $. It follows from H\"{o}lder's inequality that%
\[
\left\vert \Omega\right\vert ^{\frac{1}{p}-1}\left(  \int_{\Omega}\left\vert
f\right\vert ^{p}dx\right)  ^{1-\frac{1}{p}}\int_{\Omega}\left\vert
f\right\vert dx\leq\int_{\Omega}\left\vert f\right\vert ^{p}dx\text{.}%
\]
Then, we can obtain%
\begin{align*}
S_{\Omega}^{\delta} &  <\int_{\Omega}\left(  e^{4\pi u_{k}^{2}}-\lambda
_{k}u_{k}^{p}\right)  dx\\
&  <\int_{\Omega}\left(  e^{4\pi u_{k}^{2}}-\lambda_{k}\left\vert
\Omega\right\vert ^{\frac{1}{p}-1}\left(  \int_{\Omega}u_{k}^{p}dx\right)
^{1-\frac{1}{p}}\left\vert u_{k}\right\vert \right)  dx\\
&  \leq S_\Omega\left(  \lambda_{k}\left\vert \Omega\right\vert ^{\frac{1}{p}%
-1}\left(  \int_{\Omega}u_{k}^{p}dx\right)  ^{1-\frac{1}{p}},1\right)
\text{.}%
\end{align*}
By Lemma \ref{225}, we have
\[
\lambda_{k}\left\vert \Omega\right\vert ^{\frac{1}{p}-1}\left(  \int_{\Omega
}\left\vert u_{k}\right\vert ^{p}dx\right)  ^{1-\frac{1}{p}}<{\lambda}^{\ast
}\left(  1\right)  \text{.}%
\]
Applying the H\"{o}lder's inequality once more, we have%
\[
\int_{\Omega}u_{k}^{p-1}dx\leq\left(  \int_{\Omega}u_{k}^{p}dx\right)
^{1-\frac{1}{p}}\left\vert \Omega\right\vert ^{\frac{1}{p}}\text{.}%
\]
From the above analysis, we obtain%
\[
\lambda_{k}\int_{\Omega}u_{k}^{p-1}dx<{\lambda}^{\ast}\left(  1\right)
\left\vert \Omega\right\vert \text{,}%
\]
which is equivalent to
\[
\lambda_{k}\int_{\Omega}v_{k}^{p-1}dx<{\lambda}^{\ast}\left(  1\right)
\left\vert \Omega\right\vert \left(  4\pi\right)  ^{\frac{p-1}{2}}\text{.}%
\]
Consequently, by Lemma \ref{le4.5}, we have%
\[
\frac{\lambda_{k}}{E_{k}}\int_{\Omega}\gamma_{k}v_{k}^{p-1}dx\leq
\frac{{\lambda}^{\ast}\left(  1\right)  \left\vert \Omega\right\vert \left(
4\pi\right)  ^{\frac{p-1}{2}}}{\gamma_{k}\left(  S_{\Omega}^{\delta}%
-|\Omega|\right)  }\left(  1+o_{k}\left(  1\right)  \right)  =o_{k}\left(
1\right)
\]
as $k\rightarrow\infty$. Thus, the claim (\ref{01}) is proved.

It follows from Lemma \ref{le4.6}, (\ref{01}) and (\ref{239}), the limit
equation of (\ref{242}) is (\ref{02}). Thus, $\omega=G_{\Omega,x_{0}}$ and
\[
\int_{\Omega}\left(  \gamma_{k}v_{k}\right)  ^{q}dx\rightarrow\int_{\Omega
}G_{\Omega,x_{0}}^{q}dx\text{.}%
\]
In summary, we accomplish the proof of Proposition \ref{pro4.9}.
\end{proof}

\end{document}